\documentclass[hidelinks]{article}
\usepackage
[
a4paper,
left=2.5cm,
right=2.5cm,
top=2.5cm,
bottom=2.5cm,
]
{geometry}

\usepackage{amsmath}
\usepackage{amsfonts}
\usepackage{amsthm}
\usepackage{amssymb}
\usepackage{enumitem} 
\usepackage{hyperref}
\usepackage{setspace}
\usepackage{bbm}
\usepackage[utf8]{inputenc}

\def\C{\mathbb{C}}
\def\R{\mathbb{R}}

\DeclareMathOperator*{\Bernoulli}{Bernoulli}

\DeclareMathOperator{\median}{median}
\DeclareMathOperator{\KL}{KL}
\DeclareMathOperator*{\argmin}{argmin}
\DeclareMathOperator*{\argmax}{argmax}
\DeclareMathOperator*{\dTV}{d_{TV}}

\theoremstyle{plain}
\newtheorem{theorem}{Theorem}[section]
\newtheorem{lemma}[theorem]{Lemma}
\newtheorem{corollary}[theorem]{Corollary}
\newtheorem{proposition}[theorem]{Proposition}

\theoremstyle{remark}
\newtheorem{remark}{Remark}

\hypersetup{
    colorlinks=true,
    linkcolor=blue,
    filecolor=blue,      
    urlcolor=cyan,
    citecolor=red
}

\title{Optimal estimation of the null distribution in large-scale inference}
\author{Subhodh Kotekal$^1$ and Chao Gao$^2$ \\ \textit{University of Chicago}}

\date{}
\begin{document}
    \maketitle
    \footnotetext[1]{Email: \texttt{skotekal@uchicago.edu}. The research of SK is supported in part by NSF
Grants ECCS-2216912 and DMS-1547396. }
	\footnotetext[2]{Email: \texttt{chaogao@uchicago.edu}. The research of CG is supported in part by NSF
Grants ECCS-2216912 and DMS-2310769, NSF Career Award DMS-1847590, and an Alfred Sloan fellowship. }
    \abstract{
    The advent of large-scale inference has spurred reexamination of conventional statistical thinking. In a series of highly original articles, Efron persuasively illustrated the danger for downstream inference in assuming the veracity of a posited null distribution. In a Gaussian model for \(n\) many \(z\)-scores with at most \(k < \frac{n}{2}\) nonnulls, Efron suggests estimating the parameters of an \emph{empirical null} \(N(\theta, \sigma^2)\) instead of assuming the \emph{theoretical null} \(N(0, 1)\). Looking to the robust statistics literature by viewing the nonnulls as outliers is unsatisfactory as the question of optimal rates is still open; even consistency is not known in the regime \(k \asymp n\) which is especially relevant to many large-scale inference applications. However, provably rate-optimal robust estimators have been developed in other models (e.g. Huber contamination) which appear quite close to Efron's proposal. Notably, the impossibility of consistency when \(k \asymp n\) in these other models may suggest the same major weakness afflicts Efron's popularly adopted recommendation. A sound evaluation thus requires a complete understanding of information-theoretic limits. We characterize the regime of \(k\) for which consistent estimation is possible, notably without imposing any assumptions at all on the nonnull effects. Unlike in other robust models, it is shown consistent estimation of the location parameter is possible if and only if \(\frac{n}{2} - k = \omega(\sqrt{n})\), and of the scale parameter in the entire regime \(k < \frac{n}{2}\). Furthermore, we establish sharp minimax rates and show estimators based on the empirical characteristic function are optimal by exploiting the Gaussian character of the data. \newline

    }
    \section{Introduction}

    Consider the model 
    \begin{equation}\label{model:additive}
        X_j = \theta + \gamma_j + \sigma Z_j
    \end{equation}
    where \(\theta \in \R\) is the location parameter of interest, \(\gamma_1,...,\gamma_n \in \R\) are the unknown effects, \(Z_1,...,Z_n \overset{iid}{\sim} N(0, 1)\) are the noise variables, and \(\sigma > 0\) is the scale parameter. We are concerned with estimation of \(\theta\) and \(\sigma^2\) given observations from (\ref{model:additive}). We denote the joint distribution of \(\{X_j\}_{j=1}^{n}\) by \(P_{\theta, \gamma, \sigma}\) when the data are given by (\ref{model:additive}); the expectation with respect to \(P_{\theta, \gamma, \sigma}\) is denoted by \(E_{\theta, \gamma, \sigma}\). Of course, the parameters of interest are not identifiable in (\ref{model:additive}) as written. To ensure identifiability, we assume a limited number of \(\gamma_j\) are nonzero, that is, \(\sum_{j=1}^{n} \mathbbm{1}_{\{\gamma_j \neq 0\}} \leq k < \frac{n}{2}\). 

    The model (\ref{model:additive}) is a natural choice in the very typical setting in which a fraction of the data points exhibit perturbed values (``signal'' or ``contamination'' depending on the context) against an unknown background level. Indeed, in investigating this typical statistical setting by considering models similar to (\ref{model:additive}), the literatures of robust statistics and large-scale inference have witnessed vigorous research activity.
    
    \subsection{Robust statistics}\label{section:robust}
    An immediate connection between model (\ref{model:additive}) and robust statistics can be seen by simply conceptualizing \(X_j\) as an inlier if \(\gamma_j = 0\) and as an outlier otherwise. Of course, the robust statistics literature has not studied model (\ref{model:additive}) exclusively. To describe one foundational backdrop (among others), consider a dataset \(X_1,...,X_n\) in which 
    \begin{equation*}
        X_j \sim 
        \begin{cases}
            N(\theta, \sigma^2) &\textit{if } j \in \mathcal{I}, \\
            \textit{contamination} &\textit{if } j \in \mathcal{O}. 
        \end{cases}
    \end{equation*}
    Here, \(\mathcal{I}\) denotes the unknown set of inliers from which the location parameter is to be estimated, and \(\mathcal{O}\) with \(|\mathcal{O}| \leq k < \frac{n}{2}\) denotes the unknown set of outliers. With this backdrop in place, researchers stipulate various models for the contamination process and investigate its effects on statistical tasks.
    
    A popular choice with a long history in statistics is Huber's contamination model, proposed by Huber in a highly influential article \cite{huber_robust_1964}. A variant is given by 
    \begin{equation}\label{model:Huber_deterministic}
        X_j \overset{ind}{\sim} 
        \begin{cases}
            N(\theta, \sigma^2) &\textit{if } j \in \mathcal{I}, \\
            \delta_{\eta_j} &\textit{if } j \in \mathcal{O}. 
        \end{cases}
    \end{equation}
    Here, the sets \(\mathcal{I}\) and \(\mathcal{O}\) are unknown but fixed in advance to the generation of the data. The contamination points \(\eta_j\) are unknown and arbitrary, but also fixed in advance to the data generation. From a minimax perspective, this variant is closely related to the usual Bayes formulation of Huber's contamination model \cite{huber_robust_1964}, that is, 
    \begin{equation}\label{model:Huber_mixture}
        X_1,...,X_n \overset{iid}{\sim} (1-\varepsilon) N(\theta, \sigma^2) + \varepsilon Q,
    \end{equation}
    where \(\varepsilon = \frac{k}{n}\) is the contamination level and \(Q\) is an unknown, arbitrary probability distribution on \(\R\). That \(Q\) can be any probability distribution on \(\R\) is due to the fact that the \(\eta_j\) can be any values in \(\R\). In a minimax context, the conclusions about in-probability estimation rates in the frequentist version (\ref{model:Huber_deterministic}) and the Bayes version (\ref{model:Huber_mixture}) are the same \cite{chen_robust_2018,diakonikolas_robust_2019}. Sample median is perhaps the simplest and most famous robust estimator, and it was shown to be optimal in Huber's model by Chen et al. in \cite{chen_robust_2018}. In their article, the minimax rate of estimation was derived and applies to both Bayes (\ref{model:Huber_mixture}) and frequentist (\ref{model:Huber_deterministic}) contexts. In particular, they showed that if \(\varepsilon < \frac{1}{5}\), then \(|\median(X_1,...,X_n) - \theta|^2 \lesssim \sigma^2\left(\frac{1}{n} + \varepsilon^2\right) \asymp \sigma^2\left(\frac{1}{n} + \frac{k^2}{n^2}\right)\) with high probability. A matching lower bound (up to universal constants) was also obtained. 

    Returning to (\ref{model:additive}), the situation seems only slightly different from that in Huber's contamination model. Rather than outlier observations being set to arbitrary unknown values, outliers from (\ref{model:additive}) exhibit \textit{mean shifts}. Mean shift contamination has been extensively studied in the robust statistics literature, mostly in a regression context. To elaborate, consider the regression model for data \(\{(x_j, Y_j)\}_{j=1}^{n}\),
    \begin{equation}\label{model:mean_shift_regression}
            Y_j = \langle x_j, \beta\rangle + \gamma_j + \sigma Z_j,
    \end{equation}
    where \(\beta, x_1,...,x_n \in \R^p\) and \(Z_1,...,Z_n \overset{iid}{\sim} N(0, 1)\). The responses are contaminated by the nonzero mean shifts \(\gamma_j \in \R\), and there are at most \(k\) nonzero of them. Evidently, the location model (\ref{model:additive}) is a special case of (\ref{model:mean_shift_regression}). To the best of our knowledge, the first appearances of this mean shift regression model were in \cite{sardy_robust_2001,gannaz_robust_2007,mccann_robust_2007}. Gannaz \cite{gannaz_robust_2007}, noting \(\gamma = (\gamma_1,...,\gamma_n) \in \R^n\) is a \(k\)-sparse vector, suggested 
    \begin{equation}\label{estimator:ell_1}
        (\hat{\beta}, \hat{\gamma}) = \argmin_{\substack{\beta \in \R^p \\ \gamma \in \R^n}} \frac{1}{2}||Y - X\beta - \gamma||^2 + \rho ||\gamma||_1,
    \end{equation}
    where \(Y = (Y_1,...,Y_n) \in \R^n\) are the responses, \(X \in \R^{n \times p}\) is the design matrix, and \(\rho\) is a penalty hyperparameter to be tuned. Though the motivation for the \(\ell_1\) penalty from sparsity appears distant from robustness, it turns out \cite{sardy_robust_2001, gannaz_robust_2007,donoho_high_2016,antoniadis_wavelet_2007} that the \(\hat{\beta}\) solving (\ref{estimator:ell_1}) is also a robust M-estimator, namely it solves \(\min_{\beta \in \R^p} \sum_{j=1}^{n} \mathcal{L}_{\text{Huber}}(Y_j - \langle x_j, \beta \rangle ; \rho)\) where \(\mathcal{L}_{\text{Huber}}\) is the Huber loss function 
    \begin{equation*}
        \mathcal{L}_{\text{Huber}}(t ; \rho) = 
        \begin{cases}
            \frac{t^2}{2} &\textit{if } |t| \leq \rho, \\
            \rho|t| - \frac{\rho^2}{2} &\textit{if } |t| > \rho. 
        \end{cases}
    \end{equation*}
    Owen and She \cite{she_outlier_2011}, motivated from \(\hat{\beta}\)'s sensitivity to leverage points, suggest modifying (\ref{estimator:ell_1}). Since the \(\ell_1\) penalty in (\ref{estimator:ell_1}) corresponds to soft-thresholding in \(\hat{\gamma}\), they suggest replacing soft-thresholding by other thresholding functions, even those which may correspond to nonconvex penalties. Their article is methodological in nature and no theoretical guarantees are offered. In more modern contexts, the regression parameter \(\beta\) is also sparse. Consequently, a number of articles \cite{nguyen_robust_2013,dalalyan_outlier-robust_2019,collier_multidimensional_2019} (see also \cite{foygel_corrupted_2014} and references therein for a compressed sensing perspective) investigate a natural extension of (\ref{estimator:ell_1}) by placing an additional \(\ell_1\) penalty on \(\beta\). 
    
    Returning to our mean shift location model (\ref{model:additive}), Collier and Dalalyan's article \cite{collier_multidimensional_2019} study the estimator 
    \begin{equation*}
        (\hat{\theta}, \hat{\gamma}) = \argmin_{\substack{\mu \in \R \\ \gamma \in \R^n}} ||X - \mu\mathbf{1}_n - \gamma||^2 + \rho ||\gamma||_1,
    \end{equation*}
    where \(X = (X_1,...,X_n) \in \R^n\). Assuming \(\frac{k}{n} \leq c\) for a sufficiently small universal constant \(0 < c < \frac{1}{2}\) and \(\sigma = 1\), they show that for the choice \(\rho \asymp \sqrt{\log n}\) the estimator achieves \(|\hat{\theta} - \theta|^2 \lesssim \frac{1}{n} + \frac{k^2}{n^2} \log n\) with high probability, which is actually slower than the rate \(\frac{1}{n} + \frac{k^2}{n^2}\) achievable by sample median. The logarithmic factor is a consequence of a typical analysis for \(\ell_1\)-penalized estimators. Collier and Dalalyan suggest a modification by introducing individual penalities, replacing \(\rho ||\gamma||_1\) with \(\sum_{j=1}^{n} \rho_j |\gamma_j|\). An optimal choice of the \(\rho_j\) depends on unknown quantities, and so to circumvent this issue they suggest an iterative scheme and propose an iterative soft-thresholding estimator. They obtain the upper bound \(\frac{1}{n} + \frac{k^2}{n^2}\), matching the performance of sample median. 

    These results may give the impression that the Huber rate \(\frac{1}{n} + \frac{k^2}{n^2}\), which is minimax optimal under (\ref{model:Huber_deterministic}) or (\ref{model:Huber_mixture}), should also be the sharp rate in (\ref{model:additive}). This impression is strengthened when one expresses the observations as 
    \begin{equation}\label{model:mean_shift}
        X_j \overset{ind}{\sim}
        \begin{cases}
            N(\theta, \sigma^2) &\textit{if } j \in \mathcal{I}, \\
            N(\eta_j, \sigma^2) &\textit{if } j \in \mathcal{O},
        \end{cases}
    \end{equation}
    where \(\mathcal{O} = \{1 \leq j \leq n: \gamma_j \neq 0\}\), \(\mathcal{I} = \mathcal{O}^c\), and \(\eta_j = \theta + \gamma_j\). Since the shifts \(\gamma_j\) in (\ref{model:Huber_deterministic}) are completely arbitrary as in (\ref{model:mean_shift}), it is tempting to intuit that estimation ought to be as hard as in Huber's model. However, this intuition neglects the Gaussian character of the data. Consider the corresponding Bayes formulation (i.e. analogous to how (\ref{model:Huber_mixture}) corresponds to (\ref{model:Huber_deterministic})), 
    \begin{equation}\label{model:mean_shift_mixture}
        X_1,...,X_n \overset{iid}{\sim} (1-\varepsilon) N(\theta, \sigma^2) + \varepsilon (Q * N(0, \sigma^2)),
    \end{equation}
    where \(\varepsilon = \frac{k}{n}\) and \(*\) denotes convolution. Minimax rate conclusions are the same in (\ref{model:mean_shift_mixture}) and (\ref{model:mean_shift}) (in the way those in (\ref{model:Huber_mixture}) and (\ref{model:Huber_deterministic}) are the same as discussed earlier). A comparison between (\ref{model:mean_shift_mixture}) and the Bayes formulation of Huber's contamination model (\ref{model:Huber_mixture}) reveals the additional structure available in the mean shift model. In (\ref{model:Huber_mixture}), the distribution with weight \(\varepsilon\) can be any distribution \(Q\), whereas in (\ref{model:mean_shift_mixture}) it must be of the form \(Q * N(0, \sigma^2)\). Clearly the set of distributions of this convolutional form is a smaller set than the set of all distributions. 
    
    Therefore, if a rate-optimal estimator is to achieve a faster rate than that in Huber's model, it must make specific use of the Gaussian character of the data. The penalized regression estimators discussed earlier appear not to exploit this structure. A fundamental question arises; for which values of \(k\) is consistent estimation of \(\theta\) possible, and for which values is it impossible? More specifically, is the regime of consistency different from that in Huber's model? 

    \subsection{Large-scale inference}\label{section:LSI}
    In many modern scientific fields, the researcher is faced with the problem of simultaneously considering a very large number of hypothesis tests. Consider \(z\)-scores, \(X_1,...,X_n\), each of which represents a test statistic corresponding to a hypothesis test. Traditionally, it is assumed a \(z\)-score follows exactly the standard normal distribution if its corresponding hypothesis is null, and a shifted normal if its corresponding hypothesis is nonnull,  
    \begin{equation*}
        X_j \sim 
        \begin{cases}
            N(0, 1) &\textit{if } j \in \mathcal{H}_0, \\
            N(\gamma_j, 1) &\textit{if } j \in \mathcal{H}_0^c.
        \end{cases}
    \end{equation*}
    Here, \(\mathcal{H}_0\) denotes the set of null hypotheses and the nonzero \(\gamma_j\) are the nonnull effects. Assume \(|\mathcal{H}_0^c| \leq k < \frac{n}{2}\), which is usual in the relevant scientific settings. 
    
    Efron, in a sequence of highly original and persuasive articles \cite{efron_large-scale_2004,efron_correlation_2007,efron_microarrays_2008}, demonstrated this traditional assumption that the null \(z\)-scores exactly follow the standard normal distribution is not at all mild. Seemingly slight misspecification can substantially affect downstream inferential conclusions. As Efron points out, a variety of practical issues can cause misspecification of the null hypothesis. One major issue is correlation between \(z\)-scores \cite{efron_correlation_2007,owen_variance_2005,qiu_correlation_2005,qiu_effects_2005}. To illustrate in a stylized benchmark example, consider a one factor correlation model \cite{fan_estimating_2012,fan_estimation_2017,leek_general_2008,leek_capturing_2007}, where the \(z\)-scores are given by \(X_j = \gamma_j + \sqrt{\rho} W + \sqrt{1-\rho} Z_j\) with \(W,Z_1,...,Z_n \overset{iid}{\sim} N(0, 1)\). Here, \(W\) is the common factor inducing equicorrelation of level \(\rho\). Observe that, marginally, the null distribution is correctly specified as \(X_j \sim N(0, 1)\) if \(j \in \mathcal{H}_0\). However, due to the correlation, the ensemble of null \(z\)-scores does not behave like a standard normal distribution. Conditionally on \(W\), 
    \begin{equation*}
        X_j | W \overset{ind}{\sim} 
        \begin{cases}
            N(\sqrt{\rho} W, 1-\rho) &\textit{if } j \in \mathcal{H}_0, \\
            N(\sqrt{\rho}W + \gamma_j, 1-\rho) &\textit{if } j \in \mathcal{H}_0^c. 
        \end{cases}
    \end{equation*}
    It is more appropriate to treat the ensemble of null \(z\)-scores as being drawn from a shifted normal distribution with variance \(1-\rho\), rather than the centered normal distribution with unit variance. As Efron \cite{efron_correlation_2007,efron_microarrays_2008} illustrates with more sophisticated arguments, correlation can substantially shift and alter the width of the ensemble null \(z\)-scores, and subsequently can have disastrous effects on downstream inference if a standard normal null distribution is assumed (see also \cite{qiu_correlation_2005,fan_estimating_2012}). Correlation's effects and methodology to handle it were further studied in \cite{schwartzman_comment_2010,azriel_empirical_2015,sun_solving_2018,stephens_false_2017,heller_optimal_2021,nie_large-scale_2023}, but a complete theoretical picture has not yet been attained. 
    
    To address issues stemming from a misspecified theoretical null, Efron argues an \emph{empirical null} capturing the null \(z\)-score ensemble should be estimated from the data. Retaining Gaussianity, Efron \cite{efron_large-scale_2004} suggests considering
    \begin{equation}\label{model:large_scale}
        X_j \overset{ind}{\sim} 
        \begin{cases}
            N(\theta, \sigma^2) &\textit{if } j \in \mathcal{H}_0, \\
            N(\eta_j, \sigma^2) &\textit{if } j \in \mathcal{H}_0^c,
        \end{cases}
    \end{equation}
    where \(\theta \in \R\) and \(\sigma > 0\) are parameters to be estimated. The models (\ref{model:additive}) and (\ref{model:large_scale}) are equivalent with $\eta_j=\theta+\gamma_j$ and $\mathcal{H}_0=\{1\leq j\leq n:\gamma_j=0\}$. As in our discussion of robust statistics, the model (\ref{model:large_scale}) is closely related to a mixture formulation (\ref{model:mean_shift_mixture}), which has the name \emph{(Gaussian) two-groups model} in the context of large-scale inference. Indeed, minimax conclusions in the two models are the same. Further, \(Q\) now has the interpretation of a signal distribution rather than contamination. The two-groups model has been widely adopted for large-scale inference and tremendous methodological development has occurred \cite{stephens_false_2017,jin_estimating_2007,cai_optimal_2010,sun_oracle_2007,cai_simultaneous_2009}. 

    It is worth emphasizing the setting \(k \asymp n\) is especially relevant to modern large-scale inference. Indeed, it is not at all an extreme case that a constant fraction (e.g. \(0.1\%, 1\%, 10\%, 25\%\) etc.) of the \(z\)-scores are nonnull in various applications. To elaborate on an example application, a viewpoint in genomics which has been gaining traction, especially in the context of genome-wide association studies (GWAS), is that it is not so appropriate to suppose that genetic predisposition to a complex trait or disease is driven by a highly sparse set of genetic variants (e.g. single nucleotide polymorphisms (SNPs)) with large effects \cite{pritchard_gwas_2001,goldstein_nejm_2009,cantor_review_2010,nature_review_2019}. In many published GWAS, only a small fraction of the phenotypic variation has so far been explained by the few genetic variants identified by the studies. One proposed explanation to this ``missing heritability problem'' \cite{missing_heritability_2009,young_heritability_2019} is that complex traits involve many, many genetic variants which individually may not have large effect sizes (and thus go undetected) but collectively have large influence \cite{pritchard_gwas_2001,goldstein_nejm_2009,yang_height_2010,schizophrenia_polygenic_2009,wood_height_2014}. Hence, many traits of interest are said to be polygenic, and an analysis of data from a GWAS which assumes a very sparse set of nonnull effects is unsuitable; in other words, the regime \(k \asymp n\) is the appropriate one. In fact, these developments in genomics have majorly pushed theoreticians to consider dense regimes when investigating various tasks and are directly cited as motivation \cite{bradic_ejs_2018,bradic_jasa_2018,bradic_aos_2022,zheng_nonsparse_2021,janson_eigenprism_2017,chen_glm_2023,hall_dense_2014,verzelen_SNR_2018}. For an explicit example, the authors of \cite{hall_dense_2014} directly point to \cite{goldstein_nejm_2009,hirschorn_nejm_2009,kraft_nejm_2009}.

    More specifically on the topic of estimation of the null distribution, the seminal work \cite{jin_estimating_2007} in the statistics literature was motivated precisely by the regime \(k \asymp n\); the authors write on pages 495-496,
    \begin{quotation}
        ``In many applications, it is more approriate to model the setting as \emph{nonsparse}, that is, the proportion of nonnull effects does not tend to zero when the number of hypotheses grows to infinity. In such settings, Efron's (2004) approach does not perform well, and it is not hard to show that the estimators of the null are generally inconsistent\ldots because the relevant information about the null is highly distorted by the nonnull effects in both of them." 
    \end{quotation}
    As the authors of \cite{jin_estimating_2007} also point out, other methods for null estimation proposed up to that point in time were also not appropriate for the same reason \cite{meinshausen_null_2006,swanepoel_null_1999}. Thus, the regime \(k \asymp n\) is of great interest to the large-scale inference community.
        
    The limitations of the robust statistics literature are readily apparent as the best known estimators for \(\theta\) are not even consistent in this regime! The minimax rate in Huber's contamination model \(\frac{1}{n} + \frac{k^2}{n^2} \asymp 1\) is of constant order when \(k \asymp n\). Even in the mean shift setting, the best estimators in the literature are only shown to achieve Huber's rate, to the best of our knowledge. Moreover, special care should be taken to avoid making assumptions like \(\frac{k}{n} \leq c\) for some sufficiently small constant \(c\). Theory relevant to the practice of large-scale inference ought to address the case where \(k\) is quite close to \(\frac{n}{2}\).

    In the large-scale inference literature, arguably the most successful estimators for the empirical null parameters in the contexts of (\ref{model:mean_shift}) and (\ref{model:mean_shift_mixture}) are based on Fourier analysis. To the best of our knowledge, this suite of methods was first introduced and developed in \cite{jin_estimating_2007,cai_optimal_2010,jin_proportion_2008}. Fourier-analytic approaches have also seen success in other statistical topics such as deconvolution and estimation in stochastic processes \cite{meister_deconvolution_2009,belomestny_estimation_2015,belomestny_spectral_2006}. Also, a Fourier-based estimator was recently shown to nearly achieve the optimal estimation rate of a robust location estimation problem in a heavy-tailed formulation of robustness \cite{bahmani_nearly_2021}. Most directly related to our setting is the work of Cai and Jin \cite{cai_optimal_2010}. To roughly describe the estimators of Cai and Jin \cite{cai_optimal_2010}, observe the characteristic function of (\ref{model:mean_shift_mixture}) is 
    \begin{equation*}
       \psi(\omega) := e^{-\frac{\omega^2 \sigma^2}{2}} \left((1-\varepsilon) e^{i\omega \theta} + \varepsilon \hat{Q}(\omega)\right),
    \end{equation*}
    where \(\hat{Q}\) denotes the characteristic function of \(Q\). To enable analogous discussion, in the frequentist context of (\ref{model:large_scale}) we define the corresponding \(\hat{Q}\) through the equation 
    \begin{equation*}
        \frac{1}{n}\sum_{j=1}^{n} E(e^{i\omega X_j}) = e^{-\frac{\omega^2\sigma^2}{2}}\left(\frac{n-|\mathcal{O}|}{n} e^{i\omega \theta} + \frac{|\mathcal{O}|}{n} \hat{Q}(\omega)\right),
    \end{equation*}
    which directly yields \(\hat{Q}(\omega) = \frac{1}{|\mathcal{O}|}\sum_{j \in \mathcal{O}} e^{i\omega \eta_j}\). The following discussion moves freely between the contexts of (\ref{model:mean_shift_mixture}) and (\ref{model:large_scale}) using the same notation \(\hat{Q}\) due to the correspondence of the two contexts noted earlier. Cai and Jin's key observation is that for large frequencies \(\omega\), 
    \begin{equation*}
        \psi(\omega) \approx e^{-\frac{\omega^2\sigma^2}{2}} (1-\varepsilon) e^{i\omega \theta},
    \end{equation*}
    if it is assumed \(\hat{Q}\) decays fast enough (e.g. if \(Q\) is smooth enough). Hence, the null parameters can be extracted from the large frequency components of the characteristic function. Of course, the population characteristic function is not accessible, so it is estimated via the empirical characteristic function. Taking an appropriate \(\omega\) to balance a typical bias-variance tradeoff, estimators \(\hat{\theta}_{\text{CJ}}\) and \(\hat{\sigma}_{\text{CJ}}^2\) are formed. 
    
    Their result can be stated informally as follows. Fix \(\alpha > 2\) and \(\beta \in [0, 1/2)\). Let \(cn^{-\beta} \leq \varepsilon \leq c\) where \(0 < c < 1\) is a universal constant. Note \(c\) can be larger than \(\frac{1}{2}\). Assuming some moment conditions on \(Q\) as well as the decay conditions \(\limsup_{\omega \to \infty} |\hat{Q}(\omega)| |\omega|^\alpha < \infty\) and \(\limsup_{\omega \to \infty} |\hat{Q}'(\omega)||\omega|^{\alpha+1} < \infty\), they establish 
    \begin{align}
        E\left(|\hat{\theta}_{\text{CJ}} - \theta|^2\right) &\lesssim \frac{\varepsilon^2}{\log^{\alpha+1}(n)}, \label{eqn:cai_jin_mean}\\
        E\left(|\hat{\sigma}_{\text{CJ}}^2 - \sigma^2|^2\right) &\lesssim \frac{\varepsilon^2}{\log^{\alpha+2}(n)}. \label{eqn:cai_jin_var} 
    \end{align}
    Matching minimax lower bounds are also obtained. From these bounds, Cai and Jin \cite{cai_optimal_2010} have shown that consistent estimation is possible even in the practically salient case \(\varepsilon \asymp 1\). Furthermore, their result is quite appealing as it sharply captures the effect of the interaction between the smoothness of the signal distribution \(Q\) and the signal sparsity on optimal estimation rates. 

    Though impressive, their results leave some important, fundamental questions unresolved. First, they require a number of regularity assumptions on the signal distribution \(Q\) which are undesirable from a robust statistics perspective; nothing can be immediately said about the models (\ref{model:mean_shift_mixture}) or (\ref{model:additive}). Furthermore, note that \(\hat{Q}(\omega) = \frac{1}{|\mathcal{O}|} \sum_{j \in \mathcal{O}} e^{i\omega \eta_j}\) and so it can be seen that the condition \(\limsup_{\omega \to \infty} |\hat{Q}'(\omega)| |\omega|^{\alpha+1} < \infty\) requires \(\limsup_{\omega \to \infty} \left| \frac{1}{|\mathcal{O}|} \sum_{j \in \mathcal{O}} i\eta_j e^{i\omega \eta_j} \right||\omega|^{\alpha+1} < \infty\). The outliers \(\eta_j\) are effectively required to be bounded. Second, their minimax theory requires the null parameters \(\theta\) and \(\sigma^2\) to be bounded by fixed constants. Third, one might hope that plugging in \(\alpha = 0\) may give the correct rates of the problem without smoothness. This is not true. The bound (\ref{eqn:cai_jin_mean}) applies for any \(cn^{-\beta} \leq \varepsilon \leq c\), that is, it notably holds for the choice \(\varepsilon = \frac{1}{2}\). Plugging in \(\alpha = 0\) gives the rate \(\log^{-1}(n)\), asserting that consistent estimation is possible in (\ref{model:mean_shift_mixture}) even when \(\varepsilon = \frac{1}{2}\), which is clearly false since $\theta$ is not identifiable when \(\varepsilon=\frac{1}{2}\) without any assumptions on $Q$. Hence, the problem of establishing the optimal estimation rate without any assumptions on \(Q\) is still completely open. 

    \subsection{Main contribution}\label{section:main_contribution}

    Throughout the remainder of the paper, we focus on the frequentist model (\ref{model:additive}) though we remind the reader all our results also hold for the Bayes formulation (\ref{model:mean_shift_mixture}). The following discussion on our minimax theory is in the context of (\ref{model:additive}) with no assumption on the unknown null mean \(\theta \in \R\) nor on the unknown scale \(\sigma > 0\); we only assume \(\gamma \in \R^n\) is sparse \(||\gamma||_0 \leq k\) where \(1 \leq k < \frac{n}{2}\) is assumed for identifiability reasons as discussed earlier. In this setting, our first contribution is to obtain the sharp minimax rate of estimating $\theta$ with respect to square loss,  
    \begin{equation}\label{rate:theta}
        \epsilon^*(k, n)^2 \asymp 
        \begin{cases}
            \frac{\sigma^2}{n} &\textit{if } 1\leq k \leq \sqrt{n}, \\
            \frac{\sigma^2 k^2}{n^2} \log^{-1}\left(\frac{ek^2}{n}\right) &\textit{if } \sqrt{n} < k \leq \frac{n}{4}, \\
            \sigma^2 \log^{-1}\left(\frac{e(n-2k)^2}{n}\right) &\textit{if } \frac{n}{4} < k \leq \frac{n}{2}-\sqrt{n}, \\
            \sigma^2\log\left(\frac{en}{(n-2k)^2}\right)  &\textit{if } \frac{n}{2}-\sqrt{n}<k<\frac{n}{2}.
        \end{cases}
    \end{equation}
    The minimax rate turns out to exhibit scaling not only in \(k\) but also in \(n-2k\). Scaling of the form \(n-k\) has been observed in a sparse signal detection problem under correlation \cite{kotekal_minimax_2023} as well as a one-sided version of (\ref{model:additive}) where the shape constraints \(\gamma_i \geq 0\) are imposed \cite{carpentier_estimating_2021}. It is quite natural to see \(n-k\) as it denotes the number of ``null'' or ``uncontaminated'' samples. The term \(n-2k\) denotes the number of inliers (nulls) minus the number of outliers (nonnulls). Its appearance is not surprising since the problem becomes more difficult as \(k\) approaches \(\frac{n}{2}\), at which point \(\theta\) is no longer identifiable. 
    
    The fundamental question regarding when consistent estimation is possible can now be answered. It is quickly concluded that consistent estimation of \(\theta\) is possible if and only if \(n - 2k = \omega(\sqrt{n})\). In contrast, consistent estimation is possible in Huber's model (\ref{model:Huber_deterministic}) if and only if \(k = o(n)\). Importantly, \(\theta\) can be consistently estimated in (\ref{model:additive}) in the practically relevant setting where \(k\) is a constant fraction of \(n\) mentioned in Section \ref{section:LSI}.
    
    When \(n - 2k \lesssim \sqrt{n}\), consistent estimation of \(\theta\) is not possible. The minimax rate in this regime can be directly established from recently available results  in \cite{kotekal_sparsity_2023}. A kernel mode estimator turns out to be rate-optimal. Section \ref{section:inconsistent} discusses this regime further.
    
    We note that the minimax rate in the regime \(1 \leq k \leq cn\) for a small universal constant \(c \in (0, 1)\) had been independently developed by Carpentier and Verzelen\footnote{(A. Carpentier and N. Verzelen, personal communication, 2023)}. Though their result nicely pins down the rate in this regime, the remaining regime is unaddressed. Our result delivers a full characterization of the minimax rate by covering the practically important case where \(k\) can be a large fraction of \(n\). Broadly, both our upper and lower bound approaches are conceptually different from theirs. 
    
    \begin{remark}[The effect of the data's Gaussian character]\label{remark:gaussian_character}
        By comparing to Huber's contamination model (\ref{model:Huber_deterministic}), our result enables us to understand the effect of the data's Gaussian character. To elaborate, in both (\ref{model:additive}) and Huber's contamination model (\ref{model:Huber_deterministic}) the minimax rate is parametric for \(k \leq \sqrt{n}\). For \(\sqrt{n} \leq k \leq \frac{n}{5}\) (recall \cite{chen_robust_2018} assumes the contamination level is below \(1/5\)), there is a \(\log^{-1}(k^2/n)\) improvement. It turns out one can easily prove (see Appendix \ref{section:huber}), using the methods of \cite{chen_robust_2018}, that the minimax rate in Huber's model (\ref{model:Huber_deterministic}) is \(\log\left(en/(n-2k)\right)\) for \(\frac{n}{5} \leq k < \frac{n}{2}\). Hence, when \(n-2k = o(n)\) and \(n-2k = \omega(\sqrt{n})\) the rate in Huber's model is actually growing whereas \(\epsilon^*(k,n)^2 \asymp \log^{-1}\left(e(n-2k)^2/n\right)\) is vanishing. All of these improvements are attributable solely to the Gaussian character of the data.

        More specifically, the characteristic function of the standard Gaussian noise \(Z \sim N(0, 1)\) is known exactly. Consequently, a Fourier-based estimator can be employed which aims to remove the noise by dividing the empirical characterstic function by the noise distribution's characteristic function. Exact knowledge of the (standardized) noise distribution is critical to this so-called deconvolution approach (see Section \ref{section:known_var_cf} for a more detailed discussion).
    \end{remark}
    
    \begin{remark}[Comparison to a one-sided version]
        It is also interesting to compare the location estimation rate \(\epsilon^*(k, n)^2\) to the minimax rate obtained for the one-sided version of (\ref{model:additive}) in \cite{carpentier_estimating_2021} where \(\gamma_i \geq 0\), 
        \begin{equation*}
            \epsilon_{\text{one-sided}}^*(k,n)^2 \asymp 
            \begin{cases}
                \frac{\sigma^2}{n} &\textit{if } 1\leq k \leq \sqrt{n}, \\
                \frac{\sigma^2 k^2}{n^2} \log^{-3}\left(\frac{k^2}{n}\right) &\textit{if } \sqrt{n} < k < \frac{n}{2}, \\
                \frac{\sigma^2 \log^4\left(\frac{n}{n-k}\right)}{\log^3 n} &\textit{if } \frac{n}{2} \leq k \leq n-1. 
            \end{cases}
        \end{equation*}
        Note \(k\) can be as large as \(n-1\) since the shape constraint on \(\gamma\) ensures identifiability of \(\theta\). The situation in the one-sided case is quite different from the two-sided case we consider; the rate \(\epsilon^*(k,n)\) is not at all an obvious extension of \(\epsilon_{\text{one-sided}}^*(k, n)\). The estimators proposed in \cite{carpentier_estimating_2021} heavily rely on the shape constraint and thus are not suitable for (\ref{model:additive}). 
    \end{remark}

    Our second contribution is to establish the minimax rate for estimating the variance $\sigma^2$,
    \begin{equation}\label{rate:variance}
        \epsilon_{\text{var}}^*(k, n)^2 \asymp 
        \begin{cases}
            \frac{1}{n} &\textit{if } 1\leq k \leq \sqrt{n}, \\
            \frac{k^2}{n^2} \log^{-2}\left(1 + \frac{k}{\sqrt{n}}\right) &\textit{if } \sqrt{n}<k<\frac{n}{2}. 
        \end{cases}
    \end{equation}
    This rate should be compared to the minimax rate for variance estimation in Huber's contamination model (\ref{model:Huber_mixture}). The article \cite{chen_robust_2018} showed, under the assumption \(\frac{k}{n} \leq \frac{1}{5}\), the rate \(\frac{1}{n} + \frac{k^2}{n^2}\). Namely, it is exactly the same rate for estimating the location parameter \(\theta\). In contrast, we see \(\epsilon_{\text{var}}^*(k, n)^2\) is not only faster than the rate in Huber's model, but it also enjoys a logarithmic improvement over \(\epsilon^*(k, n)^2\).

    It also turns out that the minimax rate for variance estimation is exactly the same as if \(\theta\) were already known. In particular, our result is an extension of the minimax rate for variance estimation in the sparse Gaussian sequence model \cite{comminges_adaptive_2021} (i.e. \(\theta = 0\) is known). However, the estimator proposed in \cite{comminges_adaptive_2021} requires \(\frac{k}{n} \leq c\) for a small universal constant \(0 < c < \frac{1}{8}\). As discussed in Section \ref{section:LSI}, such a condition is undesirable. The estimator we propose establishes that the minimax rate proved in \cite{comminges_adaptive_2021} extends beyond this regime. Variance estimation was also considered in \cite{cai_optimal_2010}, but their result suffers from requiring various assumptions on the outliers which have been discussed in Section \ref{section:LSI}. 
    
    As a consequence of our results on parameter estimation, the minimax rate for estimating the null distribution in total variation follows directly. The rate for location estimation turns out to dominate, and so most of the conclusions for estimating \(\theta\) carry over to minimax estimation in total variation.

    A few high-level remarks about our approach are in order. For both location and variance estimation, we take inspiration from earlier work \cite{comminges_adaptive_2021,cai_optimal_2010,jin_estimating_2007,jin_proportion_2008} and adopt a Fourier-based approach in both the upper and lower bounds. In Section \ref{section:known_var_cf}, the location estimator's development is discussed in detail for the case of known variance. In Section \ref{section:unknown_var}, the estimator is generalized to handle unknown variance; rate-optimal variance estimation is also discussed. For location estimation, the lower bound is proved via a two-point testing argument by constructing two marginal distributions in the Bayes model (\ref{model:mean_shift_mixture}) which have characteristic functions which agree on a large interval containing zero. The construction is notably different from existing work \cite{comminges_adaptive_2021,cai_optimal_2010} in order to sharply capture the second order \(n-2k\) scaling in the minimax rate; previous constructions miss this scaling entirely. As mentioned earlier, the lower bound for variance estimation was known \cite{comminges_adaptive_2021}, but a matching upper bound for all \(k < \frac{n}{2}\) was absent from the literature. 
    
    Our Fourier-based upper bound for location estimation, though similar in examining frequency space, is very different from the estimator proposed by Cai and Jin \cite{cai_optimal_2010}. A detailed discussion is given in Section \ref{section:known_var_cf}. Briefly, as noted in Section \ref{section:LSI}, Cai and Jin localize \(\theta\) by inspecting the empirical characteristic function at a suitably large frequency; the assumptions they place on the outliers are crucial to their estimator's success. A big advantage of the characteristic function is that it transforms the unbounded quantity \(X_j\) to the bounded quantity \(e^{i\omega X_j}\). Specifically, it is very well-suited to dealing with unbounded outliers. However, the assumptions (e.g. bounded moments of outliers, bounded null parameters) of \cite{cai_optimal_2010} limit this advantage of the characteristic function. In contrast to \cite{cai_optimal_2010}, we make no assumptions on the outliers and instead consider an estimator which fits to the nuisance outliers in frequency space. Interestingly, the interval of frequencies considered in the upper bound is the same interval of frequencies used in the lower bound, thus pointing to the optimality of the estimator.

    The preceding methodological discussion is in the context where \(k\) is assumed known, which is an idealized assumption typically unsuitable in practice. In Section \ref{section:adapt_k}, adaptive estimators are developed through Lepski's method \cite{lepski_problem_1990, lepski_asymptotically_1991,lepski_asymptotically_1992} which can achieve the sharp minimax rates even when \(k\) is unknown. Aside from adaptation, it is interesting to consider non-Gaussian noise \(Z_i \overset{iid}{\sim} F\), where \(F\) is some known, symmetric distribution such as the Laplace distribution. As described in Section \ref{section:deconvolution}, the characteristic function approach to null estimation can be directly modified to accommodate non-Gaussian noise by simply using the characteristic function of \(F\) in place of \(N(0, 1)\). It turns out the upper bound achieved by this modified estimator depends on the large frequency decay of \(F\)'s characteristic function, which is a phenomenon noted earlier in the deconvolution literature \cite{meister_deconvolution_2009,fan_optimal_1991,carroll_optimal_1988}.

    \subsection{Notation}
    For \(a, b \in \R\) the notation \(a \lesssim b\) denotes the existence of a universal constant \(c > 0\) such that \(a \leq cb\). The notation \(a \gtrsim b\) is used to denote \(b \lesssim a\). Additionally \(a \asymp b\) denotes \(a \lesssim b\) and \(a \gtrsim b\). The symbol \(:=\) is frequently used when defining a quantity or object. Furthermore, we frequently use \(a \vee b := \max(a, b)\) and \(a \wedge b := \min(a, b)\). We generically use the notation \(\mathbbm{1}_A\) to denote the indicator function for an event \(A\). For a vector \(v \in \R^n\) and a subset \(S \subset [n]\), we sometimes use the notation \(v_S \in \R^{n}\) to denote the vector with coordinate \(j\) equal to \(v_j\) if \(j \in S\) and zero otherwise. Additionally, \(||v||_0 := \sum_{j=1}^{n} \mathbbm{1}_{\{v_j \neq 0\}}\), \(||v||_1 := \sum_{j=1}^{n} |v_j|\), and \(||v||^2 := \sum_{j=1}^{n} v_j^2\). For two probability measures \(P\) and \(Q\) on a measurable space \((\mathcal{X}, \mathcal{A})\), the total variation distance is defined as \(\dTV(P, Q) := \sup_{A \in \mathcal{A}} |P(A) - Q(A)|\). If \(P\) is absolutely continuous with respect to \(Q\), then the \(\chi^2\)-divergence is defined as \(\chi^2(P||Q) := \int_{\mathcal{X}} \left(\frac{dP}{dQ} - 1\right)^2 \, dQ\). We will frequently use the same notation for two probability densities \(p\) and \(q\). For sequences \(\{a_k\}_{k=1}^{\infty}\) and \(\{b_k\}_{k=1}^{\infty}\), the notation \(a_k = o(b_k)\) denotes \(\lim_{k \to \infty} \frac{a_k}{b_k} = 0\) and the notation \(a_k = \omega(b_k)\) is used to denote \(b_k = o(a_k)\). For a point \(x \in \R\), the symbol \(\delta_x\) denotes the probability measure which places full probability mass at the point \(x\). The symbol \(*\) denotes convolution and the same symbol will be used in the context of the convolution of probability measures as well as functions. Complex numbers are used throughout and so \(i \in \C\) denotes the imaginary unit satisfying \(i^2 = -1\). For \(z \in \C\), \(\bar{z}\) denotes the complex conjugate of \(z\). In other places, \(i\) may be used as an indexing variable; the context will be clear to the reader even if not explicitly stated. Furthermore, we will \(| \cdot |\) to denote modulus. In some places, it is used with a finite set, in which case it denotes cardinality. 

    \section{\texorpdfstring{A Fourier-based estimator}{A Fourier-based estimator}}\label{section:known_var_cf}
    In this section, we propose a Fourier-based location estimator for the regime \(k \leq \frac{n}{2} - C\sqrt{n}\). 
    According to the minimax rate formula (\ref{rate:theta}), this is the regime where consistent estimation is possible. A rate-optimal estimator in the other regime $k>\frac{n}{2} - C\sqrt{n}$ will be presented in Section \ref{section:kme}.
    For ease of presentation and to aid understanding, we first discuss the case where the variance is known and taken to be unit. For notational ease in this section, \(P_{\theta, \gamma}\) and \(E_{\theta, \gamma}\) will be used in place of \(P_{\theta, \gamma, 1}\) and \(E_{\theta, \gamma, 1}\) respectively.

    Inspired by the success of Fourier-based methods \cite{jin_estimating_2007,cai_optimal_2010}, we look to the empirical characteristic function. As discussed in Section \ref{section:LSI}, the central idea of \cite{cai_optimal_2010,jin_estimating_2007} is no longer applicable since the characteristic function of \(Q\) need not decay when no assumptions are imposed. Before defining our estimator, we briefly outline Cai and Jin's estimator \cite{cai_optimal_2010}. Their article works in the context of the Bayes model (\ref{model:mean_shift_mixture}) but is also applicable in the frequentist model (\ref{model:mean_shift}) by taking \(Q = \frac{1}{|\mathcal{O}|}\sum_{j \in \mathcal{O}} \delta_{\eta_j}\). Let \(\psi(\omega) = e^{-\frac{\omega^2\sigma^2}{2}} \left((1-\varepsilon) e^{i\omega \theta} + \varepsilon \hat{Q}(\omega)\right)\) and \(\psi_0(\omega) = e^{i\omega \theta - \frac{\omega^2\sigma^2}{2}}\) denote the characteristic functions of the marginal and the null distributions. To extract \(\theta\) from the characteristic function, Cai and Jin \cite{cai_optimal_2010,jin_estimating_2007} first define the function \(\mu : \R \to \R\) with
    \begin{align*}
        \mu(\omega;\xi) = \left.\frac{\Im(\overline{\xi(t)} \xi'(t))}{|\xi(t)|^2}\right|_{t = \omega}
    \end{align*}
    for \(\omega \in \R\) and for a differentiable function \(\xi : \R \to \C\). Here, \(\Im(z)\) denotes the imaginary part of the complex number \(z \in \C\). Importantly, it is shown that \(\mu(\omega;\psi_0) = \theta\) for all \(\omega \neq 0\). Of course, \(\psi_0\) is not available to the statistician. Taking the empirical characteristic function \(\hat{\psi}(\omega) := \frac{1}{n} \sum_{j=1}^{n} e^{i\omega X_j}\), Cai and Jin \cite{cai_optimal_2010} define the estimator \(\hat{\theta} = \mu(\omega^* ; \hat{\psi})\) with a specific choice of \(\omega^*\). As explained in \cite{cai_optimal_2010,jin_estimating_2007}, the idea is that plugging in the empirical characteristic function should yield \(\hat{\theta} = \mu(\omega^*, \hat{\psi}) \approx \mu(\omega^*, \psi)\). Since \(\theta = \mu(\omega^*, \psi_0)\), Cai and Jin \cite{cai_optimal_2010} essentially need \(\psi\) to be close to \(\psi_0\). Since \(\psi(\omega) = (1-\varepsilon) \psi_0(\omega) + \varepsilon \hat{Q}(\omega) e^{-\frac{\omega^2\sigma^2}{2}}\), assumptions on \(Q\) cannot be avoided if \(\psi\) and \(\psi_0\) are to be close. This is a major drawback to writing \(\theta\) as a function of \(\psi_0\).
    
    We obtain \(\theta\) from \(\psi\) directly and thus do not need \(\psi\) and \(\psi_0\) to be close. Consequently, assumptions on the outliers are totally avoided. To motivate our construction, consider the following line of reasoning. Recall the notation \(\mathcal{O} = \{1 \leq j \leq n: \gamma_j \neq 0\}\) and \(\mathcal{I} = \mathcal{O}^c\). Consider the expectation of the empirical characteristic function at any frequency \(\omega \in \R\) is 
    \begin{equation*}
        E_{\theta, \gamma}\left(\frac{1}{n} \sum_{j=1}^{n} e^{i\omega X_j}\right) = e^{i\omega \theta -\frac{\omega^2}{2}}\left(\frac{n-|\mathcal{O}|}{n} + \frac{1}{n} \sum_{j \in \mathcal{O}} e^{i\omega \gamma_j} \right). 
    \end{equation*}
    Cai and Jin \cite{cai_optimal_2010} essentially assume \(\frac{1}{|\mathcal{O}|}\sum_{j \in \mathcal{O}} e^{i\omega \gamma_j}\) is small for a choice of large frequency. Instead, we will fit the function \(\omega \mapsto \frac{1}{|\mathcal{O}|}\sum_{j \in \mathcal{O}} e^{i\omega \gamma_j}\) directly. The following result states how \(\theta\) can be obtained from a population level optimization program. 

    \begin{lemma} 
        If \(k < \frac{n}{2}\), then 
        \begin{equation}\label{eqn:population_level}
            \theta = \argmin_{\mu \in \R} \sup_{\omega \in \R} \inf_{\substack{\zeta \in \C \\ |\zeta| \leq 1}} \left|E_{\theta, \gamma}\left(\frac{1}{n} \sum_{j=1}^{n} e^{i\omega (X_j - \mu) + \frac{\omega^2}{2}}\right) - \frac{n-k}{n} - \frac{k}{n} \zeta \right|.
        \end{equation}
    \end{lemma}
    \begin{proof}
        Clearly \(\theta\) is a minimizer at which the objective function is equal to zero. It remains to show \(\theta\) is a unique minimizer. Suppose \(\mu\) is also a minimizer. Then 
        \begin{align*}
            0 &= \sup_{\omega \in \R} \inf_{\substack{\zeta \in \C \\ |\zeta| \leq 1}} \left|E_{\theta, \gamma}\left(\frac{1}{n} \sum_{j=1}^{n} e^{i\omega (X_j - \mu) + \frac{\omega^2}{2}}\right) - \frac{n-k}{n} - \frac{k}{n} \zeta \right| \\
            &= \sup_{\omega \in \R} \inf_{\substack{\zeta \in \C \\ |\zeta| \leq 1}} \left|\frac{n-k}{n} \left(e^{i\omega(\theta - \mu)} - 1\right) + \frac{k}{n}\left(\frac{1}{k} \sum_{j \in \mathcal{O'}} e^{i\omega(\theta + \gamma_j - \mu)} -  \zeta \right)\right|
        \end{align*}
        where \(\mathcal{O'}\) is \(\mathcal{O}\) along with some arbitrary indices taken from \(\mathcal{I}\) (if needed) to ensure \(|\mathcal{O}'| = k\). Suppose \(\mu \neq \theta\). By reverse triangle inequality,  
        \begin{align*}
            0 &\geq \sup_{\omega \in \R} \inf_{\substack{\zeta \in \C \\ |\zeta| \leq 1}} \left|\frac{n-k}{n} \left(e^{i\omega(\theta - \mu)} - 1\right)\right| - \left|\frac{k}{n}\left(\frac{1}{k} \sum_{j \in \mathcal{O'}} e^{i\omega(\theta + \gamma_j - \mu)} -  \zeta \right)\right| \\
            &\geq \sup_{\omega \in \R} \frac{n-k}{n} \left|e^{i\omega(\theta - \mu)} - 1\right| - \frac{2k}{n} \\
            &\geq \frac{n-k}{n} \left|e^{i\frac{\pi}{\theta-\mu}(\theta-\mu)} - 1\right| - \frac{2k}{n} \\
            &= \frac{2(n-2k)}{n}
        \end{align*}
        which is a contradiction since \(k < \frac{n}{2}\). Therefore, \(\theta = \mu\), which is to say \(\theta\) is the unique minimizer.
    \end{proof}

    Since \(\theta\) can be written as the minimizer of a population level optimization program, it is natural to consider the estimator which minimizes an empirical version of the population level program. Following instinct, we would like to replace the expectation in (\ref{eqn:population_level}) with its empirical counterpart. However, it is nonsensical to do this without any modification since (\ref{eqn:population_level}) involves supremum over \(\omega \in \R\) and so the term \(\frac{1}{n} \sum_{j=1}^{n} e^{i\omega (X_j - \mu) + \frac{\omega^2}{2}}\) has variance which may blow up. To rectify this issue, we truncate and only consider \(|\omega| \leq \tau\). Consider 
    \begin{equation}\label{estimator:theta_known_var}
        \hat{\theta} = \argmin_{\mu \in \R} \sup_{|\omega| \leq \tau} \inf_{\substack{\zeta \in \C \\ |\zeta| \leq 1}} \left| \frac{1}{n} \sum_{j=1}^{n} e^{i\omega(X_j - \mu) + \frac{\omega^2}{2}} - \frac{n-k}{n} - \frac{k}{n} \zeta\right|. 
    \end{equation}
    \noindent Here, \(\tau\) is a tuning parameter to be selected. It will turn out to determine the bias-variance tradeoff in the risk of \(\hat{\theta}\). 
    \begin{theorem}\label{thm:cf_estimator}
        Fix \(\delta \in (0, 1)\). There exist \(C, C', c > 0\) depending only on \(\delta\) such that if \(n\) is sufficiently large depending only on \(\delta\), \(k \leq \frac{n}{2} - C\sqrt{n}\), and \(\tau = 1 \vee c\sqrt{\log\left(1 + \frac{k^2(n-2k)^2}{n^3}\right)}\), then
        \begin{equation*}
            \sup_{\substack{\theta \in \R \\ ||\gamma||_0 \leq k}} P_{\theta, \gamma}\left\{ |\hat{\theta} - \theta| > C'\left(\frac{1}{\sqrt{n}} + \frac{k}{n}\tau^{-1}\right)\right\} \leq \delta. 
        \end{equation*}
    \end{theorem}
    \noindent Note
    \begin{equation*}
        \frac{1}{\sqrt{n}} + \frac{k}{n \sqrt{1 \vee \log\left(1 + \frac{k^2(n-2k)^2}{n^{3}}\right)}} \asymp 
        \begin{cases}
            \frac{1}{\sqrt{n}} &\textit{if } k \leq \sqrt{n}, \\
            \frac{k}{n \sqrt{\log\left(1 + \frac{k^2}{n}\right)}} &\textit{if } \sqrt{n} < k < \frac{n}{4}, \\
            \frac{k}{n \sqrt{\log\left(1 + \frac{(n-2k)^2}{n}\right)}} &\textit{if } \frac{n}{4} \leq k < \frac{n}{2}-\sqrt{n},
        \end{cases}
    \end{equation*}
    from the inequality \(u/2 \leq \log(1+u) \leq u\) for \(u \in (0, 1)\). Thus, \(\hat{\theta}\) indeed achieves the rate (\ref{rate:theta}) when \(k \leq \frac{n}{2} - C\sqrt{n}\). 
    
    As with other Fourier-based methods, \(\hat{\theta}\) exploits the Gaussian character of the data to denoise. Namely, (\ref{estimator:theta_known_var}) involves dividing by the characteristic function of the standard Gaussian distribution, i.e. the term \(e^{\omega^2/2}\) appears. As discussed in Section \ref{section:deconvolution}, the estimator can be generalized to handle non-Gaussian noise. A couple of remarks are in order. 
    
    \begin{remark}
        The minimax rate for \(1 \leq k \leq cn\) for a small universal constant \(c \in (0, 1)\) had been independently developed by Carpentier and Verzelen\footnote{(A. Carpentier and N. Verzelen, personal communication, 2023)}, though their estimator is conceptually different from ours. Carpentier and Verzelen also employ a Fourier-based approach, but essentially aim to extract \(\theta\) from \(\psi_0\) (similar to \cite{cai_optimal_2010}) instead of explicitly accounting for the outliers as done in our estimator.
    \end{remark}
    
    \begin{remark}[Computation]
        Fitting to the outliers in the frequency domain is advantageous as it only involves optimization over the two-dimensional (one-dimensional complex) variable \(\zeta\). In contrast, typical ideas of fitting to outliers in the spatial domain (e.g. as in the mean shift contamination literature discussed in Section \ref{section:robust}) require optimization over a \(k\)-dimensional variable (or an \(n\)-dimensional variable with a penalty, as in \cite{collier_multidimensional_2019} mentioned in Section \ref{section:robust}). The computation of (\ref{estimator:theta_known_var}) is thus straightforward. First, consider the minimization over \(\mu\). One can take the interval having length of order \(\sqrt{\log n}\) centered at \(\median(X_1,...,X_n)\) and discretize it with grid points having distance of order the statistical rate \(\epsilon(k, n)\) given by (\ref{rate:theta}). Since the domain of optimization over \(\omega\) is a bounded interval, it can be discretized with a fine enough grid. A glance at the proof suggests taking grid points having distance of order \(\frac{1}{\sqrt{n \log n}}\) since we are guaranteed \(|\hat{\theta} - \theta|\lesssim \sqrt{\log n}\) by the discretization employed for \(\mu\). Minimization over \(\zeta\) is direct as it is a convex problem with a convex constraint. The estimator can then be computed by a simple two-dimensional grid search. 
    \end{remark}

    \section{Methodology: unknown variance}\label{section:unknown_var}

    In this section, the model (\ref{model:additive}) is considered in the setting where the variance \(\sigma^2\) is unknown. Sections \ref{section:sigmatilde} and \ref{section:sigmahat} address variance estimation and Section \ref{section:location_estimation} addresses generalizing the Fourier-based estimator of Section \ref{section:known_var_cf} to handle unknown variance. 

    \subsection{A pilot variance estimator}\label{section:sigmatilde}
    A pilot estimator which captures the order of \(\sigma^2\) will be needed, and we will directly use the correlation estimator of \cite{kotekal_sparsity_2023}. As the context of \cite{kotekal_sparsity_2023} appears different from our current setting, we discuss, for the reader's understanding, the development of the pilot estimator of \cite{kotekal_sparsity_2023}. One idea for a variance estimator is
    \begin{equation*}
        \min_{|S|=\lceil \frac{n}{2} \rceil} \frac{1}{|S|- 1} \sum_{j \in S} (X_j - \bar{X}_S)^2,
    \end{equation*}
    where \(\bar{X}_S = \frac{1}{|S|} \sum_{j \in S} X_j\). The intuition is that the \(X_j\) all have the same mean \(\theta\) for \(j \in \mathcal{I}\), and so the usual sample variance estimator could be used if \(\mathcal{I}\) were known. Since it is unknown, a search must be done. Subsets of size at least \(\frac{n}{2}\) are searched over because \(\mathcal{I}\) is of at least this size. If \(S \subset \mathcal{I}\), then \(\frac{1}{|S| - 1} \sum_{j \in S} (X_j - \bar{X}_S)^2\) will be a decent estimator of \(\sigma^2\). On the other hand, if \(S \cap \mathcal{O} \neq \emptyset\), then \(\frac{1}{|S| - 1} \sum_{j \in S} (X_j - \bar{X}_S)^2\) will overestimate \(\sigma^2\). Outliers can be conceptualized as contributing additional variability. Therefore, it is natural to minimize over all subsets \(S \subset \{1,...,n\}\) with \(|S| = \lceil \frac{n}{2}\rceil\). However, the variance estimator in our setting appears to require an exhaustive search and thus seems computationally intractable. The idea is rescued in \cite{kotekal_sparsity_2023} by random sampling to obtain a polynomial-time estimator. 

    Independently draw subsets \(E_1,...,E_m \subset \left\{1,...,n\right\}\) of size \(\ell\) uniformly at random and define 
    \begin{equation}\label{estimator:var_tilde}
        \tilde{\sigma}^2 := \min_{1 \leq r \leq m} \frac{1}{\ell-1} \sum_{j \in E_r} (X_j - \bar{X}_{E_r})^2,
    \end{equation}
    where \(\bar{X}_{E_r} = \frac{1}{\ell} \sum_{j \in E_r} X_j\). It is immediate that \(\tilde{\sigma}^2\) is an estimator computable in polynomial time in \(n\) whenever \(m\) scales polynomially in \(n\). The following result can be proved via the same argument as in \cite{kotekal_sparsity_2023}, so we omit the proof.

    \begin{proposition}[\cite{kotekal_sparsity_2023}]\label{prop:var_tilde}
        Fix \(\delta \in (0, 1)\). There exist constants \(C_1, C_2, L > 0\) depending only on \(\delta\) such that if \(m = \lceil n^{C_1} \rceil\) and \(2 < \ell = \lceil C_2 \log n \rceil\), then 
        \begin{equation*}
            \inf_{\substack{\theta \in \R \\ ||\gamma||_0 < n/2 \\ \sigma > 0}} P_{\theta, \gamma, \sigma} \left\{ L^{-1} \leq \frac{\tilde{\sigma}^2}{\sigma^2} \leq L\right\} \geq 1-\delta. 
        \end{equation*}
    \end{proposition}
    \noindent With the pilot estimator in hand, a Fourier-based variance estimator can now be constructed. 

    \subsection{A rate-optimal variance estimator}\label{section:sigmahat}
    Our Fourier-based variance estimator is quite different from Cai and Jin's Fourier-based variance estimator \cite{cai_optimal_2010}. Recalling they work in the context of (\ref{model:mean_shift_mixture}) (but their work is applicable to (\ref{model:mean_shift}) as well), let \(\psi(\omega) = e^{-\omega^2\sigma^2/2} \left((1-\varepsilon) e^{i\omega \theta} + \varepsilon \hat{Q}(\omega)\right)\) and \(\psi_0(\omega) = e^{i\omega \theta - \omega^2\sigma^2/2}\) denote the characteristic functions of the marginal and the null distributions. Cai and Jin \cite{cai_optimal_2010,jin_estimating_2007} define the function \(v : \R \to \R\) with
    \begin{align*}
        v(\omega;\xi) = -\left.\frac{\frac{d}{dt}|\xi(t)|}{t|\xi(t)|}\right|_{t = \omega}
    \end{align*}
    for \(\omega \in \R\) and for a differentiable function \(\xi : \R \to \C\). Importantly, \(v(\omega;\psi_0) = \sigma^2\) for all \(\omega \neq 0\). With the empirical characteristic function \(\hat{\psi}(\omega) := \frac{1}{n} \sum_{j=1}^{n} e^{i\omega X_j}\), Cai and Jin \cite{cai_optimal_2010} define the estimator \(\hat{\sigma}^2 = v(\omega^* ; \hat{\psi})\) with some specific choice \(\omega^*\). The same considerations noted for their location estimator in Section \ref{section:known_var_cf} apply to their variance estimator. As mentioned before, their estimator requires decay of \(\hat{Q}\) and so various assumptions were imposed on the outliers. 
    
    The estimator we propose avoids assumptions by exploiting the connection between the variance and the norm of the characteristic function. Set \(\hat{N}(\omega) := \left|\frac{1}{n}\sum_{j=1}^{n} e^{i\omega X_j}\right|\). Define 
    \begin{equation}\label{estimator:sigma_hat}
        \hat{\sigma}^2 = \inf_{a \leq \omega \leq b} -\frac{2\log \hat{N}(\omega)}{\omega^2}
    \end{equation}
    where \(a \leq b\) are tuning parameters to be set. 

    \begin{theorem}\label{thm:var}
        Fix \(\delta \in (0, 1)\). There exist constants \(C, c > 0\) depending only on \(\delta\) such that if \(n\) is sufficiently large depending only on \(\delta\), \(a = c \tilde{\sigma}^{-1}\left(\sqrt{1 \vee \log\left(\frac{ek^2}{n}\right)}\right)\), and \(b = 100a\), then 
        \begin{equation*}
            \sup_{\substack{\theta \in \R \\ ||\gamma||_0 \leq k \\ \sigma > 0}}P_{\theta, \gamma, \sigma} \left\{ \frac{|\hat{\sigma}^2 - \sigma^2|}{\sigma^2} > \frac{Ck}{n\log\left(1 + \frac{k}{\sqrt{n}}\right)}  \right\} \leq \delta, 
        \end{equation*}
        where \(\tilde{\sigma}\) is the pilot estimator from Section \ref{section:sigmatilde}.
    \end{theorem}

    \noindent Note 
    \begin{equation*}
        \frac{k}{n\log\left(1 + \frac{k}{\sqrt{n}}\right)} \asymp 
        \begin{cases}
            \frac{1}{\sqrt{n}} &\textit{if } 1 \leq k \leq \sqrt{n}, \\
            \frac{k}{n\log\left(1 + \frac{k}{\sqrt{n}}\right)} &\textit{if } \sqrt{n} < k < \frac{n}{2}
        \end{cases}
    \end{equation*}
    by the inequality \(u/2 \leq \log(1 + u) \leq u\) for \(u \in (0, 1)\). Thus, \(\hat{\sigma}^2\) indeed achieves the rate (\ref{rate:variance}). 

    The estimator given in (\ref{estimator:sigma_hat}) is quite close to the variance estimator proposed in the context of sparse Gaussian sequence model \cite{comminges_adaptive_2021}, which is a special case of (\ref{model:additive}) with $\theta=0$. Even though $\theta$ is unknown in the setting of (\ref{model:additive}), it is clear \(\hat{N}(\omega)\) is shift invariant; it retains the same value even if the original data \(\left\{X_j\right\}_{j=1}^{n}\) are replaced with the shifted data \(\{X_j+\mu\}_{j=1}^{n}\) for any \(\mu \in \R\). The difference between (\ref{estimator:sigma_hat}) and the estimator of \cite{comminges_adaptive_2021} lies entirely in taking infimum over \(a \leq \omega \leq b\); the authors of \cite{comminges_adaptive_2021} explicitly choose \(\omega\) and use the estimator
\begin{equation}
\hat{\sigma}^2_{\omega}=-\frac{2\log \hat{N}(\omega)}{\omega^2}. \label{eq:v-e-f-f}
\end{equation}
However, it was only proved in \cite{comminges_adaptive_2021} that this estimator is rate-optimal in the regime \(\frac{k}{n} \leq c\) for some sufficiently small constant \(0 < c < \frac{1}{8}\). In large-scale inference contexts, there is a practical need to address the case \(k \geq \frac{n}{8}\). 
    
    To explain how (\ref{estimator:sigma_hat}) improves on (\ref{eq:v-e-f-f}), suppose \(|\mathcal{O}| = k\) without loss of generality and consider the population counterpart of \(\hat{N}(\omega)\), 
    \begin{equation*}
        N(\omega) := \left|\frac{1}{n} \sum_{j=1}^{n} e^{i\omega(\theta+\gamma_j) - \frac{\sigma^2\omega^2}{2}}\right| = e^{-\frac{\sigma^2\omega^2}{2}} \left|1 - \frac{k}{n}\left(1 - \frac{1}{k}\sum_{j \in \mathcal{O}} e^{i \omega \gamma_j} \right)\right|. 
    \end{equation*}
 A rearrangement of the above equality leads to the following population counterpart of (\ref{eq:v-e-f-f}),
 \begin{equation}
 -\frac{2\log N(\omega)}{\omega^2}=\sigma^2-\frac{2}{\omega^2}\log\left|1 - \frac{k}{n}\left(1 - \frac{1}{k}\sum_{j \in \mathcal{O}} e^{i \omega \gamma_j} \right)\right|.\label{eq:v-e-f-f-p}
 \end{equation}
 The key quantity in (\ref{eq:v-e-f-f-p}) is the nonnegative bias term $-\frac{2}{\omega^2}\log\left|1 - \frac{k}{n}\left(1 - \frac{1}{k}\sum_{j \in \mathcal{O}} e^{i \omega \gamma_j} \right)\right|$, and the extra infimum over \(a \leq \omega \leq b\) in (\ref{estimator:sigma_hat}) is to achieve a smaller bias. To elaborate, we first note that
 \begin{equation}
 \left|1 - \frac{1}{k}\sum_{j \in \mathcal{O}} e^{i \omega \gamma_j}\right|^2=1+\left|\frac{1}{k}\sum_{j \in \mathcal{O}} e^{i \omega \gamma_j}\right|^2- \frac{2}{k} \sum_{j \in \mathcal{O}} \cos(\omega \gamma_j)\leq 2\left(1 - \frac{1}{k} \sum_{j \in \mathcal{O}} \cos(\omega \gamma_j)\right). \label{eq:used2}
 \end{equation}
 Thus, a lower bound for $\frac{1}{k} \sum_{j \in \mathcal{O}} \cos(\omega \gamma_j)$ leads to an upper bound for the positive bias term in (\ref{eq:v-e-f-f-p}). The paper \cite{comminges_adaptive_2021} uses the inequality $\frac{1}{k} \sum_{j \in \mathcal{O}} \cos(\omega \gamma_j)\geq -1$, so that
\begin{equation}
-\frac{2}{\omega^2}\log\left|1 - \frac{k}{n}\left(1 - \frac{1}{k}\sum_{j \in \mathcal{O}} e^{i \omega \gamma_j} \right)\right|\leq \frac{2}{\omega^2}\log\left(\frac{n}{n-2k}\right). \label{eq:pbb1}
\end{equation}
 It is clear $k$ needs to be bounded away from $\frac{n}{2}$ so that the bias bound (\ref{eq:pbb1}) does not blow up. In fact, the bound (\ref{eq:pbb1}) is sharp in the sense that for any $\omega$, there exist $\gamma_1,\ldots,\gamma_k$ such that (\ref{eq:pbb1}) is an equality. Therefore, the condition that $k$ is bounded away from $\frac{n}{2}$ imposed by \cite{comminges_adaptive_2021} cannot be improved for the estimator (\ref{eq:v-e-f-f}).
 
The additional infimum in our estimator (\ref{estimator:sigma_hat}) leads to better bias control. To elaborate, consider the following population counterpart of (\ref{estimator:sigma_hat}),
  \begin{equation}
 \inf_{a \leq \omega \leq b}-\frac{2\log N(\omega)}{\omega^2}=\sigma^2-\sup_{a \leq \omega \leq b}\frac{2}{\omega^2}\log\left|1 - \frac{k}{n}\left(1 - \frac{1}{k}\sum_{j \in \mathcal{O}} e^{i \omega \gamma_j} \right)\right|.\label{eq:v-e-a-f-p}
 \end{equation}
 Compared with (\ref{eq:v-e-f-f-p}), now the positive bias term becomes $-\sup_{a \leq \omega \leq b}\frac{2}{\omega^2}\log\left|1 - \frac{k}{n}\left(1 - \frac{1}{k}\sum_{j \in \mathcal{O}} e^{i \omega \gamma_j} \right)\right|$. In view of the inequality (\ref{eq:used2}), it suffices to find a lower bound for $\sup_{a \leq \omega \leq b}\frac{1}{k} \sum_{j \in \mathcal{O}} \cos(\omega \gamma_j)$, which is given by the following proposition.

    \begin{proposition}\label{prop:cosines}
        Suppose \(\gamma_1,...,\gamma_k \in \R\). Define the function \(f : \R \to \R\) with \(f(\omega) = \frac{1}{k} \sum_{j=1}^{k} \cos(\omega \gamma_j)\). If \(\alpha > 0\), then \(\sup_{\omega \in [\alpha, 100\alpha]} f(\omega) \geq -\frac{1}{5}\). 
    \end{proposition}
To apply Proposition \ref{prop:cosines}, we set $b=100a$, and then the bias term in (\ref{eq:v-e-a-f-p}) can be bounded by
$$-\sup_{a \leq \omega \leq 100a}\frac{2}{\omega^2}\log\left|1 - \frac{k}{n}\left(1 - \frac{1}{k}\sum_{j \in \mathcal{O}} e^{i \omega \gamma_j} \right)\right|\leq -\frac{2}{a^2}\log\left(1-\sqrt{\frac{12}{5}}\frac{k}{n}\right)\asymp \frac{1}{a^2},$$
whenever the condition $\frac{k}{n}< \frac{1}{2}$ holds.
In particular, the bias bound does not blow up to infinity when $k=\frac{n}{2}$, and an improvement over (\ref{eq:pbb1}) for \(k\) near \(\frac{n}{2}\) has been achieved. The estimator (\ref{estimator:sigma_hat}) is thus motivated from this bias improvement. The choice of the hyperparameter \(a\) involves a typical bias-variance tradeoff and Theorem \ref{thm:var} attests the success of (\ref{estimator:sigma_hat}).  

    \subsection{A variance-adaptive location estimator}\label{section:location_estimation}
    In this section, the Fourier-based estimator of Section \ref{section:known_var_cf} is generalized to handle unknown variance. Define 
    \begin{equation*}
        [\sigma_{-}^2, \sigma_{+}^2] := \left[ \hat{\sigma}^2\left(1 - \frac{Rk}{n \log\left(1 + \frac{k}{\sqrt{n}}\right)}\right), \hat{\sigma}^2\left(1 + \frac{Rk}{n \log\left(1 + \frac{k}{\sqrt{n}}\right)}\right)\right]
    \end{equation*}
    where \(R\) is some constant to be set. Then, the estimator is
    \begin{equation}\label{estimator:theta_unknown_var}
        (\hat{\theta}, \hat{v}) = \argmin_{\substack{\mu \in \R \\ \sigma_{-}^2 \leq v \leq \sigma_{+}^2}} \sup_{|\omega| \leq \tau} \inf_{\substack{\zeta \in \C \\ |\zeta| \leq 1}} \left| \frac{1}{n} \sum_{j=1}^{n} e^{i\omega(X_j - \mu) + \frac{v^2\omega^2}{2}} - \frac{n-k}{n} - \frac{k}{n} \zeta\right|.
    \end{equation}
    \noindent The following theoretical guarantee is available.
    \begin{theorem}\label{thm:theta_unknown_var}
        Fix \(\delta \in (0, 1)\). Let \(\hat{\sigma}^2\) be the variance estimator from Section \ref{section:sigmahat} at confidence level \(\frac{\delta}{2}\). There exist \(C, C', c, R > 0\) depending only on \(\delta\) such that the following holds. If \(n\) is sufficiently large depending only on \(\delta\), \(1 \leq k \leq \frac{n}{2} - C\sqrt{n}\), and \(\tau = \hat{\sigma}^{-1}\left(1 \vee c \sqrt{\log\left(1 + \frac{k^2(n-2k)^2}{n^3}\right)}\right)\), then 
        \begin{equation*}
            \sup_{\substack{\theta \in \R \\ ||\gamma||_0 \leq k \\ \sigma > 0}} P_{\theta, \gamma, \sigma}\left\{ \frac{|\hat{\theta} - \theta|}{\sigma} > C' \frac{k}{n} \log^{-\frac{1}{2}}\left(1 + \frac{k^2(n-2k)^2}{n^3}\right) \right\} \leq \delta
        \end{equation*} 
        where \(\hat{\theta}\) is given by (\ref{estimator:theta_unknown_var}). 
    \end{theorem}
    \noindent This estimator achieves the same rate as if the variance is known; adaptation to unknown \(\sigma^2\) is possible. The proof is broadly the same as the argument in the case where \(\sigma^2\) is known. The only difference is in the additional optimization over \(v\), which is not a serious complication. 
    
    \section{Lower bounds}
    In this section, minimax lower bounds are given. Location estimation is addressed in Section \ref{section:lower_bounds_location} and variance estimation is discussed in Section \ref{section:lower_bounds_variance}. 

    \subsection{Location estimation}\label{section:lower_bounds_location}
    Since adaptation to \(\sigma^2\) when estimating the location parameter is possible, we will prove the lower bounds as if \(\sigma^2 = 1\) were known. Recall that \(P_{\theta, \gamma}\) and \(E_{\theta, \gamma}\) are used in place of \(P_{\theta, \gamma, 1}\) and \(E_{\theta, \gamma, 1}\) respectively. The Fourier transform of an integrable function \(f\) is \(\hat{f}(t) = \int e^{-itx} f(x)\,dx\) and the Fourier transform of a finite Borel measure \(\pi\) is \(\hat{\pi}(t) = \int e^{-itx} \pi(dx)\). 
    
    A minimax lower bound can be established for \(1 \leq k \leq \frac{n}{2} - \sqrt{n}\) via a Fourier-based approach, showcasing an interesting parallel with the upper bound involving a Fourier-based estimator proposed in Section \ref{section:known_var_cf}. 
    \begin{theorem}\label{thm:lower_bound_cf}
        Suppose \(1 \leq k \leq \frac{n}{2} - \sqrt{n}\). There exist some universal constants \(C, c > 0\) such that 
        \begin{equation*}
            \inf_{\hat{\theta}} \sup_{\substack{\theta \in \R \\ ||\gamma||_0 \leq k}} P_{\theta, \gamma}\left\{ |\hat{\theta} - \theta| > C \cdot \frac{k}{n}\log^{-1/2}\left(1+\frac{k^2(n-2k)^2}{n^3}\right)\right\} \geq c.
        \end{equation*}
    \end{theorem}

    \noindent Note Theorem \ref{thm:lower_bound_cf} implies consistent estimation is impossible when \(n - 2k \lesssim \sqrt{n}\) as at least constant order error is unavoidable. Furthermore, consider
    \begin{equation*}
        \frac{k}{n \sqrt{\log\left(1 + \frac{k^2(n-2k)^2}{n^{3}}\right)}} \asymp 
        \begin{cases}
            \frac{1}{\sqrt{n}} &\textit{if } k \leq \sqrt{n}, \\
            \frac{k}{n \sqrt{\log\left(1 + \frac{k^2(n-2k)^2}{n^{3}}\right)}} &\textit{if } \sqrt{n} < k < \frac{n}{2}-\sqrt{n},
        \end{cases}
    \end{equation*}
    from the inequality \(u/2 \leq \log(1+u) \leq u\) for \(u \in (0, 1)\), and so the lower bound does indeed match the upper bound. The minimax lower bound for $k>\frac{n}{2}-\sqrt{n}$ will be presented in Section \ref{section:inconsistent_lbound}. 

    The proof of Theorem \ref{thm:lower_bound_cf} proceeds by first linking the model (\ref{model:additive}) to its mixture formulation (\ref{model:mean_shift_mixture}) via a standard concentration argument. The lower bound argument then involves considering the testing problem 
    \begin{align*}
        H_0 &: X_1,...,X_n \overset{iid}{\sim} (1-\varepsilon)N(-\theta, 1) + \varepsilon (Q_0 * N(0, 1)), \\
        H_1 &: X_1,...,X_n \overset{iid}{\sim} (1-\varepsilon)N(\theta, 1) + \varepsilon (Q_1 * N(0, 1)).
    \end{align*}
    where contamination distributions \(Q_0\) and \(Q_1\) as well as the location parameter \(\theta \in \R\) are all to be selected. Here, \(\varepsilon = \frac{k}{n}\). As usual in minimax lower bound arguments, the goal is to construct \(Q_0\) and \(Q_1\) such that \(\theta\) can be taken as large as possible while ensuring \(H_0\) and \(H_1\) cannot be distinguished. It turns out a Fourier-based approach \cite{cai_optimal_2010,carpentier_adaptive_2019,comminges_adaptive_2021} yields a rate-optimal construction. Letting \(f_0\) and \(f_1\) denote the marginal densities of \(H_0\) and \(H_1\) respectively, the parameters \(Q_0, Q_1,\) and \(\theta\) are selected such that the Fourier transforms of \(f_0\) and \(f_1\) agree on as wide an interval \([-\tau, \tau]\) as possible. As argued in the literature \cite{cai_optimal_2010,comminges_adaptive_2021,carpentier_adaptive_2019}, the \(\chi^2\)-divergence admits a bound in terms of the Fourier transforms of the marginal densities,
    \begin{equation*}
        \chi^2(f_1\,||\, f_0) \lesssim \sum_{k=0}^{\infty} \frac{1}{2^k k!} \int_{-\infty}^{\infty} \left|\frac{d^{k}}{dt^k}\left(\hat{f}_1(t) - \hat{f}_0(t)\right)e^{-it\theta}\right|^2 \, dt.
    \end{equation*}
   If \(Q_0, Q_1,\) and \(\theta\) are chosen such that the Fourier transforms match \(\hat{f}_1(t) = \hat{f}_0(t)\) on \([-\tau, \tau]\), then there is hope for the above integral to be small. It turns out the optimal choice of \(\tau\) in Theorem \ref{thm:cf_estimator} is precisely the correct choice in the lower bound as well, and essentially the optimal choice for the location parameter is \(|\theta| \asymp \frac{\varepsilon}{\tau}\). 

    Let us elaborate by making some more technical remarks. The condition \(\hat{f}_0(t) = \hat{f}_1(t)\) implies one should pick \(Q_0\) and \(Q_1\) such that 
    \begin{equation*}
        \hat{Q}_1(t) - \hat{Q}_0(t) = \frac{1-\varepsilon}{\varepsilon} \cdot 2i \sin(t\theta).
    \end{equation*}
    Since \(\varepsilon < \frac{1}{2}\) implies \(\frac{1-\varepsilon}{\varepsilon} > 1\), it is clear that this cannot hold for all \(t \in \R\) since \(|\hat{Q}_1(t) - \hat{Q}_0(t)| \leq 2\) because \(|\hat{Q}_1(t)| \vee |\hat{Q}_0(t)| \leq 1\) by virtue of \(Q_0\) and \(Q_1\) being probability measures. An idea is to seek \(Q_0\) and \(Q_1\) such that
    \begin{equation}\label{eqn:zero_approx}
        \hat{Q}_1(t) - \hat{Q}_0(t) = 
        \begin{cases}
            0 &\textit{if } t < -2\tau, \\
            \frac{1-\varepsilon}{\varepsilon} \cdot 2i\sin(\tau \theta) \cdot \frac{-2\tau-t}{\tau} &\textit{if } -2\tau \leq t \leq -\tau, \\
            \frac{1-\varepsilon}{\varepsilon} \cdot 2i \sin(t\theta) &\textit{if } -\tau < t < \tau, \\
            \frac{1-\varepsilon}{\varepsilon} \cdot 2i \sin(\tau\theta) \cdot \frac{2\tau-t}{\tau} &\textit{if } \tau \leq t \leq 2\tau, \\
            0 &\textit{if } t > 2\tau,
        \end{cases}
    \end{equation}
    where \(\tau\) is a tuning parameter to be set and will depend on the choice of \(\theta\). The linear functions for \(\tau \leq |t| \leq 2\tau\) simply taper to zero. When \(|\theta| \asymp \frac{\varepsilon}{\tau}\), the existence of such \(Q_0\) and \(Q_1\) is given by Proposition \ref{prop:prior_construction}. 
    
    Though appealing in its simplicity, it seems to us this avenue is only capable of delivering a sharp minimax lower bound for, say, the regime \(\varepsilon \leq \frac{1}{4}\). The behavior for \(\varepsilon\) close to \(\frac{1}{2}\) is missed; indeed, it is not clear from this idea why the term \(1-2\varepsilon \left(= \frac{n-2k}{n}\right)\) should even appear in the rate. Roughly speaking, choosing \(|\theta| \asymp \frac{\varepsilon}{\tau}\) yields the following bound on the \(\chi^2\)-divergence
    \begin{equation*}
        \chi^2(f_0\,||\, f_1) \lesssim \varepsilon^2 e^{-C\tau^2},
    \end{equation*}
    where \(C > 0\) is a universal constant. To have \(\chi^2(f_0\,||\,f_1) \lesssim \frac{1}{n}\) would force the choice \(\tau \asymp \sqrt{\log\left(en\varepsilon^2\right)}\asymp\sqrt{\log\left(\frac{ek^2}{n}\right)}\), thus missing the \(n-2k\) scaling. 
    
    The main issue is that the approximation of \(\frac{1-\varepsilon}{\varepsilon} \cdot 2i \sin(t\theta)\) by \(0\) for \(|t| > 2\tau\) is not good enough. Instead, consider seeking for \(|t| > 2\tau\), 
    \begin{equation}\label{eqn:smart_approx}
        \hat{Q}_1(t) - \hat{Q}_0(t) = 2\varepsilon \cdot 2i\sin(t\theta).
    \end{equation}
    Observe that now the approximation error is \(\left|\frac{1-\varepsilon}{\varepsilon} \left(2i \sin(t\theta)\right) - 2\varepsilon\left(2i\sin(t\theta)\right)\right| = (1-2\varepsilon)\frac{(1+\varepsilon)}{\varepsilon} \left|2i\sin(t\theta)\right|\), which scales with \(1-2\varepsilon\). Of course, a major question is whether it is even possible to find two probability measures \(Q_0\) and \(Q_1\) satisfying (\ref{eqn:smart_approx}). It turns out via a simple reparameterization that this question can be reduced to asking the same question in the context of the previous idea, namely whether there exists a pair \(\tilde{Q}_0,\tilde{Q}_1\) satisfying (\ref{eqn:zero_approx}). Consider the parameterization  
    \begin{align*}
        Q_0 &= 2\varepsilon \delta_{\theta} + (1-2\varepsilon) \tilde{Q}_0, \\
        Q_1 &= 2\varepsilon \delta_{-\theta} + (1-2\varepsilon) \tilde{Q}_1. 
    \end{align*}
    Then \(\hat{Q}_1(t) - \hat{Q}_0(t) = 2\varepsilon \cdot 2i \sin(t\theta) + (1-2\varepsilon)\left(\widehat{\tilde{Q}}_1(t) - \widehat{\tilde{Q}}_0(t)\right)\). Then both the condition (\ref{eqn:smart_approx}) for \(|t| \geq 2\tau\) and the condition \(\hat{Q}_1(t) - \hat{Q}_0(t) = \frac{1-\varepsilon}{\varepsilon} \cdot 2i \sin(t\theta)\) for \(|t| \leq \tau\) is satisfied if \(\tilde{Q}_0, \tilde{Q}_1\) satisfy (\ref{eqn:zero_approx}) with contamination proportion \(\frac{\varepsilon}{1+2\varepsilon}\) in place of \(\varepsilon\). The choice of parameterization admits a natural explanation. In the case \(\varepsilon = \frac{1}{2}\), we have \(Q_0 = \delta_\theta\) and \(Q_1 = \delta_{-\theta}\). This is natural as the location parameter is not identifiable and \(\chi^2(f_1\,||\, f_0) = 0\). This extreme case suggests the above parameterization, which interpolates between the construction (\ref{eqn:zero_approx}) (useful for \(\varepsilon < \frac{1}{4}\)) and the extremal case \(\varepsilon = \frac{1}{2}\). 
    
    The technical parts of the proof consist entirely in constructing suitable \(\tilde{Q}_0\) and \(\tilde{Q}_1\) and bounding the induced \(\chi^2\)-divergence using arguments inspired by \cite{comminges_adaptive_2021}. The correct choices of \(\tau\) and \(\theta\) emerge as consequences. Roughly speaking, choosing \(|\theta| \asymp \frac{\varepsilon}{\tau}\) yields the following bound on the \(\chi^2\)-divergence
    \begin{equation*}
        \chi^2(f_0\,||\, f_1) \lesssim (1-2\varepsilon)^2\varepsilon^2 e^{-C\tau^2},
    \end{equation*}
    where \(C > 0\) is a universal constant. Then the correct choice \(\tau \asymp \sqrt{\log\left(1+n\varepsilon^2(1-2\varepsilon)^2\right)}\asymp \sqrt{\log\left(1 + \frac{k^2(n-2k)^2}{n^3}\right)}\) is available to yield \(\chi^2(f_0\,||\,f_1) \lesssim \frac{1}{n}\). This rough sketch captures the essential spirit of the proof of Theorem \ref{thm:lower_bound_cf}.

    \subsection{Variance estimation}\label{section:lower_bounds_variance}
    For variance estimation, a minimax lower bound is given by Part (ii) of Proposition 7 in \cite{comminges_adaptive_2021}, matching the upper bound of Section \ref{section:sigmahat}. We give a statement of their lower bound in the present context for completeness. 

    \begin{theorem}[Part (ii) of Proposition 7 in \cite{comminges_adaptive_2021}]\label{thm:var_lbound}
        There exist some universal constants \(C, c > 0\) such that 
        \begin{equation*}
            \inf_{\hat{\theta}} \sup_{\substack{\theta \in \R \\ ||\gamma||_0 \leq k \\ \sigma > 0}} P_{\theta, \gamma, \sigma} \left\{ \frac{|\hat{\sigma}^2 - \sigma^2|}{\sigma^2} > \frac{Ck}{n\log\left(1 + \frac{k}{\sqrt{n}}\right)} \right\} \geq c.
        \end{equation*}
    \end{theorem}

    \section{Discussion}

    In this section, some remarks are made discussing finer points of interest. 

    \subsection{\texorpdfstring{Inconsistent, yet rate-optimal estimation for \(n-2k \lesssim \sqrt{n}\)}{Inconsistent, yet rate-optimal estimation for n-2k <= sqrt(n)}}\label{section:inconsistent}
    In the regime \(n - 2k \lesssim \sqrt{n}\), consistent estimation of \(\theta\) is impossible as implied by the lower bound of Section \ref{section:lower_bounds_location}. From the perspective of large-scale inference which motivates the two-groups model (\ref{model:additive}), it seems uninteresting to investigate this regime further. However, it turns out that the minimax rate in this regime still exhibits a subtle logarithmic dependence on $n-2k$, which can be derived from recent understanding of a kernel mode estimator \cite{kotekal_sparsity_2023}.

    \subsubsection{Upper bound}\label{section:kme}   
    Motivated by the recent results in \cite{kotekal_sparsity_2023}, we consider a kernel mode estimator with a box kernel. For a bandwidth \(h > 0\), define \(\hat{G}_h(t) := \frac{1}{2nh} \sum_{j=1}^{n} \mathbbm{1}_{\{|t - X_j| \leq h\}}\) and set 
    \begin{equation}\label{estimator:kme}
        \hat{\theta} := \argmax_{t \in \R} \hat{G}_h(t). 
    \end{equation}
    To simplify presentation, the variance will be assumed known and equal to one in this section; adaptation to \(\sigma^2\), as well as \(k\), has been discussed in \cite{kotekal_sparsity_2023}.
    
    A mode-type estimator is natural for the purpose of null distribution estimation. In the setting of two-groups model (\ref{model:additive}), the majority of the coordinates share the same mean value $\theta$. Efron \cite{efron_large-scale_2004,efron_microarrays_2008} essentially estimates a marginal density of the \(z\)-scores and then takes the center of the peak as an estimator of $\theta$. The kernel mode estimator in the form of (\ref{estimator:kme}) has a long history. It was first proposed by Parzen \cite{parzen_estimation_1962} in his original paper of kernel density estimation. Consistency, rates of convergence and asymptotic distribution were investigated by Chernoff \cite{chernoff_estimation_1964} and Eddy \cite{eddy_optimum_1980}.
        
    The recent paper \cite{kotekal_sparsity_2023} studies the statistical property of the kernel mode estimator in the context of robust estimation. Unlike the classical literature \cite{parzen_estimation_1962,chernoff_estimation_1964,eddy_optimum_1980}, it turns out the optimal choice of the bandwidth is a widening, rather than shrinking, one.

    \begin{theorem}[Theorem 3.1 \cite{kotekal_sparsity_2023}]\label{thm:kme}
        Suppose \(1 \leq k < \frac{n}{2}\). There exist universal constants \(C_1, C_2 > 0\) such that the following holds. For any \(\delta \in (0, 1)\), there exists \(L_\delta\) depending only on \(\delta\) such that if \(n\) is sufficiently large depending only on \(\delta\) and \(h = C_1 \sqrt{1 \vee \log\left(L_\delta n/(n-2k)^2\right)}\), then 
        \begin{equation*}
            \sup_{\substack{\theta \in \R \\ ||\gamma||_0 \leq k}} P_{\theta, \gamma}\left\{ |\hat{\theta} - \theta| > C_2 h\right\} \leq \delta
        \end{equation*}
        where \(\hat{\theta}\) is given by (\ref{estimator:kme}).
    \end{theorem}
    
    Applying Theorem \ref{thm:kme} to the regime $\frac{n}{2}-C\sqrt{n}<k<\frac{n}{2}$, the minimax rate (\ref{rate:theta}) is achieved in the inconsistency regime, thus complementing the result of Theorem \ref{thm:cf_estimator} for the Fourier-based estimator.
    
    \subsubsection{Lower bound}\label{section:inconsistent_lbound}

    In Section \ref{section:kme}, a kernel mode estimator is used to achieve the minimax rate in the inconsistency regime. The following matching minimax lower bound can be established, which is stated for \(\sigma^2 = 1\) without loss of generality.

    \begin{theorem}\label{thm:inconsistent_lbound}
        Suppose \(\frac{n}{2} - \sqrt{n} < k < \frac{n}{2}\). There exist some universal constants \(C, c > 0\) such that 
        \begin{equation*}
            \inf_{\hat{\theta}} \sup_{\substack{\theta \in \R \\ ||\gamma||_0 \leq k}} P_{\theta, \gamma}\left\{ |\hat{\theta} - \theta| > C \sqrt{\log\left(1 + \frac{n}{(n-2k)^2}\right)}\right\} \geq c. 
        \end{equation*}
    \end{theorem}
    
    Together with the result of Theorem \ref{thm:lower_bound_cf}, we have established the lower bound of the minimax rate (\ref{rate:theta}) for all $1\leq k<\frac{n}{2}$.

    The construction in the proof of Theorem \ref{thm:inconsistent_lbound} is, at its core, similar to that in \cite{kotekal_sparsity_2023}. Since the context of \cite{kotekal_sparsity_2023} is quite different, we still present a self-contained proof in the paper. Instead of constructing contamination distributions and selecting a location parameter such that the characteristic functions of the marginals match on a large interval containing zero, a more direct argument is available. Strictly speaking, the concentration argument used in Section \ref{section:lower_bounds_location} to link the frequentist model (\ref{model:additive}) to the Bayes model (\ref{model:mean_shift_mixture}) can no longer be used. The proof of Theorem \ref{thm:inconsistent_lbound} is thus directly constructed for the frequentist model (\ref{model:additive}).

    \subsection{Null estimation in total variation}\label{section:TV}
    In some situations, estimation of the null distribution with respect to some information-theoretic distance or divergence is of interest. Our results regarding parameter estimation turn out to directly yield minimax rates for estimating the null distribution in the total variation distance. The connection between the total variation of two univariate Gaussian distributions and their parameters is given by the following result of Devroye, Mehrabian, and Reddad \cite{devroye_total_2023}.
    \begin{lemma}[Theorem 1.3 in \cite{devroye_total_2023}]
        If \(\mu_1, \mu_2 \in \R\) and \(\sigma_1, \sigma_2 > 0\), then 
        \begin{equation*}
            \dTV\left(N(\mu_1, \sigma_1^2), N(\mu_2, \sigma_2^2) \right) \asymp 1 \wedge \left(\frac{|\sigma_1^2 - \sigma_2^2|}{\sigma_1^2 \vee \sigma_2^2} \vee \frac{|\mu_1 - \mu_2|}{\sigma_1 \vee \sigma_2}\right).  
        \end{equation*}
    \end{lemma}
    \noindent This result immediately yields the minimax rate for estimation of the null distribution in total variation, namely 
    \begin{equation*}
        \epsilon_{\text{TV}}^*(k, n) \asymp 1 \wedge \frac{k}{n\sqrt{\log\left(1 + \frac{k^2(n-2k)^2}{n^3}\right)}} \asymp 
        \begin{cases}
            \frac{k}{n}\log^{-\frac{1}{2}}\left(1 + \frac{k^2(n-2k)^2}{n^3}\right) &\textit{if } 1 \leq k \leq \frac{n}{2} - \sqrt{n}, \\
            1 &\textit{if } \frac{n}{2}-\sqrt{n} < k < \frac{n}{2}. 
        \end{cases}
    \end{equation*}
    Notably, the location estimation rate dominates and thus determines the total variation rate. A formal statement of the minimax rate follows, which we give without proof. 
    \begin{theorem}
    Suppose \(1 \leq k < \frac{n}{2}\), and consider the estimators \(\hat{\theta}\) and \(\hat{\sigma}^2\) from Theorems \ref{thm:var} and \ref{thm:theta_unknown_var}, respectively. For any $\delta\in (0,1)$, there exists some constant $C>0$ depending only on $\delta$ such that 
        \begin{equation*}
            \sup_{\substack{\theta \in \R \\ ||\gamma||_0 \leq k \\ \sigma > 0}} P_{\theta, \gamma, \sigma} \left\{ \dTV\left(N(\hat{\theta}, \hat{\sigma}^2), N(\theta, \sigma^2)\right) > C\left( \frac{k}{n\sqrt{\log\left(1 + \frac{k^2(n-2k)^2}{n^3}\right)}} \wedge 1 \right)\right\} \leq \delta. 
        \end{equation*}
         Furthermore, there exist some universal constants \(C', c > 0\) such that 
        \begin{equation*}
            \inf_{\hat{\theta}, \hat{\sigma}} \sup_{\substack{\theta \in \R \\ ||\gamma||_0 \leq k \\ \sigma > 0}} P_{\theta, \gamma, \sigma}\left\{  \dTV\left(N(\hat{\theta}, \hat{\sigma}^2), N(\theta, \sigma^2)\right) \geq C'\left( \frac{k}{n\sqrt{\log\left(1 + \frac{k^2(n-2k)^2}{n^3}\right)}} \wedge 1 \right) \right\} \geq c.
        \end{equation*}
    \end{theorem}
    \noindent Since for \(k \geq \frac{n}{2} - \sqrt{n}\) the lower bound is \(\epsilon_{\text{TV}}^*(k, n) \gtrsim 1\), it follows from the boundedness of total variation that \(N(\hat{\theta}, \hat{\sigma}^2)\) achieves the minimax rate even though no guarantees are available for \(\hat{\theta}\) when \(n-2k \lesssim \sqrt{n}\). We always have \(\dTV\left(N(\hat{\theta}, \hat{\sigma}^2), N(\theta, \sigma^2) \right) \lesssim 1\). This marks a significant difference between parameter estimation (in which minimax estimation for \(k \geq \frac{n}{2} - \sqrt{n}\) is nontrivial) and estimation in total variation.   
    
    \subsection{\texorpdfstring{Adaptation to \(k\)}{Adaptation to k}}\label{section:adapt_k}
    Thus far it has been assumed that the signal/contamination level \(k\) in (\ref{model:additive}) is known to the statistician. In practice, this is usually not the case and so adaptation to unknown \(k\) is an important problem. As frequently seen in the literature, a standard application of Lepski's method \cite{lepski_problem_1990,lepski_asymptotically_1991,lepski_asymptotically_1992} gives an adaptive estimator of the location parameter \(\theta\). Fix \(\delta, \eta \in (0, 1)\). We will consider adaptation to \(k\) for \(1 \leq k \leq \frac{n}{2} - C_\delta \sqrt{n}\) where \(C_\delta > 0\) is the constant depending on \(\delta\) from Theorem \ref{thm:theta_unknown_var}. Let \(\tilde{\sigma}^2\) denote the pilot variance estimator in Proposition \ref{prop:var_tilde} at confidence level \(\eta\). By Proposition \ref{prop:var_tilde} there exists \(L_\eta \geq 1\) such that \(\tilde{\sigma}^2 L_\eta^{-1} \leq \sigma^2 \leq \tilde{\sigma}^2 L_\eta\) with probability at least \(1-\eta\) uniformly over the parameter space. To apply Lepski's method, define
    \begin{align}\label{eqn:K_grid}
        \mathcal{K} &:= \left\{1,2,3,..., \left\lfloor \frac{n}{2} - C_\delta \sqrt{n}\right\rfloor\right\}. 
    \end{align}
    Let \(\hat{\theta}_k\) denote the estimator of Theorem \ref{thm:theta_unknown_var} at signal level \(k\) and confidence level \(\delta\). Note \(\hat{\theta}_k\) has associated with it a choice \(\tau = \tau_k\) depending on \(k\) given by Theorem \ref{thm:theta_unknown_var}. Define 
    \begin{equation}\label{rate:minimax}
        \epsilon(k, n) = \frac{k}{n\sqrt{\log\left(1 + \frac{k^2(n-2k)^2}{n^3}\right)}}.
    \end{equation}
    For \(k \in \mathcal{K}\), define the intervals 
    \begin{equation*}
        J_k = \left[\hat{\theta}_k - C_{\delta, \eta}'\sqrt{L_\eta} \tilde{\sigma}\epsilon(k, n), \hat{\theta}_k + C_{\delta, \eta}'\sqrt{L_\eta} \tilde{\sigma}\epsilon(k, n) \right],
    \end{equation*}
    where \(C_{\delta, \eta}'\) is chosen large enough depending on \(\delta\) and \(\eta\). The adaptive estimator \(\hat{\theta}\) is defined to be any element of the set 
    \begin{equation}\label{eqn:Lepski}
        \bigcap_{\substack{k \in \mathcal{K} \\ k \geq k'}} J_k,
    \end{equation}
    where \(k' \in \mathcal{K}\) is the smallest value such that the set is nonempty. If no such \(k'\) exists, set \(\hat{\theta} = 0\). Note that computing \(\hat{\theta}\) does not require any knowledge of the true signal/contamination level nor the true variance. 
    \begin{theorem}\label{thm:adapt_k}
        Suppose \(\delta, \eta \in (0, 1)\). There exists \(C_\delta, C_{\delta, \eta} > 0\) sufficiently large depending only on \(\delta\) and \(\delta, \eta\) respectively such that the following holds. If \(n\) is sufficiently large depending only on \(\delta, \eta\) and \(1 \leq k^* \leq \frac{n}{2} - C_\delta \sqrt{n}\), then 
        \begin{equation*}
            \sup_{\substack{\theta \in \R \\ ||\gamma||_0 \leq k^* \\ \sigma > 0}} P_{\theta, \gamma, \sigma}\left\{ \frac{|\hat{\theta} - \theta|}{\sigma} > C_{\delta, \eta}\epsilon(k^*, n)\right\} \leq \eta + \delta,
        \end{equation*}
        where \(\epsilon(k, n)\) is given by (\ref{rate:minimax}). 
    \end{theorem}
    \noindent Thus, the adaptive estimator \(\hat{\theta}\) achieves the minimax rate. It is not so surprising that adaptation to \(k\) is possible since it has been established in a one-sided contamination model \cite{carpentier_estimating_2021}. Of course, adaptation is trivial in Huber's contamination model since the sample median is already rate-optimal.

    The construction of an adaptive variance estimator is similar. Fix \(\delta, \eta \in (0, 1)\). Recall that $\tilde{\sigma}^2$ is a pilot variance estimator satisfying \(\tilde{\sigma}^2L_{\eta}^{-1} \leq \sigma^2 \leq \tilde{\sigma}^2 L_\eta\) with probability at least \(1-\eta\) uniformly over the parameter space. Let \(\hat{\sigma}^2_k\) denote the estimator from Section \ref{section:sigmahat} at confidence level \(\delta\) and signal level \(k\). Define 
    \begin{equation}\label{rate:minimax_var}
        \epsilon_{\text{var}}(k, n) = \frac{k}{n\log\left(1 + \frac{k}{\sqrt{n}}\right)}.    
    \end{equation}
    For \(1 \leq k < \frac{n}{2}\), define the intervals 
    \begin{equation*}
        J_k = \left[\hat{\sigma}^2_k - \tilde{\sigma}^2C_{\delta, \eta}'L_\eta\epsilon_{\text{var}}(k, n), \hat{\sigma}_k^2 + \tilde{\sigma}^2 C_{\delta, \eta}' L_\eta \epsilon_{\text{var}}(k, n) \right],
    \end{equation*}
    where \(C'_{\delta, \eta}\) is a sufficiently large constant depending only on \(\eta\) and \(\delta\). The adaptive estimator \(\hat{\sigma}^2\) is defined to be any element of the set 
    \begin{equation}\label{eqn:Lepski_var}
        \bigcap_{k' \leq k} J_k,
    \end{equation}
    where \(1 \leq k' < \frac{n}{2}\) is the smallest value such that the set is nonempty. If no such \(k'\) exists, set \(\hat{\sigma}^2 = 1\). Again, note that computing \(\hat{\sigma}^2\) does not require any knowledge of the true signal/contamination level. The following guarantee is available. 

    \begin{theorem}\label{thm:var_adapt_k}
        Suppose \(1 \leq k^* < \frac{n}{2}\). If \(\delta, \eta \in (0, 1)\), and \(n\) is sufficiently large depending only on \(\eta\) and \(\delta\), then there exists \(C_{\delta, \eta} > 0\) depending only on \(\eta\) and \(\delta\) such that 
        \begin{equation*}
            \sup_{\substack{\theta \in \R \\ ||\gamma||_0 \leq k^* \\ \sigma > 0}} P_{\theta, \gamma, \sigma}\left\{ \frac{|\hat{\sigma}^2 - \sigma^2|}{\sigma^2} > C_{\delta, \eta} \epsilon_{\text{var}}(k^*, n) \right\} \leq \eta + \delta. 
        \end{equation*}
    \end{theorem}

    With an adaptive variance estimator in hand which can achieve the minimax rate of location estimation, an adaptive estimator of the null distribution in total variation distance can be immediately constructed, yielding the following result which we state without proof. 
    \begin{theorem}
        For $\delta\in(0,1)$, there exists \(\tilde{C} > 0\) depending only on \(\delta\) such that the following holds. If \(n\) is sufficiently large depending only on \(\delta\), then 
        \begin{equation*}
            \sup_{\substack{\theta \in \R \\ ||\gamma||_0 \leq k \\ \sigma > 0}} P_{\theta, \gamma, \sigma}\left\{ \dTV\left(N(\hat{\theta}, \hat{\sigma}^2), N(\theta, \sigma^2)\right) > \tilde{C} \left(\frac{k}{n\sqrt{\log\left(1 + \frac{k^2(n-2k)^2}{n^3}\right)}} \wedge 1\right)\right\} \leq \delta
        \end{equation*}
        with appropriate hyperparameters for the adaptive estimators \(\hat{\theta}\) and \(\hat{\sigma}^2\). 
    \end{theorem}
    \noindent Thus, the minimax rate can be achieved without knowledge of \(k\).
    
    \begin{remark}[Computation]
        The grid (\ref{eqn:K_grid}) has \(O(n)\) cardinality, which means computing all the \(\hat{\theta}_k\)'s requires solving \(O(n)\)-many instances of the optimization problem (\ref{estimator:theta_unknown_var}). This computational burden can be mitigated by replacing \(\mathcal{K}\) by a geometric grid with \(O(\log n)\) cardinality without any loss in the statistical rate. Indeed, such a modification is standard as Lepski's method is very frequently implemented with a geometric grid in the literature. 
        
        To be clear, some care is needed since the minimax rate \(\epsilon(k, n)\) does not satisfy \(\epsilon(k, n) \asymp \epsilon(2k, n)\) for all \(k\), which the specialist recognizes as an important property for justifying the use of a geometric grid in Lepski's method. However, the property \(\epsilon(k, n) \asymp \epsilon(2k, n)\) \emph{does} hold for, say, \(1 \leq k \leq \frac{n}{8}\). Furthermore, we have \(\epsilon(k, n) = \epsilon\left(\frac{n}{2} - \left(\frac{n}{2} - k\right), n\right) \asymp \epsilon\left(\frac{n}{2} - 2\left(\frac{n}{2}-k\right), n\right)\) for \(\frac{3n}{8} \leq k \leq \frac{n}{2}-C\sqrt{n}\). Note for the remaining sparsities \(\frac{n}{8} \leq k \leq \frac{3n}{8}\), we have \(\epsilon(k, n) \asymp \epsilon(3n/8, n)\). 
        
        In other words, the minimax rates for \(k\) and \(k'\) are the same when both \(k \asymp k'\) and \(\frac{n}{2} - k \asymp \frac{n}{2}-k'\). Consequently, one can replace \(\mathcal{K}\) with \(\mathcal{K}_1 \cup \mathcal{K}_2\) having \(O(\log n)\) cardinality, where \(\mathcal{K}_1\) is a geometric grid over \(\{1,...,n/8\}\) and \(\mathcal{K}_2\) is a geometric grid over \(\{3n/8,...,n/2-C\sqrt{n}\}\).
    \end{remark}

    \subsection{General noise distributions}\label{section:deconvolution}
    The model (\ref{model:additive}) stipulates the noise is Gaussian. However, it is also interesting to consider the generic version of (\ref{model:additive}) where \(Z_1,...,Z_n \overset{iid}{\sim} F\) for some known symmetric distribution \(F\) (e.g. Cauchy, Laplace, etc.). For discussion, take \(\sigma^2 = 1\) to be known. 
    
    Since the parametric rate can always be achieved by the sample median for \(k \lesssim \sqrt{n}\), for sake of discussion consider the regime \(\sqrt{n} \leq k \leq \frac{n}{2}-C\sqrt{n}\). Our methodology requires only a slight modification to accommodate noise variables drawn from \(F\). Note in (\ref{estimator:theta_known_var}) the term \(e^{\frac{\omega^2}{2}}\) is the reciprocal of the characteristic function of \(N(0, 1)\) evaluated at \(\omega\). Therefore, a natural idea is to consider the estimator 
    \begin{equation*}
        \hat{\theta}_F = \argmin_{\mu \in \R} \sup_{|\omega| \leq \tau} \inf_{\substack{\zeta \in \C \\ |\zeta| \leq 1}} \left| \frac{1}{n} \sum_{j=1}^{n} \frac{e^{i\omega(X_j - \mu)}}{\psi_F(\omega)} - \frac{n-k}{n} - \frac{k}{n} \zeta \right|, 
    \end{equation*}
    where \(\psi_F(\omega) = E(e^{i\omega Z_1})\) is the characteristic function of \(F\).

    The arguments in the Gaussian case can be straightforwardly modified to show that with the choice \(\tau\) such that \(\frac{1}{\psi_F(\tau)} \lesssim \frac{k(n-2k)}{n^{3/2}}\), we have \(|\hat{\theta}_F - \theta| \lesssim \frac{k}{n} \tau^{-1}\) with high probability. For example, if \(F\) is the Laplace distribution with unit scale, then \(\psi_F(\omega) = \frac{1}{1+\omega^2}\), and so the choice \(\tau^2 \asymp \frac{k(n-2k)}{n^{3/2}}\) yields \(|\hat{\theta}_F - \theta|^2 \lesssim \frac{1}{\sqrt{n}} \cdot \frac{k}{(n-2k)}\). Notably, this Laplace rate is faster than the Gaussian rate (\ref{rate:theta}) by a polynomial factor! The improvement in the rate can be seen as the result of the Laplace distribution having a characteristic function which decays polynomially in \(\omega\) rather than exponentially like that of the standard Gaussian distribution. This phenomenon has been noted before in the deconvolution literature \cite{meister_deconvolution_2009,fan_optimal_1991,carroll_optimal_1988}. 

    The reader may be wondering about the existence of a choice of \(\tau\) such that \(\frac{1}{\psi_F(\tau)} \lesssim \frac{k(n-2k)}{n^{3/2}}\) to get the bound \(|\hat{\theta}_F - \theta| \lesssim \frac{k}{n} \tau^{-1}\). In the regime \(\sqrt{n} \lesssim k \leq \frac{n}{2} - C\sqrt{n}\), we always have \(\frac{k(n-2k)}{n^{3/2}} \gtrsim 1\). Since \(\psi_F\) is continuous with \(\psi_F(0) = 1\) (as it is a characteristic function), one can always pick \(\tau\) to be a sufficiently small constant depending only on \(F\) to satisfy \(\frac{1}{\psi_F(\tau)} \lesssim \frac{k(n-2k)}{n^{3/2}}\). Of course, the choice of constant \(\tau\) only delivers the Huber rate \(|\hat{\theta}_F - \theta| \lesssim \frac{k}{n}\), and so it is of great desire to choose growing \(\tau\). We conjecture that, for any \(F\), there indeed exists such a growing choice of \(\tau\) so that Huber's rate can always be beaten. Furthermore, an interesting problem is to prove a matching minimax lower bound for estimating \(\theta\) with a generic noise distribution \(F\). 
   
    \section{Proofs}
    This section contains the analyses of the Fourier-based location estimator of Section \ref{section:known_var_cf} and the variance estimator of Section \ref{section:sigmahat}.

    \subsection{\texorpdfstring{Fourier-based estimator: known variance}{Fourier-based estimator: known variance}}

    \begin{proof}[Proof of Theorem \ref{thm:cf_estimator}]
        Fix \(\delta \in (0, 1)\). It follows from an argument using the bounded differences inequality (e.g. Theorem \ref{thm:bounded_differences}) that for a sufficiently large \(L > 0\) depending only on \(\delta\), the event \(\mathcal{E} := \left\{ \sup_{\omega \in \R} \left| \frac{1}{n} \sum_{j=1}^{n} \left(e^{i\omega(X_j - \theta)} - E_{\theta, \gamma}\left(e^{i\omega(X_j - \theta)} \right)\right)\right| \leq \frac{L}{\sqrt{n}}\right\}\) has \(P_{\theta, \gamma}\)-probability of at least \(1-\delta\) uniformly over \((\theta, \gamma)\). Examining (\ref{estimator:theta_known_var}), by definition of \(\hat{\theta}\) we have on the event \(\mathcal{E}\), 
        \begin{align}
            &\sup_{|\omega| \leq \tau} \inf_{\substack{\zeta \in \C, \\ |\zeta| \leq 1}} \left| \frac{1}{n} \sum_{j=1}^{n} e^{i\omega(X_j - \hat{\theta}) + \frac{\omega^2}{2}} - \frac{n-k}{n} - \frac{k}{n} \zeta\right| \nonumber \\
            &\leq \sup_{|\omega| \leq \tau} \inf_{\substack{\zeta \in \C, \\ |\zeta| \leq 1}} \left| \frac{1}{n} \sum_{j=1}^{n} e^{i\omega(X_j - \theta) + \frac{\omega^2}{2}} - \frac{n-k}{n} - \frac{k}{n} \zeta\right| \nonumber \\
            &\leq \sup_{|\omega| \leq \tau} \left| \frac{1}{n} \sum_{j=1}^{n} e^{i\omega(X_j - \theta) + \frac{\omega^2}{2}} - \frac{n-k}{n} - \frac{k}{n} \cdot \frac{1}{k} \sum_{j \in \mathcal{O}} e^{i\omega \gamma_j}\right| \nonumber \\
            &= \sup_{|\omega| \leq \tau} \left| \frac{1}{n} \sum_{j=1}^{n} \left(e^{i\omega(X_j - \theta) + \frac{\omega^2}{2}} - E_{\theta, \gamma}\left(e^{i\omega(X_j - \theta) + \frac{\omega^2}{2}} \right)\right)\right| \nonumber \\
            &\leq \frac{Le^{\frac{\tau^2}{2}} }{\sqrt{n}}. \label{eqn:cf_theta_upperbound}
        \end{align}
        Now, consider that on \(\mathcal{E}\) we also have 
        \begin{align*}
            \sup_{|\omega| \leq \tau }\left|\frac{1}{n} \sum_{j=1}^{n} \left(e^{i\omega(X_j - \hat{\theta}) + \frac{\omega^2}{2}} - e^{i\omega(\theta + \gamma_j - \hat{\theta})}\right)\right| &= \sup_{|\omega| \leq \tau }\left|e^{-i\omega\hat{\theta} + \frac{\omega^2}{2}}\right| \left|\frac{1}{n} \sum_{j=1}^{n} \left(e^{i\omega X_j} - e^{i\omega(\theta + \gamma_j) - \frac{\omega^2}{2}}\right)\right| \\
            &\leq e^{\frac{\tau^2}{2}} \cdot \sup_{\omega \in \R}\left|\frac{1}{n} \sum_{j=1}^{n} \left(e^{i\omega X_j} - E_{\theta, \gamma}\left(e^{i\omega X_j}\right)\right)\right| \\
            &= e^{\frac{\tau^2}{2}} \cdot \sup_{\omega \in \R}\left|\frac{1}{n} \sum_{j=1}^{n} \left(e^{i\omega (X_j-\theta)} - E_{\theta, \gamma}\left(e^{i\omega (X_j - \theta)}\right)\right)\right| \\
            &\leq \frac{Le^{\frac{\tau^2}{2}}}{\sqrt{n}}. 
        \end{align*}
        Therefore, by reverse triangle inequality, it follows that on \(\mathcal{E}\) 
        \begin{align}
            &\sup_{|\omega| \leq \tau} \inf_{\substack{\zeta \in \C, \\ |\zeta| \leq 1}} \left| \frac{1}{n} \sum_{j=1}^{n} e^{i\omega(X_j - \hat{\theta}) + \frac{\omega^2}{2}} - \frac{n-k}{n} - \frac{k}{n} \zeta\right| \nonumber \\
            &\geq \sup_{|\omega| \leq \tau} \inf_{\substack{\zeta \in \C, \\ |\zeta| \leq 1}} \left| \frac{1}{n} \sum_{j=1}^{n} e^{i\omega(\theta + \gamma_j - \hat{\theta})} - \frac{n-k}{n} - \frac{k}{n} \zeta\right| - \frac{Le^{\frac{\tau^2}{2}}}{\sqrt{n}} \nonumber \\
            &= \sup_{|\omega| \leq \tau} \inf_{\substack{\zeta \in \C, \\ |\zeta| \leq 1}} \left| \frac{n-k}{n}\left(e^{i\omega (\theta - \hat{\theta})} - 1\right) + \frac{k}{n}\left(\frac{1}{k} \sum_{j \in \mathcal{O}'} e^{i\omega (\theta+\gamma_j - \hat{\theta})} - \zeta \right)\right| - \frac{Le^{\frac{\tau^2}{2}}}{\sqrt{n}} \nonumber \\
            &\geq \sup_{|\omega| \leq \tau} \inf_{\substack{z \in \C, \\ |z| \leq 2}} \left| \frac{n-k}{n}\left(e^{i\omega (\theta - \hat{\theta})} - 1\right) + \frac{k}{n}z\right| - \frac{Le^{\frac{\tau^2}{2}}}{\sqrt{n}}.  \label{eqn:reverse_triangle}
        \end{align}
        Here, \(\mathcal{O}'\) is defined to be \(\mathcal{O}\) with any arbitrary \(k - |\mathcal{O}|\) indices added from \(\mathcal{I}\). This is simply to ensure \(|\mathcal{O}'| = k\). Let us now examine the first term on the right hand side of (\ref{eqn:reverse_triangle}). We break up the analysis into two cases to separately obtain the parametric part \(\frac{1}{\sqrt{n}}\) and the nonparametric part \(\frac{k}{n \tau}\) in our target bound. Let \(C_1 > 0\) be sufficiently large depending only on \(\delta\). \newline
        
        \noindent \textbf{Case 1:} Suppose \(k \leq C_1 \sqrt{n}\). The argument is fairly simple to obtain the parametric rate. From (\ref{eqn:reverse_triangle}) and (\ref{eqn:cf_theta_upperbound}) it follows 
        \begin{equation*}
            \sup_{|\omega| \leq \tau} \left|e^{i\omega(\theta - \hat{\theta})} - 1\right| \leq \left(\frac{n}{n-k}\right) \left(\frac{2Le^{\frac{\tau^2}{2}}}{\sqrt{n}} + \frac{2k}{n}\right) \leq  \frac{1}{\sqrt{n}}\left(\frac{1}{1 - \frac{C_1}{\sqrt{n}}}\right)\left(2Le^{\frac{\tau^2}{2}} + 2C_1\right) \leq \frac{\tilde{C}}{\sqrt{n}}
        \end{equation*}
        where \(\tilde{C} > 0\) depends only on \(\delta\). Here, we have used \(n\) is sufficiently large (depending only on \(\delta\)). We have also used \(k \leq C_1 \sqrt{n}\) implies \(\log\left(1 + \frac{k^2(n-2k)^2}{n^3}\right) \leq \log\left(1 + C_1^2\right)\), and so \(\tau = 1\) by selecting \(c\) sufficiently small depending only on \(\delta\). Let \(\tilde{C}' = \frac{12}{5}\tilde{C}\). We claim \(|\hat{\theta} - \theta| \leq \frac{\tilde{C}'}{\sqrt{n}}\). To prove the claim, suppose not, that is, suppose \(|\hat{\theta} - \theta| > \frac{\tilde{C}'}{\sqrt{n}}\). Consider the choice \(\omega^* = \frac{\tilde{C}'}{\sqrt{n}(\theta - \hat{\theta})}\) and observe \(|\omega^*| \leq 1 = \tau\). Therefore, from the inequality \(\frac{5}{6}x \leq \sin x \leq |e^{ix} - 1|\) for \(x \in (0, 1)\), we have 
        \begin{align*}
            \frac{2\tilde{C}}{\sqrt{n}} = \frac{5}{6} \cdot \frac{\tilde{C}'}{\sqrt{n}} = \frac{5}{6} \omega^*(\theta - \hat{\theta}) \leq \left|e^{i\omega^*(\theta - \hat{\theta})} - 1\right| \leq \sup_{|\omega| \leq \tau} \left|e^{i\omega(\theta - \hat{\theta})} - 1\right| \leq \frac{\tilde{C}}{\sqrt{n}}
        \end{align*}
        which is a contradiction. Hence, we have the claim. To summarize, we have shown \(|\hat{\theta} - \theta| \leq \frac{\tilde{C}'}{\sqrt{n}}\) on the event \(\mathcal{E}\) when \(k \leq C_1 \sqrt{n}\), as desired. \newline 

        \noindent \textbf{Case 2:} Suppose \(k > C_1 \sqrt{n}\). For any fixed \(\omega\) with \(|\omega| \leq \tau\), it is straightforward to solve the optimization problem 
        \begin{equation*}
            \inf_{\substack{z \in \C, \\ |z| \leq 2}} \left| \frac{n-k}{n}\left(e^{i\omega (\theta - \hat{\theta})} - 1\right) + \frac{k}{n}z\right|
        \end{equation*}
        directly and obtain the solution 
        \begin{equation*}
            z^* = 
            \begin{cases}
                -\frac{n-k}{k} \left(e^{i\omega (\theta - \hat{\theta})} - 1\right) &\textit{if } \left|e^{i\omega(\theta - \hat{\theta})} - 1\right| \leq \frac{2k}{n-k}, \\
                - \frac{2}{\left|e^{i\omega(\theta - \hat{\theta})} - 1\right|} \cdot \left(e^{i\omega(\theta - \hat{\theta})} - 1\right) &\textit{otherwise}. 
            \end{cases}
        \end{equation*}
        Plugging in \(z^*\), we see the value of the optimization problem is zero when \(\left|e^{i\omega(\theta - \hat{\theta})} - 1\right| \leq \frac{2k}{n-k}\). Therefore,
        \begin{align*}
            \sup_{|\omega| \leq \tau} \inf_{\substack{z \in \C, \\ |z| \leq 2}} \left| \frac{n-k}{n}\left(e^{i\omega (\theta - \hat{\theta})} - 1\right) + \frac{k}{n}z\right| &= \sup_{|\omega| \leq \tau} \left| \frac{n-k}{n}\left(e^{i\omega (\theta - \hat{\theta})} - 1\right) + \frac{k}{n}z^*\right| \\
            &= \sup_{\substack{|\omega| \leq \tau, \\ \left|e^{i\omega(\theta - \hat{\theta})} - 1\right| > \frac{2k}{n-k}}} \left| \frac{n-k}{n}\left(e^{i\omega (\theta - \hat{\theta})} - 1\right) + \frac{k}{n}z^*\right| \\
            &= \sup_{\substack{|\omega| \leq \tau, \\ \left|e^{i\omega(\theta - \hat{\theta})} - 1\right| > \frac{2k}{n-k}}} \left|\frac{n-k}{n} - \frac{2k/n}{|e^{i\omega(\theta - \hat{\theta})} - 1|} \right| \left| e^{i\omega (\theta - \hat{\theta})} - 1\right|.
        \end{align*}

        To summarize, from (\ref{eqn:cf_theta_upperbound}) and (\ref{eqn:reverse_triangle}) it has been shown that on \(\mathcal{E}\) we have 
        \begin{equation}\label{eqn:setup_contradiction}
            \sup_{\substack{|\omega| \leq \tau, \\ \left|e^{i\omega(\theta - \hat{\theta})} - 1\right| > \frac{2k}{n-k}}} \left|\frac{n-k}{n} - \frac{2k/n}{|e^{i\omega(\theta - \hat{\theta})} - 1|} \right| \left| e^{i\omega (\theta - \hat{\theta})} - 1\right| \leq \frac{2Le^{\frac{\tau^2}{2}}}{\sqrt{n}}. 
        \end{equation}
        Set \(C = \left(\frac{4L}{4-\pi}\right)^2\) and take \(c \leq \frac{1}{2}\) (and sufficiently small so that the Case 1 analysis goes through). We claim that (\ref{eqn:setup_contradiction}) implies \(|\hat{\theta} - \theta| \leq 8\frac{k}{n} \tau^{-1}\). For sake of contradiction, suppose \(|\hat{\theta} - \theta| > \frac{8k}{n} \tau^{-1}\). Now, consider the choice \(\omega^* := \frac{\pi \wedge (8k/n)}{\theta - \hat{\theta}}\). Note that \(|\omega^*| = \frac{\pi \wedge (8k/n)}{|\hat{\theta} - \theta|} \leq \frac{8}{8} \tau = \tau\). Further consider \(|e^{i\omega^*(\theta - \hat{\theta})} - 1| = |e^{i(\pi \wedge (8k/n))} - 1|\). If \(\pi \leq 8k/n\), then \(|e^{i\omega^*(\theta - \hat{\theta})} - 1| = 2 > \frac{2k}{n-k}\) since \(k < \frac{n}{2}\). On the other hand, if \(\pi > 8k/n\), then \(|e^{i\omega^*(\theta - \hat{\theta})} - 1| = |e^{i(8k/n)} - 1| \geq \frac{4k}{n} > \frac{2k}{n-k}\). Hence, \(\omega^*\) lives in the domain of optimization in (\ref{eqn:setup_contradiction}) and so 
        \begin{align*}
            \frac{2Le^{\frac{\tau^2}{2}}}{\sqrt{n}}
            &\geq 
            \begin{cases}
                \left(\frac{n-k}{n} - \frac{1}{2}\right)\left(\frac{4k}{n}\right) &\textit{if } \frac{8k}{n} < \pi, \\
                2 \cdot \frac{n-2k}{n} &\textit{if } \frac{8k}{n} \geq \pi
            \end{cases} \\
            &\geq 
            \begin{cases}
                \frac{4-\pi}{2} \cdot \frac{k}{n} &\textit{if } \frac{8k}{n} < \pi, \\
                2 \cdot \frac{n-2k}{n} &\textit{if } \frac{8k}{n} \geq \pi.
            \end{cases} \\
            &\geq 
            \begin{cases}
                \frac{4-\pi}{2} \cdot \frac{k}{n} \cdot \frac{n-2k}{n} &\textit{if } \frac{8k}{n} < \pi, \\
                \frac{4-\pi}{2} \cdot \frac{n-2k}{n} \cdot \frac{k}{n} &\textit{if } \frac{8k}{n} \geq \pi.
            \end{cases} \\
            &= \frac{4-\pi}{2} \cdot \frac{(n-2k)k}{n^2}
        \end{align*}
        where we have used that \(\frac{n-2k}{n}\) and \(\frac{k}{n}\) are both less than or equal to \(1\) in the penultimate step. Rearranging, we have 
        \begin{equation*}
            1 \leq \frac{4L}{4-\pi} \cdot \frac{n^{3/2}}{(n-2k)k} e^{\frac{\tau^2}{2}}.  
        \end{equation*}
        \noindent Since \(C_1 \sqrt{n} \leq k \leq \frac{n}{2} - C\sqrt{n}\), we can take \(C_1, C\) large depending only on \(\delta\) to obtain \(\frac{(n-2k)^2k^2}{n^3} \geq C_2\) for a sufficiently large \(C_2 > 0\) depending only on \(\delta\). Consequently, 
        \begin{equation*}
            \tau^2 = 1 \vee c^2 \log\left(1 + \frac{(n-2k)^2k^2}{n^3}\right) < \frac{1}{2} \log\left(\frac{(n-2k)^2k^2}{n^3}\right) \leq \log\left(\frac{(n-2k)^2k^2}{n^3}\right) - 2\log\left(\frac{4L}{4-\pi}\right).
        \end{equation*}
        Therefore, \(1 \leq \frac{4L}{4-\pi} \cdot \frac{n^{3/2}}{(n-2k)k} e^{\frac{\tau^2}{2}} < 1\) and so we have a contradiction. Hence, \(|\hat{\theta} - \theta| \leq \frac{8k}{n} \tau^{-1}\) on \(\mathcal{E}\). The proof is complete.         
    \end{proof}

    \subsection{Variance estimation}
    \begin{proof}[Proof of Theorem \ref{thm:var}]
        Fix \(\delta \in (0, 1)\). It follows from an argument using the bounded differences inequality (e.g. Theorem \ref{thm:bounded_differences}) that for a sufficiently large \(L > 0\) depending only on \(\delta\), the event \(\mathcal{E} := \left\{ \sup_{\omega \in \R} \left| \frac{1}{n} \sum_{j=1}^{n} \left(e^{i\omega X_j} - E_{\theta, \gamma, \sigma}\left(e^{i\omega X_j} \right)\right)\right| \leq \frac{L}{\sqrt{n}}\right\}\) satisfies \(\inf_{\substack{\theta \in \R, \\ ||\gamma||_0 \leq k, \\ \sigma > 0}} P_{\theta, \gamma, \sigma}(\mathcal{E}) \geq 1-\frac{\delta}{2}\). For \(\omega \in \R\), define \(N(\omega) := \left|\frac{1}{n}\sum_{j=1}^{n} e^{i\omega(\theta+\gamma_j) - \frac{\sigma^2\omega^2}{2}}\right|\). By definition of \(\hat{\sigma}^2\), for any \(a \leq \omega \leq b\), we have on the event \(\mathcal{E}\)  
        \begin{align}
            \hat{\sigma}^2 - \sigma^2 &\leq \frac{2}{\omega^2}\left|\log \hat{N}(\omega) - \log N(\omega)\right| - \frac{2\log N(\omega)}{\omega^2} - \sigma^2 \nonumber \\
            &\leq \frac{2}{\omega^2} \frac{|\hat{N}(\omega) - N(\omega)|}{\hat{N}(\omega) \wedge N(\omega)} - \frac{2 \log N(\omega)}{\omega^2} - \sigma^2 \nonumber \\
            &\leq \frac{2}{\omega^2} \frac{\frac{L}{\sqrt{n}}}{\left(N(\omega) - \frac{L}{\sqrt{n}}\right)_{+}} - \frac{2 \log N(\omega)}{\omega^2} - \sigma^2 \label{eqn:sigma_hat_upperbound}
        \end{align}
        where we have used the inequality \(|\log(x) - \log(y)|\leq \frac{|x - y|}{x \wedge y}\) for \(x, y > 0\). On the other hand, consider that for every \(\omega \in \R\), we have on the event \(\mathcal{E}\) 
        \begin{align*}
            -\frac{2\log \hat{N}(\omega)}{\omega^2} &\geq -\frac{2 \log\left(N(\omega) + \frac{L}{\sqrt{n}}\right)}{\omega^2} \\
            &= -\frac{2\log\left(e^{-\frac{\sigma^2\omega^2}{2}} \left|\frac{1}{n}\sum_{j=1}^{n} e^{i\omega(\theta+\gamma_j)}\right| + \frac{L}{\sqrt{n}}\right)}{\omega^2} \\
            &\geq -\frac{2 \log\left(e^{-\frac{\sigma^2\omega^2}{2}} + \frac{L}{\sqrt{n}}\right)}{\omega^2} \\
            &= -\frac{2}{\omega^2} \left(\log\left(e^{-\frac{\sigma^2\omega^2}{2}} + \frac{L}{\sqrt{n}}\right) - \log\left(e^{-\frac{\sigma^2\omega^2}{2}}\right)\right) + \sigma^2.
        \end{align*}
        Rearranging and using the inequality \(|\log(x)-\log(y)| \leq \frac{|x-y|}{ x \wedge y}\) for \(x, y > 0\), we have 
        \begin{align*}
            -\frac{2\log \hat{N}(\omega)}{\omega^2} - \sigma^2 \geq -\frac{2}{\omega^2} \frac{\frac{L}{\sqrt{n}}}{e^{-\frac{\sigma^2\omega^2}{2}}}. 
        \end{align*}
        Taking infimum over \(\omega \in [a, b]\) yields 
        \begin{equation}\label{eqn:sigma_hat_lowerbound}
            \hat{\sigma}^2 - \sigma^2 \geq - \sup_{\omega \in [a, b]} \frac{2}{\omega^2} \cdot \frac{Le^{\frac{\sigma^2\omega^2}{2}}}{\sqrt{n}}. 
        \end{equation}

        \noindent To summarize, combining our upper and lower bounds yields 
        \begin{align}\label{eqn:sigma_hat_error}
            |\hat{\sigma}^2 - \sigma^2| &\leq \left(\sup_{\omega \in [a, b]} \frac{2}{\omega^2} \cdot \frac{Le^{\frac{\sigma^2\omega^2}{2}}}{\sqrt{n}}\right) + \inf_{\omega \in [a, b]} \left( \frac{2}{\omega^2} \frac{\frac{L}{\sqrt{n}}}{\left(N(\omega) - \frac{L}{\sqrt{n}}\right)_{+}} - \frac{2 \log N(\omega)}{\omega^2} - \sigma^2 \right) \nonumber \\
            &= \left(\sup_{\omega \in [a, b]} \frac{2}{\omega^2} \cdot \frac{Le^{\frac{\sigma^2\omega^2}{2}}}{\sqrt{n}}\right) + \inf_{\omega \in [a, b]} \left( \frac{2}{\omega^2} \frac{\frac{L}{\sqrt{n}}}{\left(N(\omega) - \frac{L}{\sqrt{n}}\right)_{+}} - \frac{2}{\omega^2} \log \left|\frac{1}{n}\sum_{j=1}^{n} e^{i\omega(\theta+\gamma_j)} \right| \right).
        \end{align}
        Define the event \(\mathcal{E}' = \left\{\frac{1}{L'} \leq \frac{\tilde{\sigma}}{\sigma} \leq L'\right\}\) where \(L'\) depends only on \(\delta\) and, according to Proposition \ref{prop:var_tilde}, ensures \(\inf_{\substack{\theta \in \R, \\ ||\gamma||_0 \leq k, \\ \sigma > 0}} P_{\theta, \gamma, \sigma}(\mathcal{E}') \geq 1-\frac{\delta}{2}\). Let us examine the second term in (\ref{eqn:sigma_hat_error}). Let \(\mathcal{O}'\) be the set given by \(\mathcal{O}\) along with some arbitrary additional indices from \(\mathcal{I}\) to ensure \(|\mathcal{O}'| = k\). Then  
        \begin{equation}\label{eqn:error_bound}
            \left|\frac{1}{n} \sum_{j=1}^{n} e^{i\omega(\theta + \gamma_j)}\right| = \left|1 - \frac{k}{n}\left(1 - \frac{1}{k} \sum_{j \in \mathcal{O}'} e^{i\omega \gamma_j}\right)\right| \geq 1 - \frac{k}{n}\left|1 - \frac{1}{k} \sum_{j \in \mathcal{O}'} e^{i\omega \gamma_j}\right|. 
        \end{equation}
        Consider that 
        \begin{equation*}
            \left|1 - \frac{1}{k}\sum_{j\in \mathcal{O}'} e^{i\omega \gamma_j}\right|^2 = 1 + \left|\frac{1}{k}\sum_{j \in \mathcal{O}'}e^{i\omega \gamma_j}\right|^2 - \frac{2}{k} \sum_{j\in \mathcal{O}'} \cos(\omega \gamma_j) \leq 2\left(1 - \frac{1}{k}\sum_{j \in \mathcal{O}'} \cos(\omega \gamma_j)\right). 
        \end{equation*}
        Proposition \ref{prop:cosines} implies there exists \(\omega^* \in [a, b]\) such that \(\frac{1}{k}\sum_{j \in \mathcal{O}'} \cos(\omega^* \gamma_j) \geq -\frac{1}{5}\), and so (\ref{eqn:error_bound}) can be bounded as \(\left|\frac{1}{n} \sum_{j=1}^{n} e^{i\omega^*(\theta + \gamma_j)}\right| \geq \frac{n-k\sqrt{\frac{12}{5}}}{n}\). Note this directly implies \(N(\omega^*) \geq e^{-\frac{\sigma^2(\omega^*)^2}{2}}\left(\frac{n-k\sqrt{\frac{12}{5}}}{n}\right)\). Therefore, the bound (\ref{eqn:sigma_hat_error}) reduces to 
        \begin{align}
            |\hat{\sigma}^2 - \sigma^2| &\leq \left(\sup_{\omega \in [a, b]} \frac{2}{\omega^2} \cdot \frac{Le^{\frac{\sigma^2\omega^2}{2}}}{\sqrt{n}}\right) + \frac{2}{(\omega^*)^2}\left( \frac{\frac{L}{\sqrt{n}}}{\left(e^{-\frac{\sigma^2(\omega^*)^2}{2}}\left(1 - \sqrt{\frac{3}{5}}\right)-\frac{L}{\sqrt{n}}\right)_{+}} + \log\left(\frac{1}{1-\frac{k}{n}\sqrt{\frac{12}{5}}}\right)\right) \nonumber \\
            &\leq \frac{2}{a^2} \cdot \frac{Le^{\frac{\sigma^2b^2}{2}}}{\sqrt{n}} + \frac{2}{a^2}\left( \frac{\frac{L}{\sqrt{n}}}{\left(e^{-\frac{\sigma^2b^2}{2}}\left(1 - \sqrt{\frac{3}{5}}\right)-\frac{L}{\sqrt{n}}\right)_{+}} + \frac{4k}{n}\right) \label{eqn:sigma_hat_masterbound}
        \end{align}
        where we have used \(\frac{k}{n} < \frac{1}{2}\) to obtain \(\frac{n-k\sqrt{\frac{12}{5}}}{n} \geq 1 - \sqrt{\frac{3}{5}}\). We have also used the inequality \(\log\left(1/\left(1-x\sqrt{\frac{12}{5}}\right)\right) \leq 4x\) for \(0 \leq x \leq \frac{1}{2}\). Consider that \(a = c\tilde{\sigma}^{-1} \sqrt{1 \vee \log\left(\frac{ek^2}{n}\right)}\). On the event \(\mathcal{E}'\) we have \(a \geq \frac{c}{L'} \sigma^{-1} \sqrt{1\vee \log\left(\frac{ek^2}{n}\right)}\) and \(b \leq 100L' c \sigma^{-1} \sqrt{1 \vee \log\left(\frac{ek^2}{n}\right)}\). Choosing \(c\) appropriately small depending only on \(L, L'\) and noting \(n\) is sufficiently large depending only on \(\delta\), it thus follows from (\ref{eqn:sigma_hat_masterbound}) that on the event \(\mathcal{E} \cap \mathcal{E}'\)
        \begin{equation*}
            |\hat{\sigma}^2 - \sigma^2| \leq \frac{C'\sigma^2}{1 \vee \log\left(\frac{ek^2}{n}\right)} \left(\frac{1 \vee \frac{k}{\sqrt{n}}}{\sqrt{n}} + \frac{4k}{n} \right) \leq \frac{C' \sigma^2 k}{n \log\left(1 + \frac{k}{\sqrt{n}}\right)}
        \end{equation*}
        for some constant \(C' > 0\) depending only on \(\delta\) (whose value may change from instance to instance). Since union bound asserts \(\mathcal{E} \cap \mathcal{E}'\) has \(P_{\theta, \gamma, \sigma}\)-probability at least \(1-\delta\) uniformly over the parameter space, the proof is complete.
    \end{proof}

    \bibliographystyle{skotekal}
    \bibliography{null_estimation}

    \appendix

    \section{\texorpdfstring{Fourier-based estimator: unknown variance}{Fourier-based estimator: unknown variance}}

    \begin{proof}[Proof of Theorem \ref{thm:theta_unknown_var}]
        Fix \(\delta \in (0, 1)\). It follows from an argument using the bounded differences inequality (e.g. Theorem \ref{thm:bounded_differences}) that for a sufficiently large \(L > 0\) depending only on \(\delta\), the event 
        \begin{equation*}
            \mathcal{E} := \left\{ \sup_{\omega \in \R} \left|\frac{1}{n} \sum_{j=1}^{n} e^{i\omega(X_j - \theta)} - E_{\theta, \gamma, \sigma}(e^{i\omega(X_j - \theta)})\right| \leq \frac{L}{\sqrt{n}}\right\} 
        \end{equation*}
        has \(P_{\theta, \gamma, \sigma}\)-probability of at least \(1-\frac{\delta}{2}\) uniformly over the parameter space. Let us also define the event 
        \begin{equation*}
            \mathcal{E}_{\text{var}} := \left\{ \sigma^2 \in [\sigma_{-}^2, \sigma_{+}^2] \right\}.    
        \end{equation*}
        For an appropriate \(R\), Theorem \ref{thm:var} asserts that \(\mathcal{E}_{\text{var}}\) has \(P_{\theta, \gamma, \sigma}\)-probability of at least \(1 - \frac{\delta}{2}\) uniformly over the parameter space. For ease of notation, let us define the function
        \begin{equation*}
            F(\mu, v) := \sup_{|\omega| \leq \tau} \inf_{\substack{\zeta \in \C, \\ |\zeta| \leq 1}} \left| \frac{1}{n} \sum_{j=1}^{n} e^{i\omega (X_j - \mu) + \frac{v^2\omega^2}{2}} - \frac{n-k}{n} - \frac{k}{n}\zeta\right|. 
        \end{equation*}
        By definition of \(\hat{\theta}\) and \(\hat{v}\), we have on the event \(\mathcal{E} \cap \mathcal{E}_{\text{var}}\) 
        \begin{equation}\label{eqn:F_upperbound}
            F(\hat{\theta}, \hat{v}) \leq F(\theta, \sigma^2) \\
            \leq \frac{Le^{\frac{\sigma^2\tau^2}{2}}}{\sqrt{n}} + \sup_{|\omega| \leq \tau} \inf_{\substack{\zeta \in \C, \\ |\zeta| \leq 1}} \left|\frac{1}{n} \sum_{j=1}^{n} e^{i\omega \gamma_j} - \frac{n-k}{n} - \frac{k}{n} \zeta\right| \\
            = \frac{Le^{\frac{\sigma^2\tau^2}{2}}}{\sqrt{n}}. 
        \end{equation}
        On the other hand, consider that on the event \(\mathcal{E} \cap \mathcal{E}_{\text{var}}\) 
        \begin{align}
            F(\hat{\theta}, \hat{v}) &= \sup_{|\omega| \leq \tau} \inf_{\substack{\zeta \in \C, \\ |\zeta| \leq 1}} \left|\frac{1}{n} \sum_{j=1}^{n} e^{i\omega(X_j - \hat{\theta}) + \frac{\hat{v}\omega^2}{2}} - \frac{n-k}{n} - \frac{k}{n} \zeta \right| \nonumber \\
            &\geq - \frac{L e^{\frac{\hat{v}\tau^2}{2}}}{\sqrt{n}} + \sup_{|\omega| \leq \tau} \inf_{\substack{\zeta \in \C, \\ |\zeta| \leq 1}} \left|\frac{1}{n} \sum_{j=1}^{n} e^{i\omega(\theta + \gamma_j - \hat{\theta}) + \frac{(\hat{v} - \sigma^2)\omega^2}{2}} - \frac{n-k}{n} - \frac{k}{n} \zeta \right| \nonumber \\
            &= - \frac{L e^{\frac{\hat{v}\tau^2}{2}}}{\sqrt{n}} + \sup_{|\omega| \leq \tau} \inf_{\substack{\zeta \in \C, \\ |\zeta| \leq 1}} \left|\frac{n-k}{n} \left( e^{i\omega(\theta - \hat{\theta}) + \frac{(\hat{v} - \sigma^2)\omega^2}{2}}-1\right) + \frac{k}{n} \left(\frac{1}{k} \sum_{j \in \mathcal{O}'} e^{i\omega(\gamma_j - \hat{\theta}) + \frac{(\hat{v} -\sigma^2)\omega^2}{2}} - \zeta\right) \right| \nonumber \\
            &\geq - \frac{L e^{\frac{\hat{v}\tau^2}{2}}}{\sqrt{n}} + \sup_{|\omega| \leq \tau} \inf_{\substack{z \in \C, \\ |z| \leq \alpha }} \left|\frac{n-k}{n} \left( e^{i\omega(\theta - \hat{\theta}) + \frac{(\hat{v} - \sigma^2)\omega^2}{2}}-1\right) + \frac{k}{n} z \right| \label{eqn:F_lowerbound}
        \end{align}
        where \(\mathcal{O}'\) is the set obtained by taking \(\mathcal{O}\) and adding arbitrary indices from \(\mathcal{I}\) if needed to ensure \(|\mathcal{O}'| = k\). Also, here we denote \(\alpha = 1 + e^{\frac{(\hat{v} - \sigma^2)\omega^2}{2}}\). Similar to the proof of Theorem \ref{thm:cf_estimator}, we break up the analysis into two cases to separately obtain the parametric part and the nonparametric part of our target bound. Let \(C_1 > 0\) be a sufficiently large constant depending only on \(\delta\). \newline

        \noindent \textbf{Case 1:} Suppose \(k \leq C_1 \sqrt{n}\). From (\ref{eqn:F_upperbound}) and (\ref{eqn:F_lowerbound}) we have 
        \begin{align*}
            \sup_{|\omega| \leq \tau} \left|e^{i\omega (\theta - \hat{\theta}) + \frac{(\hat{v} - \sigma^2)\omega^2}{2}} - 1\right| \leq \frac{n}{n-k}\left(\frac{k}{n}\left(1 + e^{\frac{|\hat{v} - \sigma^2|\tau^2}{2}}\right) + \frac{L\left(e^{\frac{\sigma^2\tau^2}{2}} + e^{\frac{\hat{v}^2\tau^2}{2}}\right)}{\sqrt{n}}\right).
        \end{align*}
        Since we are working on the event \(\mathcal{E} \cap \mathcal{E}_{\text{var}}\) where we have \(\hat{v}, \sigma^2 \in [\sigma_{-}^2, \sigma_{+}^2]\), by taking \(c\) sufficiently small depending only on \(\delta\) we have 
        \begin{equation*}
            |\hat{v} - \sigma^2|\tau^2 \leq \frac{2Rk}{n\log\left(1 + \frac{k}{\sqrt{n}}\right)} \cdot \left(1 \vee c^2 \log\left(1 + \frac{k^2(n-2k)^2}{n^3}\right)\right) \leq \frac{2Rk}{n\log\left(1 + \frac{k}{C_1\sqrt{n}}\right)} \leq \frac{4RC_1}{\sqrt{n}}
        \end{equation*}
        where we have used that one can take \(C_1 \geq 1\) and we have used the inequality \(\frac{x}{2} \leq \log(1+x)\) for \(0 \leq x \leq 1\). Since \(n\) is sufficiently large to ensure \(\frac{4RC_1}{\sqrt{n}}\), we have \(1 + e^{\frac{|\hat{v} - \sigma^2|\tau^2}{2}} \leq 2 + \frac{2eRC_1}{\sqrt{n}}\) from the inequality \(e^x \leq 1 + ex\) for \(0 \leq x \leq 1\). Further, since \(n\) is sufficiently large and \(c\) sufficiently small we have 
        \begin{align*}
            \sigma^2 \tau^2 &= \frac{\sigma^2}{\hat{\sigma}^2}\left(1 \vee c^2 \log\left(1 + \frac{k^2(n-2k)^2}{n^3}\right)\right) \leq \frac{1}{1-\frac{Rk}{n\log\left(1 + \frac{k}{\sqrt{n}}\right)}} \left(1 \vee c^2\log\left(1 + \frac{k^2(n-2k)^2}{n^3}\right)\right) \leq 2, \\
            \hat{v}^2 \tau^2 &\leq \left(1 + \frac{Rk}{n\log\left(1 + \frac{k}{\sqrt{n}}\right)}\right) \left(1 \vee c^2\log\left(1 + \frac{k^2(n-2k)^2}{n^3}\right)\right) \leq 2. 
        \end{align*}
        Putting together these bounds, we have 
        \begin{equation*}
            \sup_{|\omega| \leq \tau} \left|e^{i\omega (\theta - \hat{\theta}) + \frac{(\hat{v} - \sigma^2)\omega^2}{2}} - 1\right| \leq \frac{n}{n-k} \left(\frac{k}{n} \left(2 + \frac{2eRC_1}{\sqrt{n}}\right) + \frac{2Le}{\sqrt{n}}\right) \leq \frac{\tilde{C}}{\sqrt{n}}
        \end{equation*}
        for some positive constant \(\tilde{C}\) depending only on \(\delta\). We claim \(|\hat{\theta} - \theta| \leq \frac{\tilde{C}'\hat{\sigma}}{\sqrt{n}}\) where \(\tilde{C}'\) is a sufficiently large constant depending only on \(\delta\). To prove the claim, suppose not, that is, suppose \(|\hat{\theta} - \theta| > \frac{\tilde{C}'\hat{\sigma}}{\sqrt{n}}\). Consider the choice \(\omega^* = \frac{\tilde{C}'}{\sqrt{n}(\theta - \hat{\theta})}\) satisfies \(|\omega^*| \leq \hat{\sigma}^{-1} \leq \tau\). Using the inequality \(|ye^{ix} - 1| \geq |\cos(x)e^{ix} - 1|\) for all \(x, y \in \R\), we have 
        \begin{equation*}
            \left|\cos\left(\frac{\tilde{C}'}{\sqrt{n}}\right)e^{\frac{i\tilde{C}'}{\sqrt{n}}} - 1\right| \leq \left|e^{i\omega^*(\theta - \hat{\theta}) + \frac{(\hat{v} - \sigma^2)(\omega^*)^2}{2}} - 1\right| \leq \sup_{|\omega| \leq \tau} \left| e^{i\omega(\theta - \hat{\theta}) + \frac{(\hat{v}-\sigma^2)\omega^2}{2}} - 1 \right| \leq \frac{\tilde{C}}{\sqrt{n}}. 
        \end{equation*}
        Clearly this gives a contradiction when \(\tilde{C}'\) and \(n\) are sufficiently large depending only on \(\delta\). Hence, we have established the claim \(|\hat{\theta} - \theta| \leq \frac{\tilde{C}'\hat{\sigma}}{\sqrt{n}}\). Since we are working on the event \(\mathcal{E} \cap \mathcal{E}_{\text{var}}\), it thus follows \(|\hat{\theta} - \theta| \leq \tilde{C}''\frac{\sigma}{\sqrt{n}}\) for some constant \(\tilde{C}''\) depending only on \(\delta\). Thus we have the desired result. \newline

        \noindent \textbf{Case 2:} Suppose \(k > C_1\sqrt{n}\). Consider the second term in (\ref{eqn:F_lowerbound}). Analogous to the proof of Theorem \ref{thm:cf_estimator}, the minimizer solving 
        \begin{equation*}
            \inf_{\substack{z \in \C, \\ |z| \leq \alpha }} \left|\frac{n-k}{n} \left( e^{i\omega(\theta - \hat{\theta}) + \frac{(\hat{v} - \sigma^2)\omega^2}{2}}-1\right) + \frac{k}{n} z \right|
        \end{equation*}
        is 
        \begin{equation*}
            z^* = 
            \begin{cases}
                -\frac{n-k}{k} \left(e^{i\omega(\theta-\hat{\theta}) + \frac{(\hat{v} - \sigma^2)\omega^2}{2}} - 1\right) &\textit{if } \left| e^{i\omega(\theta-\hat{\theta}) + \frac{(\hat{v} - \sigma^2)\omega^2}{2}} - 1 \right| \leq \frac{\alpha k}{n-k}, \\
                - \frac{\alpha}{\left|e^{i\omega(\theta-\hat{\theta}) + \frac{(\hat{v} - \sigma^2)\omega^2}{2}} - 1\right|} \left(e^{i\omega(\theta-\hat{\theta}) + \frac{(\hat{v} - \sigma^2)\omega^2}{2}} - 1\right) &\textit{otherwise}.
            \end{cases}
        \end{equation*}
        Therefore, arguing analogously as in the proof of Theorem \ref{thm:cf_estimator}, we have  
        \begin{align*}
            &\sup_{|\omega| \leq \tau} \inf_{\substack{z \in \C, \\ |z| \leq \alpha }} \left|\frac{n-k}{n} \left( e^{i\omega(\theta - \hat{\theta}) + \frac{(\hat{v} - \sigma^2)\omega^2}{2}}-1\right) + \frac{k}{n} z \right| \\
            &= \sup_{|\omega| \leq \tau} \left|\frac{n-k}{n} \left( e^{i\omega(\theta - \hat{\theta}) + \frac{(\hat{v} - \sigma^2)\omega^2}{2}}-1\right) + \frac{k}{n} z^*\right| \\
            &= \sup_{\substack{|\omega| \leq \tau, \\ \left| e^{i\omega(\theta - \hat{\theta}) + \frac{(\hat{v} - \sigma^2)\omega^2}{2}}-1\right| > \frac{\alpha k}{n-k}}} \left|\frac{n-k}{n} \left( e^{i\omega(\theta - \hat{\theta}) + \frac{(\hat{v} - \sigma^2)\omega^2}{2}}-1\right) + \frac{k}{n} z^*\right| \\
            &= \sup_{\substack{|\omega| \leq \tau, \\ \left| e^{i\omega(\theta - \hat{\theta}) + \frac{(\hat{v} - \sigma^2)\omega^2}{2}}-1\right| > \frac{\alpha k}{n-k}}} \left|\frac{n-k}{n} -  \frac{\alpha k}{n\left|e^{i\omega(\theta - \hat{\theta}) + \frac{(\hat{v} - \sigma^2)\omega^2}{2}}-1\right|}\right| \left|e^{i\omega(\theta - \hat{\theta}) + \frac{(\hat{v} - \sigma^2)\omega^2}{2}}-1\right|.
        \end{align*}

        \noindent To summarize, from (\ref{eqn:F_upperbound}) and (\ref{eqn:F_lowerbound}), we have 
        \begin{equation}\label{eqn:F_upperbound_2}
            \sup_{\substack{|\omega| \leq \tau, \\ \left| e^{i\omega(\theta - \hat{\theta}) + \frac{(\hat{v} - \sigma^2)\omega^2}{2}}-1\right| > \frac{\alpha k}{n-k}}} \left|\frac{n-k}{n} -  \frac{\alpha k}{n\left|e^{i\omega(\theta - \hat{\theta}) + \frac{(\hat{v} - \sigma^2)\omega^2}{2}}-1\right|}\right| \left|e^{i\omega(\theta - \hat{\theta}) + \frac{(\hat{v} - \sigma^2)\omega^2}{2}}-1\right| \leq \frac{L \left( e^{\frac{\sigma^2\tau^2}{2}} + e^{\frac{\hat{v}^2\tau^2}{2}}\right)}{\sqrt{n}}.
        \end{equation}
        The argument continues similarly as in the proof of Theorem \ref{thm:cf_estimator}. We claim (\ref{eqn:F_upperbound_2}) implies \(|\hat{\theta} - \theta| \leq \frac{16k}{n} \tau^{-1}\). For sake of contradiction, suppose \(|\hat{\theta} - \theta| > \frac{16k}{n} \tau^{-1}\). Consider the choice \(\omega^* := \frac{\pi \wedge (16k/n)}{\theta - \hat{\theta}}\). Note \(|\omega^*| = \frac{\pi \wedge (16k/n)}{|\hat{\theta} - \theta|} \leq \tau\). Let \(\alpha^*\) denote the value of \(\alpha\) corresponding to \(\omega^*\). If \(\pi \leq 16k/n\), then \(\omega^* = \frac{\pi}{\theta - \hat{\theta}}\) and so 
        \begin{equation*}
            \left|e^{i\omega^* (\theta - \hat{\theta}) + \frac{(\hat{v} - \sigma^2)(\omega^*)^2}{2}}-1\right| = \alpha^* > \frac{\alpha^* k}{n-k}.
        \end{equation*}
        
        \noindent Likewise, if \(\pi > 16k/n\), then 
        \begin{equation*}
            \left|e^{i\omega^*(\theta - \hat{\theta}) + \frac{(\hat{v} - \sigma^2)(\omega^*)^2}{2}}-1\right| = \left|e^{i(16k/n) + \frac{(\hat{v} - \sigma^2)(\omega^*)^2}{2}}-1\right| > \frac{5k}{n}. 
        \end{equation*}
        Here, we have used the inequality \(|ye^{ix} - 1| > \frac{5}{16} x\) which holds at least on the region \(\{(x,y) \in \R^2 : y > 0 \text{ and } 0 \leq x \leq \pi\}\). We have applied this inequality with the choice \(x = \frac{16k}{n} < \pi\) and \(y = \exp\left(\frac{(\hat{v} - \sigma^2)(\omega^*)^2}{2}\right) > 0\). Further consider that since \((\omega^*)^2 \leq \tau^2\) and since \(\hat{v}, \sigma^2 \in [\sigma_{-}^2, \sigma_{+}^2]\) as we are working on \(\mathcal{E} \cap \mathcal{E}_{\text{var}}\), we have
        \begin{align*}
            \alpha^* &= 1 + e^{\frac{(\hat{v} - \sigma^2)(\omega^*)^2}{2}} \\
            &\leq 1 + e^{\frac{|\hat{v} - \sigma^2|\tau^2}{2}} \\
            &\leq 1 + \exp\left(\frac{R}{2}\frac{k}{n\log\left(1 + \frac{k}{\sqrt{n}}\right)}\left(1 \vee c^2 \log\left(1 + \frac{k^2(n-2k)^2}{n^3}\right)\right) \right) \\
            &\leq 1 + \exp\left(\frac{c^2Rk}{n}\right)
        \end{align*}
        where we have used \(C_1 < k < \frac{n}{2} - C\sqrt{n}\) for \(C_1\) and \(C\) sufficiently large. Consider that for \(c\) such that \(c^2R \leq 1\), we have \(e^{c^2Rk/n} \leq 1 + \frac{ec^2Rk}{n}\) by the inequality \(e^x \leq 1 + ex\) for \(x \in [0, 1]\). Therefore, taking \(c\) sufficiently small to ensure \(ec^2R \leq 1\) gives us \(\alpha^* \leq 2 + \frac{k}{n} \leq \frac{5}{2}\). Hence, 
        \begin{equation*}
            \left|e^{i\omega^*(\theta - \hat{\theta}) + \frac{(\hat{v} - \sigma^2)(\omega^*)^2}{2}}-1\right| > \frac{5k}{n} \geq \frac{\alpha^* k}{n-k}. 
        \end{equation*}

        \noindent In summary, \(\omega^*\) lives in the domain of optimization in (\ref{eqn:F_upperbound_2}). Therefore, from (\ref{eqn:F_upperbound})
        \begin{align*}
            \frac{L \left(e^{\frac{\sigma^2\tau^2}{2}} + e^{\frac{\hat{v}^2\tau^2}{2}}\right)}{\sqrt{n}} &\geq 
            \begin{cases}
                \left(\frac{n-2k}{n}\right) \alpha^* &\textit{if } \pi \leq 16k/n, \\
                \left(\frac{n-k}{n} - \frac{\alpha^*}{5} \right)\left(\frac{5k}{n}\right) &\textit{if } \pi > 16k/n
            \end{cases}
            \\
            &\geq 
            \begin{cases}
                \frac{4-\pi}{8}\left(\frac{n-2k}{n}\right) &\textit{if } \pi \leq 16k/n, \\
                \frac{4-\pi}{8}\left(\frac{5k}{n}\right) &\textit{if } \pi > 16k/n
            \end{cases} 
            \\
            &\geq
            \begin{cases}
                \frac{4-\pi}{8}\left(\frac{n-2k}{n}\right)\left(\frac{k}{n}\right) &\textit{if } \pi \leq 16k/n, \\
                \frac{4-\pi}{8}\left(\frac{k}{n}\right) \left(\frac{n-2k}{n}\right) &\textit{if } \pi > 16k/n
            \end{cases} 
            \\
            &=  \frac{4-\pi}{8} \cdot \frac{(n-2k)k}{n^2}.
        \end{align*}
        Note we have used \(\alpha^* \geq 1\). Rearranging, we have \(1 \leq \frac{16L}{4-\pi} \cdot \frac{n^{3/2}}{(n-2k)k} \left(e^{\frac{\sigma^2 \tau^2}{2}} + e^{\frac{\hat{v}^2\tau^2}{2}}\right)\). Since \(\tau^2 = \hat{\sigma}^{-2}\left(1 \vee c^2 \log\left(1 + \frac{k^2(n-2k)^2}{n^3}\right)\right)\), since \(n\) is larger than a sufficiently large constant depending only on \(\delta\), and since \(C_1 < k < \frac{n}{2} - C\sqrt{n}\) for \(C_1\) and \(C\) sufficiently large, we have 
        \begin{align*}
            \sigma^2 \tau^2 &= \frac{\sigma^2}{\hat{\sigma}^2}\left(1 \vee c^2 \log\left(1 + \frac{k^2(n-2k)^2}{n^3}\right)\right) \\
            &\leq \frac{1}{1-\frac{Rk}{n\log\left(1 + \frac{k}{\sqrt{n}}\right)}} \left(1 \vee c^2\log\left(1 + \frac{k^2(n-2k)^2}{n^3}\right)\right) \\
            &\leq 2c^2\log\left(1 + \frac{k^2(n-2k)^2}{n^3}\right),
        \end{align*}
        and 
        \begin{align*}
            \hat{v}^2 \tau^2 \leq \left(1 + \frac{Rk}{n\log\left(1 + \frac{k}{\sqrt{n}}\right)}\right) \left(1 \vee c^2\log\left(1 + \frac{k^2(n-2k)^2}{n^3}\right)\right) \leq 2c^2\log\left(1 + \frac{k^2(n-2k)^2}{n^3}\right). 
        \end{align*}
        Taking \(c\) sufficiently small and \(C\) sufficiently large yields 
        \begin{equation*}
            1 \leq \frac{32L}{4-\pi} \cdot \frac{n^{3/2}}{(n-2k)k} e^{2c^2\log\left(1 + \frac{k^2(n-2k)^2}{n^3}\right)} < 1
        \end{equation*}
        which is a contradiction. Hence, on the event \(\mathcal{E} \cap \mathcal{E}_{\text{var}}\) we have \(|\hat{\theta} - \theta| \leq \frac{32k}{n} \tau^{-1}\), which clearly yields our desired result.  
    \end{proof}
    
    \section{Lower bounds}\label{appendix:lower_bound_proofs}
    In this section, we prove Theorems \ref{thm:lower_bound_cf} and \ref{thm:inconsistent_lbound}. Recall that the Fourier transform of an integrable function \(f\) is \(\hat{f}(t) = \int e^{-itx} f(x)\,dx\) and the Fourier transform of a finite Borel measure \(\pi\) is \(\hat{\pi}(t) = \int e^{-itx} \pi(dx)\). Recall also that \(P_{\theta, \gamma}\) denotes the joint distribution of the data \(\{X_i\}_{i=1}^{n}\) generated from the model (\ref{model:additive}) with parameters \(\theta\) and \(\gamma\) (recall from Section \ref{section:known_var_cf} that suppression of \(\sigma\) in the notation denotes \(\sigma = 1\)).

    \subsection{Proof of Theorem \ref{thm:lower_bound_cf}}\label{appendix:cf_lower_bound_proof}

    \begin{proof}[Proof of Theorem \ref{thm:lower_bound_cf}] 
        For ease of notation, set 
        \begin{equation*}
            \psi = \frac{k}{n}\log^{-1/2}\left(1 + \frac{k^2(n-2k)^2}{n^3}\right).
        \end{equation*}
        If \(n\) is bounded above by a universal constant, it is trivial to show the lower bound \(\psi \asymp 1\). Suppose \(n\) is larger than a sufficiently big universal constant. If \(k \lesssim \sqrt{n}\), then \(\psi \asymp \frac{k}{n} \log^{-1/2}\left(1 + \frac{k^2}{n}\right) \asymp \frac{k}{n} \cdot \frac{\sqrt{n}}{k} = n^{-1/2}\) where we have used \(\frac{x}{2} \leq \log(1+x) \leq x\) for \(x \in [0, 1]\). Thus \(\psi\) is the parametric rate and a very standard two point argument will establish the lower bound. Hence, we can limit our attention to the case \(k \geq C' \sqrt{n}\) for a sufficiently large universal \(C' > 0\) and \(n\) larger than a sufficiently big universal constant. 

        We now construct the priors for \((\theta, \gamma)\) to establish the lower bound. Let \(f_0\) and \(f_1\) denote the probability densities given by Proposition \ref{prop:prior_chisquare} with the choice \(\varepsilon = \frac{k-10\sqrt{n}}{n}\) and \(c = 1\). Let \(\mu, \tau, \lambda, g_0,\) and \(g_1\) be the associated quantities from Proposition \ref{prop:prior_chisquare}. Note that 
        \begin{equation*}
            \mu = \frac{c_0\lambda}{\tau} = \left(\frac{c_0}{1+2\varepsilon}\right) \frac{\varepsilon}{1 \vee B_c \sqrt{\log(en\varepsilon^2(1-2\varepsilon)^2)}} \asymp \psi
        \end{equation*}
        where \(c_0\) is a small universal constant and \(B_c\) is a large universal constant. Note to conclude \(\mu \asymp \psi\) we have used that \(k \leq \frac{n}{2} - \sqrt{n}\) implies \(1-2\varepsilon = \frac{n-2k+20\sqrt{n}}{n} \asymp \frac{n-2k}{n}\).  Define the priors \(\pi_0\) and \(\pi_1\) as follows. Let \(\delta_1,...,\delta_n \overset{iid}{\sim} \Bernoulli\left(\varepsilon\right)\). A draw \((\theta, \gamma) \sim \pi_0\) is obtained by setting \(\theta = 0\) and setting \(\gamma = (\delta_1v_1,...,\delta_nv_n)\) where \(v_1,...,v_n \overset{iid}{\sim} g_0 * \delta_\mu\). Likewise, a draw \((\theta, \gamma) \sim \pi_1\) is obtained by setting \(\theta = 2\mu\) and setting \(\gamma = (\delta_1v_1,...,\delta_n v_n)\) where \((v_1,...,v_n) \overset{iid}{\sim} g_1*\delta_{-\mu}\). 
        
        Note that \(\pi_0\) and \(\pi_1\) are not supported on the parameter space. Namely, it is not the case that \(||\gamma||_0 \leq k\) with probability one under either prior since the \(\delta_i\) are independent. We are only guaranteed \(E\left(||\gamma||_0\right) \leq k\). Consequently, define the truncated versions \(\bar{\pi}_0\) and \(\bar{\pi}_1\) where for any event \(A\), 
        \begin{equation*}
            \bar{\pi}_i(A) = \frac{\pi_i\left(A \cap \left\{||\gamma||_0 \leq k\right\}\right)}{\pi_i\left(\left\{||\gamma||_0 \leq k\right\}\right)} 
        \end{equation*}
        for \(i=0,1\). Denote the mixture distributions \(P_i = \int P_{\theta, \gamma} \pi_i(d\theta,d\gamma)\) and \(\bar{P}_i = \int P_{\theta,\gamma} \bar{\pi}_i(d\theta,d\gamma)\) for \(i=0,1\). Note that \(P_0\) and \(P_1\) admits densities \(f_0^{\otimes n}\) and \(f_1^{\otimes n}\) where \(f_0\) and \(f_1\) are given in Proposition \ref{prop:prior_chisquare}. With these definitions in hand, consider 
        \begin{align*}
            &\inf_{\hat{\theta}} \sup_{\substack{\theta \in \R, \\ ||\gamma||_0 \leq k}} P_{\theta, \gamma}\left\{ |\hat{\theta} - \theta| \geq \mu \right\} \\
            &\geq \inf_{\hat{\theta}} \max\left\{\bar{P}_0\left\{ |\hat{\theta} - \theta| \geq \mu \right\}, \bar{P}_1\left\{ |\hat{\theta} - \theta| \geq \mu \right\}\right\} \\
            &= \inf_{\hat{\theta}} \max\left\{\bar{P}_0\left\{ |\hat{\theta}| \geq \mu \right\}, \bar{P}_1\left\{ |\hat{\theta} - 2\mu| \geq \mu  \right\}\right\} \\
            &\geq \inf_{\hat{\theta}} \max\left\{\bar{P}_0\left\{ |\hat{\theta}| > \mu \right\}, \bar{P}_1\left\{ |\hat{\theta}| \leq \mu \right\}\right\} \\
            &\geq \inf_{A} \max\left\{\bar{P}_0(A), \bar{P}_1(A^c)\right\} \\
            &\geq \frac{1}{2}\left(1 - \dTV(\bar{P}_1, \bar{P}_0)\right)
        \end{align*}
        where the infimum runs over all events \(A\). We have also used \(\left\{|\hat{\theta} - 2\mu| \geq \mu\right\} \supset \left\{|\hat{\theta}| \leq \mu\right\}\), which follows from triangle inequality. Note triangle inequality further gives \(\dTV(\bar{P}_1, \bar{P}_0) \leq \dTV(P_1, P_0) + \dTV(P_0, \bar{P}_0) + \dTV(\bar{P}_1, P_1)\). The data processing inequality gives \(\dTV(\bar{P}_i, P_i) \leq \dTV(\bar{\pi}_i, \pi_i)\) for \(i=0,1\). Consider that for any event \(B\), we have \(\pi_i(B) - \bar{\pi}_i(B) \leq \pi_i(B) - \pi_i(B \cap \{||\gamma||_0 \leq k\}) \leq \pi_i(\{||\gamma||_0 > k\})\). Likewise, \(\pi_i(B) - \bar{\pi}_i(B) = \pi_i(B \cap \{||\gamma||_0 \leq k\}) + \pi_i(B \cap \{||\gamma||_0 > k\}) - \frac{\pi_i(B \cap \{||\gamma||_0 \leq k\})}{\pi_i(\{||\gamma||_0 \leq k\})} \geq \pi_i(B \cap \{||\gamma||_0 \leq k\}) \left(1 - \frac{1}{\pi_i(\{||\gamma||_0 \leq k\})}\right) = -\pi_{i}(\{||\gamma||_0 > k\})\). Therefore, we have \(|\pi_i(B) - \bar{\pi}_i(B)| \leq \pi_i(\{||\gamma||_0 > k\})\). Taking supremum over all events \(B\) yields 
        \begin{align*}
            \dTV(\pi_i, \bar{\pi}_i) = \pi_i(\{||\gamma||_0 > k\}) \leq P\left\{\sum_{i=1}^{n} \delta_i > k \right\} = P\left\{\text{Binomial}\left(n, \varepsilon\right) > k\right\} \leq \frac{k-10\sqrt{n}}{100n} \leq \frac{1}{100}
        \end{align*}
        where we have used Chebyshev's inequality. To summarize, these calculations imply \(\dTV(\bar{P}_1, \bar{P}_0) \leq \dTV(P_0, P_1) + \frac{1}{50}\) and so it remains to bound \(\dTV(P_0, P_1)\). Note that if \(X \sim P_i\) for \(i = 0,1\), then \(\{X_j\}_{j=1}^{n}\) are mutually independent. Therefore, by Proposition \ref{prop:prior_chisquare} we have 
        \begin{equation*}
            \dTV(P_1, P_0) \leq \frac{1}{2}\sqrt{\chi^2(P_1 \,||\, P_0)} = \frac{1}{2}\sqrt{(1 + \chi^2(f_1 \,||\, f_0))^n - 1} \leq \frac{1}{2}\sqrt{e-1}.
        \end{equation*}
        Therefore, we have \(\dTV(\bar{P}_1, \bar{P}_0) \leq \frac{1}{2}\sqrt{e-1} + \frac{1}{50}\). Thus, 
        \begin{equation*}
            \inf_{\hat{\theta}} \sup_{\substack{\theta \in \R, \\ ||\gamma||_0 \leq k}} P_{\theta, \gamma}\left\{ |\hat{\theta} - \theta| \geq \mu \right\} \geq \frac{1}{2}\left(1 -  \frac{1}{2}\sqrt{e-1} + \frac{1}{50}\right) \geq \frac{1}{10}.
        \end{equation*}
        Since \(\mu \asymp \psi\), we have the desired result. The proof is complete.
    \end{proof}

    \begin{proposition}\label{prop:prior_construction}
        Suppose \(0 < \lambda \leq \frac{1}{2}\) and \(\tau > 0\). There exists a universal positive constant \(c_0 < 1\) such that the following holds. Let \(\mu = \frac{c_0\lambda}{\tau}\) and define the probability density function \(p_0 : \R \to [0, \infty)\) with
        \begin{equation*}
            p_0(x) = 
            \begin{cases}
                \frac{\tau}{4} &\textit{if } |x| \leq \frac{1}{\tau}, \\
                \frac{1}{4\tau x^2} &\textit{if } |x| > \frac{1}{\tau}.
            \end{cases}
        \end{equation*}
        Further define the real-valued function \(\Delta = -\frac{\hat{h}}{2\pi}\) where \(h : \R \to \C\) is the purely imaginary-valued function 
        \begin{equation*}
            h(t) = 
            \begin{cases}
                0 &\textit{if } t < -2\tau, \\
                -j(-t) &\textit{if } -2\tau \leq t \leq -\tau, \\
                k(t) &\textit{if } -\tau < t < \tau, \\
                j(t) &\textit{if } \tau \leq t \leq 2\tau, \\
                0 &\textit{if } t > 2\tau, 
            \end{cases}
        \end{equation*}
        with \(k(t) = 2i\frac{1-\lambda}{\lambda} \sin(t\mu)\) and \(j(t) = 2i\frac{1-\lambda}{\lambda} \sin(\tau \mu) \frac{2\tau - t}{\tau}\). Then the function \(p_1 = p_0 + \Delta\) is a probability density function.
    \end{proposition}
    \begin{proof}
        Since \(p_1 = p_0 + \Delta\), it suffices to show \(p_1\) is nonnegative and \(g\) integrates to zero in order to show \(p_1\) is a probability density function. We first show \(p_1\) is nonnegative provided \(c_0\) is sufficiently small. Since \(h\) is a purely-imaginary odd function, it follows that
        \begin{equation*}
            \Delta(x) = -\frac{\hat{h}(x)}{2\pi} = -\frac{1}{2\pi}\left(-2i \int_{0}^{\tau} k(t) \sin(tx)\,dt - 2i \int_{\tau}^{2\tau} j(t)\sin(tx) \,dt\right). 
        \end{equation*}
        By the power series representation \(k(t) = 2i\frac{1-\lambda}{\lambda} \sin(t\mu) = 2i \frac{1-\lambda}{\lambda}\sum_{n=0}^{\infty} \frac{(-1)^{n} (t\mu)^{2n+1}}{(2n+1)!}\) and Lemma \ref{lemma:k},
        \begin{align*}
            &-2i \int_{0}^{\tau} k(t) \sin(tx) \, dt \\
            &= 4\frac{1-\lambda}{\lambda} \sum_{n=0}^{\infty} \frac{(-1)^{n}(c_0\lambda)^{2n+1}}{(2n+1)!} \int_{0}^{\tau} \frac{t^{2n+1}}{\tau^{2n+1}} \sin(tx) \,dt \\
            &= 4\frac{1-\lambda}{\lambda} \sum_{n=0}^{\infty} \frac{(-1)^{n}(c_0\lambda)^{2n+1}}{(2n+1)!} \left( -\frac{\cos(\tau x)}{x} + R_{k,n}(x) \right).
        \end{align*}
        Likewise, since \(\tau \mu = c_0\lambda\) we have \(j(t) = 2i\frac{1-\lambda}{\lambda} \sum_{n=0}^{\infty} \frac{(-1)^{n}(c_0\lambda)^{2n+1}}{(2n+1)!} \frac{2\tau - t}{\tau}\). By Lemma \ref{lemma:j}
        \begin{align*}
            -2i \int_{\tau}^{2\tau} j(t) \sin(tx) \, dt = 4 \frac{1-\lambda}{\lambda} \sum_{n=0}^{\infty} \frac{(-1)^{n}(c_0\lambda)^{2n+1}}{(2n+1)!} \left( \frac{\cos(\tau x)}{x} + R_{j}(x) \right).
        \end{align*}
        Note that \(|R_{k,n}(x)| \vee |R_{j}(x)| \leq \frac{C(1 \vee n)}{\tau x^2} \vee \frac{2}{\tau x^2}\) for a universal positive constant \(C\) by Lemmas \ref{lemma:k} and \ref{lemma:j}. Therefore, it follows that 
        \begin{equation*}
            \Delta(x) = -\frac{4}{2\pi}\frac{1-\lambda}{\lambda} \sum_{n=0}^{\infty} \frac{(-1)^{n}(c_0\lambda)^{2n+1}}{(2n+1)!} (R_{k,n}(x) + R_{j}(x))
        \end{equation*}
        and so \(|\Delta(x)| \leq \frac{Cc_0}{\tau x^2}\) for a large universal constant \(C > 0\). Therefore, selecting \(c_0\) sufficiently small depending only on \(C\) ensures \(|\Delta(x)| \leq p_0(x)\) for \(|x| > \frac{1}{\tau}\) and so \(p_1\) is nonnegative here. Now let us show \(p_1\) is nonnegative on \(|x| \leq \frac{1}{\tau}\). First, consider \(I := \int_{0}^{\tau} \frac{t}{\tau} \sin(tx) \,dt\) satisfies \(|I| \leq \frac{1}{\tau} \int_{0}^{\tau} t \, dt = \frac{\tau}{2}\). Hence, \(|-\frac{4}{2\pi}(1-\lambda)c_0 I| \leq 2c_0\tau\). Further, consider that
        \begin{align*}
            &\left|\Delta(x) - \left(-\frac{4}{2\pi}(1-\lambda)c_0 I\right)\right| \\
            &\leq \frac{1}{2\pi}\left|-2i \int_{0}^{\tau} k(t)\sin(tx) \,dt - 4(1-\lambda)c_0 \int_{0}^{\tau} \frac{t}{\tau} \sin(tx) \,dt\right| + \frac{1}{2\pi}\left| -2i\int_{\tau}^{2\tau} j(t)\sin(tx) \,dt\right|\\
            &\leq \left| 4\frac{1-\lambda}{\lambda} \int_{0}^{\tau} \left(\sin(t\mu) - t\mu\right) \sin(tx) \,dt \right| + \left| 4\frac{1-\lambda}{\lambda} \sum_{n=0}^{\infty} \frac{(-1)^{n}(c_0\lambda)^{2n+1}}{(2n+1)!} \int_{\tau}^{2\tau} \frac{2\tau - t}{\tau} \sin(tx)\, dt \right|
        \end{align*}
        where we have used \(\tau = \frac{c_0\lambda}{\mu}\) to obtain the first term. Examining the first term, consider that \(|t\mu| \leq c_0 \lambda\) for \(|t| \leq \tau\) since \(\mu = \frac{c_0\lambda}{\tau}\). Therefore, \(|\sin(t\mu) - t\mu| \leq c_0^3 \lambda^3/6\). Thus, 
        \begin{equation*}
            \left|4\frac{1-\lambda}{\lambda} \int_{0}^{\tau} (\sin(t\mu) - t\mu)\sin(tx)\,dt\right| \leq 4\frac{1-\lambda}{\lambda} \int_{0}^{\tau} |\sin(t\mu) - t\mu| \, dt \leq \frac{4}{6}c_0\tau. 
        \end{equation*}
        Here, we have used \(c_0^3 \leq c_0\) (as we can take \(c_0 < 1\)) and \(\lambda^2(1-\lambda) \leq 1\) just to simplify expressions. Examining the second term, provided \(c_0 < 1\) we have 
        \begin{align*}
            4\frac{1-\lambda}{\lambda} \left|\sum_{n=0}^{\infty} \frac{(-1)^{n}(c_0\lambda)^{2n+1}}{(2n+1)!}\int_{\tau}^{2\tau} \frac{2\tau - t}{\tau} \sin(tx) \, dt \right| &\leq 4\frac{1-\lambda}{\lambda} \sum_{n=0}^{\infty} \frac{(c_0\lambda)^{2n+1}}{(2n+1)!} \cdot \frac{1}{\tau} \int_{\tau}^{2\tau} \left|2\tau - t\right| \, dt \\
            &\leq 4\frac{1-\lambda}{\lambda} \sum_{n=0}^{\infty} \frac{(c_0\lambda)^{2n+1}}{(2n+1)!} \cdot \frac{\tau}{2} \\
            &\leq 2 c_0 \tau \sum_{n=0}^{\infty} \frac{1}{(2n+1)!} \\
            &= 2c_0 \sinh(1)\tau. 
        \end{align*}
        Therefore, \(\left|\Delta(x) - \left(-\frac{4}{2\pi}(1-\lambda)c_0 I\right)\right| \leq \frac{4}{6} c_0 \tau + 2c_0 \sinh(1)\tau \leq 4c_0 \tau\). Therefore, from the bound \(\left|-\frac{4}{2\pi}(1-\lambda)c_0 I\right| \leq 2c_0\tau\), it follows that \(\Delta \geq -6c_0\tau \geq -\frac{\tau}{4}\) provided we have taken \(c_0 \leq \frac{1}{24}\). Thus, it immediately follows that \(p_1\) is nonnegative for \(|x| \leq \frac{1}{\tau}\), which concludes the proof of the claim that \(p_1\) is a nonnegative function. 

        It remains to show \(p_1\) integrates to one, and to do so it suffices to show \(\int \Delta = 0\). Note \(\Delta\) is uniformly bounded and \(|\Delta(x)| \leq \frac{Cc_0}{\tau x^2}\) as established earlier. Hence, \(\Delta\), and also \(\hat{h}\), is integrable. Therefore, \(\hat{\Delta}(t) = -\frac{1}{2\pi} \int \hat{h}(x) e^{-itx} \,dx = -h(-t) = h(t)\) since \(h\) is an odd function. Therefore, \(\int \Delta(x) \,dx = \hat{\Delta}(0) = h(0) = 0\). Thus \(p_1\) is a probability density function as claimed.  
    \end{proof}

    \begin{proposition}\label{prop:prior_chisquare}
        Suppose \(0 < \varepsilon \leq \frac{1}{2}\) and \(n \in \mathbb{N}\). For any \(c > 0\), there exists a large \(B_c > 0\) depending only on \(c\) such that the following holds. Let \(\mu, p_0,\) and \(p_1\) be given as in the statement of Proposition \ref{prop:prior_construction} with the choice \(\lambda = \frac{\varepsilon}{1+2\varepsilon}\) and \(\tau = 1 \vee B_c\sqrt{\log\left(en\varepsilon^2(1-2\varepsilon)^2\right)}\). Define the probability measures \(g_1 = 2\varepsilon \delta_{-\mu} + (1-2\varepsilon) p_1\) and \(g_0 = 2\varepsilon \delta_{\mu} + (1-2\varepsilon)p_0\) and the probability densities 
        \begin{align*}
            f_0 &= ((1-\varepsilon) \delta_0 + \varepsilon (g_0 * \delta_\mu)) * \varphi, \\
            f_1 &= ((1-\varepsilon) \delta_{2\mu} + \varepsilon(g_1 * \delta_{\mu})) * \varphi
        \end{align*}
        where \(\varphi\) denotes the probability density function of the standard Gaussian distribution and \(*\) denotes convolution. Then \(\chi^2(f_1 \,||\, f_0) \leq \frac{c}{n}\). 
    \end{proposition}
    \begin{proof}
        Note that \(f_0 \geq (1-\varepsilon) \varphi\). Since \(\varphi^{-1}(x) = \sqrt{2\pi} \sum_{k=0}^{\infty} \frac{x^{2k}}{2^k k!}\), we have 
        \begin{align*}
            \chi^2(f_1\,||\, f_0) &= \int \frac{(f_1 - f_0)^2}{f_0} \\
            &\leq \frac{\sqrt{2\pi}}{1-\varepsilon} \sum_{k=0}^{\infty} \frac{1}{2^k k!} \int x^{2k} (f_1(x) - f_0(x))^2 \, dx \\
            &\asymp \sum_{k=0}^{\infty} \frac{1}{2^k k!} \int \left|\hat{f}_1^{(k)}(t) - \hat{f}_0^{(k)}(t)\right|^2 \,dt. 
        \end{align*}
        Since the Fourier transform of a convolution of densities is the product of the Fourier transforms of the densities, we have 
        \begin{align*}
            \frac{\hat{f}_1(t) - \hat{f}_0(t)}{\hat{\varphi}(t)} &= (1-\varepsilon)\left(e^{-2it\mu} - 1\right) + \varepsilon e^{-it\mu} \left(\hat{g}_1(t) - \hat{g}_0(t)\right) \\
            &= (1-\varepsilon)\left(e^{-2it\mu} - 1\right) + \varepsilon e^{-it\mu}\left(2\varepsilon (e^{it\mu} - e^{-it\mu}) + (1-2\varepsilon)\hat{\Delta}(t)\right) \\
            &= (1-\varepsilon - 2\varepsilon^2)(e^{-2it\mu} - 1) + \varepsilon(1-2\varepsilon) e^{-it\mu} \hat{\Delta}(t) \\
            &= (1-2\varepsilon)(1+\varepsilon)(e^{-2it\mu} - 1) + \varepsilon(1-2\varepsilon)e^{-it\mu}\hat{\Delta}(t) \\
            &= (1-2\varepsilon)(1+2\varepsilon)e^{-it\mu}\left((1-\lambda)(e^{-it\mu} - e^{it\mu}) + \lambda \hat{\Delta}(t)\right)
        \end{align*}
        where \(\Delta = p_1 - p_0\) is the function from Proposition \ref{prop:prior_construction}. Recall, as argued in Proposition \ref{prop:prior_construction}, we have \(\hat{\Delta}(t) = h(t)\). Consequently, we have
        \begin{align*}
            \hat{f}_1(t) - \hat{f}_0(t) = 
            \begin{cases}
                0 &\textit{if } |t| < \tau, \\
                (1-2\varepsilon)(1+2\varepsilon)e^{-it\mu}\left((1-\lambda)(e^{-it\mu} - e^{it\mu}) + \lambda \hat{\Delta}(t)\right)\hat{\varphi}(t)  &\textit{if } \tau \leq |t| \leq 2\tau, \\
                (1-2\varepsilon)(1+\varepsilon)(e^{-2it\mu} - 1)\hat{\varphi}(t) &\textit{if } |t| > 2\tau. 
            \end{cases}
        \end{align*}
        Since \(\hat{f}_1(t) - \hat{f}_0(t)\) vanishes for \(|t| < \tau\), our bound for the \(\chi^2\) divergence simplifies 
        \begin{equation}\label{eqn:chisquare_fourier}
            \chi^2(f_1\,||\, f_0) \lesssim \sum_{k=0}^{\infty} \frac{1}{2^k k!} \int_{|t| \geq \tau} \left|\hat{f}_1^{(k)}(t) - \hat{f}_0^{(k)}(t)\right|^2 \,dt.
        \end{equation}
        The remainder of the proof is concerned with bounding this integral by splitting it across two different regions. Let us first consider the region \(|t| > 2\tau\). For such \(t\), observe that 
        \begin{align*}
            \left|\hat{f}_1^{(k)}(t) - \hat{f}_0^{(k)}(t)\right| &= (1-2\varepsilon)(1+\varepsilon) \left|\sum_{\ell=0}^{k} \binom{k}{\ell} \hat{\varphi}^{(\ell)}(t)\left(e^{-2it\mu} - 1\right)^{(k-\ell)}\right| \\
            &= (1-2\varepsilon)(1+\varepsilon)\left|(e^{-2it\mu} - 1)\hat{\varphi}^{(k)}(t) + \sum_{\ell=0}^{k-1} \binom{k}{\ell} \hat{\varphi}^{(\ell)}(t)(-2i\mu)^{k-\ell} e^{-2it\mu}  \right| \\
            &\leq 2(1-2\varepsilon)(1+\varepsilon)|\mu t|\left|\hat{\varphi}^{(k)}(t)\right| + 2(1-2\varepsilon)(1+\varepsilon)|\mu| \sum_{\ell=0}^{k-1} \binom{k}{\ell} |\hat{\varphi}^{(\ell)}(t)|
        \end{align*}
        where we have used \(2|\mu| \leq 1\) (as \(c_0\) in the definition of \(\mu\) can be taken sufficiently small and we have \(\tau \geq 1\)) implies \(|2\mu|^{k-\ell} \leq 2|\mu|\) for \(k-\ell \geq 1\). Note also we have used \(|e^{-2it\mu} - 1| \leq |2t\mu|\) to obtain the first term. Consider \(\hat{\varphi}(t) = e^{-\frac{t^2}{2}}\) and that Corollary \ref{corollary:cramer} gives \(|\hat{\varphi}^{(\ell)}(t)| \leq e^{-\frac{t^2}{4}}\sqrt{\ell!}\). Therefore, 
        \begin{align}
            &\sum_{k=0}^{\infty} \frac{1}{2^k k!} \int_{|t| \geq 2\tau} \left|\hat{f}_1^{(k)}(t) - \hat{f}_0^{(k)}(t)\right|^2 \,dt \nonumber \\
            &\lesssim (1-2\varepsilon)^2\mu^2\frac{(1+\varepsilon)^2}{1-\varepsilon}\sum_{k=0}^{\infty} \frac{1}{2^k k!} \left( \int_{|t| \geq 2\tau} t^2\left|\hat{\varphi}^{(k)}(t)\right|^2 \, dt + \int_{|t| \geq 2\tau} \left(\sum_{\ell=0}^{k-1} \binom{k}{\ell}|\hat{\varphi}^{(\ell)}(t)| \right)^2 \, dt \right) \nonumber \\
            &\lesssim (1-2\varepsilon)^2 \mu^2 e^{-L\tau^2} + (1-2\varepsilon)^2 \mu^2 \sum_{k=0}^{\infty} \frac{1}{2^k k!} \left( \sum_{\ell=0}^{k-1} \binom{k}{\ell} \sqrt{\ell!}\right)^2 \int_{|t| \geq 2\tau} e^{-\frac{t^2}{2}} \, dt \nonumber \\
            &\lesssim (1-2\varepsilon)^2 \mu^2 e^{-L\tau^2} \nonumber \\
            &\lesssim (1-2\varepsilon)^2 \varepsilon^2 e^{-L\tau^2}\label{eqn:2tau_bound}
        \end{align}
        where \(L > 0\) is a universal constant whose value may change from line to line. We have applied Lemma \ref{lemma:Laguerre} to obtain the penultimate line and \(\mu^2 = \frac{c_0^2\lambda^2}{\tau^2} \leq c_0^2\varepsilon^2\) to obtain the final line. Let us now consider the region \(\tau \leq |t| \leq 2\tau\). Recalling \(\hat{\Delta}(t) = h(t)\) as argued at the end of Proposition \ref{prop:prior_construction}, on this interval we have 
        \begin{align}\label{eqn:tau2tau_deriv}
            \left|\hat{f}_1^{(k)}(t) - \hat{f}_0^{(k)}(t)\right|^2 &\leq 2(1+2\varepsilon)^2(1-2\varepsilon)^2 (1-\lambda)^2\left|\frac{d^k}{dt^k} (e^{-2it\mu} - 1)\hat{\varphi}(t) \right|^2 \nonumber \\
            &\;\;\; + 2(1+2\varepsilon)^2(1-2\varepsilon)^2 \lambda^2 \left| \frac{d^k}{dt^k} e^{-it\mu} j(t) \hat{\varphi}(t)  \right|^2
        \end{align}
        where \(j(t) = 2i \frac{1-\lambda}{\lambda} \sin(\tau \mu) \frac{2\tau - t}{\tau}\) is the function from Proposition \ref{prop:prior_construction}. By repeating an argument similar to that yielding (\ref{eqn:2tau_bound}), the first term in (\ref{eqn:tau2tau_deriv}) yields the bound 
        \begin{equation}\label{eqn:tau2tau_deriv_1}
            \sum_{k=0}^{\infty} \frac{1}{2^k k!} \int_{\tau \leq |t| \leq 2\tau}  2(1+2\varepsilon)^2(1-2\varepsilon)^2(1-\lambda)^2\left|\frac{d^k}{dt^k} (e^{-2it\mu} - 1)\hat{\varphi}(t) \right|^2 \,dt \lesssim (1-2\varepsilon)^2 \varepsilon^2 e^{-L\tau^2}. 
        \end{equation}
        It remains to bound the second term in (\ref{eqn:tau2tau_deriv}). Consider that \(j^{(k)}(t) = 0\) for all \(k \geq 2\). Further consider \(|j^{(k)}(t)| \leq C\frac{1-\lambda}{\lambda} |\sin(\tau \mu)| \leq \frac{C}{\lambda} \cdot |\tau \mu| = C c_0\) for all \(\tau \leq t \leq 2\tau\) and \(k = 0, 1\). Here, \(C > 0\) is a universal constant whose value may change from instance to instance. Therefore, for any \(k \in \mathbb{N} \cup \{0\}\) we have
        \begin{equation*}
            \left|\frac{d^k}{dt^k} e^{-it\mu} j(t)\hat{\varphi}(t)\right| = \left|\sum_{\ell=0}^{k} \binom{k}{\ell} \left(e^{-it\mu} j(t)\right)^{(\ell)} \hat{\varphi}^{(k-\ell)}(t)\right| \leq \sum_{\ell=0}^{k} \binom{k}{\ell} \left|\left(e^{-it\mu} j(t)\right)^{(\ell)}\right| \left|\hat{\varphi}^{(k-\ell)}(t)\right|.
        \end{equation*}
        Consider that
        \begin{equation*}
            \left|\left(e^{-it\mu} j(t)\right)^{(\ell)}\right| = \left|\sum_{r=0}^{\ell} \binom{\ell}{r} \left(e^{-it\mu}\right)^{\ell-r} j^{(r)}(t) \right| \leq \sum_{r=0}^{\ell} \binom{\ell}{r} |\mu|^{\ell-r} |j^{(r)}(t)| \leq Cc_0 \max_{0 \leq r \leq 1} \binom{\ell}{r} \leq C c_0 (\ell \vee 1).
        \end{equation*}
        We have used \(|\mu| \leq 1\). Thus, we have 
        \begin{align*}
            \left|\frac{d^k}{dt^k} e^{-it\mu} j(t)\hat{\varphi}(t)\right| &\leq Cc_0 \sum_{\ell=0}^{k} \binom{k}{\ell} \left|\hat{\varphi}^{(k-\ell)}(t)\right| (\ell \vee 1)\\
            &\leq Cc_0 (k \vee 1) \sum_{\ell=0}^{k} \binom{k}{\ell} \left|\hat{\varphi}^{(k-\ell)}(t)\right| \\
            &\leq Cc_0 (k \vee 1) \sum_{\ell=0}^{k} \binom{k}{\ell} \left|\hat{\varphi}^{(\ell)}(t)\right|.
        \end{align*}
        Arguing similarly to before, we have via Corollary \ref{corollary:cramer}
        \begin{align}\label{eqn:tau2tau_deriv_2}
            &\sum_{k=0}^{\infty} \frac{1}{2^k k!} \int_{\tau \leq |t| \leq 2\tau}  2(1+2\varepsilon)^2(1-2\varepsilon)^2 \lambda^2 \left| \frac{d^k}{dt^k} e^{-it\mu} j(t) \hat{\varphi}(t)  \right|^2 \nonumber \\
            &\leq Cc_0^2 (1-2\varepsilon)^2\lambda^2 \sum_{k=0}^{\infty} \frac{k^2 \vee 1}{2^k k!} \int_{\tau \leq |t| \leq 2\tau} \left(\sum_{\ell=0}^{k} \binom{k}{\ell} \left|\hat{\varphi}^{(\ell)}(t)\right|\right)^2 \, dt \nonumber \\
            &\leq Cc_0^2 (1-2\varepsilon)^2\lambda^2 \sum_{k=0}^{\infty} \frac{k^2 \vee 1}{2^k k!} \left(\sum_{\ell=0}^{k} \binom{k}{\ell} \sqrt{\ell!}\right)^2 \int_{\tau \leq |t| \leq 2\tau} e^{-\frac{t^2}{2}} \, dt \nonumber \\
            &\leq Cc_0^2 (1-2\varepsilon)^2\varepsilon^2 e^{-L\tau^2}. 
        \end{align}
        Here, we have used \(\lambda^2 \leq \varepsilon^2\) and Lemma \ref{lemma:Laguerre}. Plugging (\ref{eqn:tau2tau_deriv_1}) and (\ref{eqn:tau2tau_deriv_2}) into the corresponding integral of (\ref{eqn:tau2tau_deriv}) yields 
        \begin{equation}\label{eqn:tau2tau_bound}
            \sum_{k=0}^{\infty} \frac{1}{2^k k!} \int_{\tau \leq |t| \leq 2\tau} \left|\hat{f}_1^{(k)}(t) - \hat{f}_0^{(k)}(t)\right|^2 \,dt \lesssim (1-2\varepsilon)^2 \varepsilon^2 e^{-L\tau^2}
        \end{equation}
        where we have used \(\mu = \frac{c_0\lambda}{\tau} \lesssim \varepsilon\). Combining (\ref{eqn:tau2tau_bound}) and (\ref{eqn:2tau_bound}) and taking \(B_c\) sufficiently large depending only on \(c\), we have  
        \begin{align*}
            \chi^2\left(f_1 \,||\, f_0\right) &\leq C (1-2\varepsilon)^2 \varepsilon^2 e^{-L\tau^2} \leq C (1-2\varepsilon)^2 \varepsilon^2 e^{-LB_c^2 \log\left(en\varepsilon^2(1-2\varepsilon)^2\right)} \leq \frac{c}{n}. 
        \end{align*}
        Note the choice of \(B_c\) depends only on \(c\) since \(C\) and \(L\) are some universal constants. The proof is complete.
    \end{proof}

    \subsection{Proof of Theorem \ref{thm:inconsistent_lbound}}
    In this section, we give a proof of Theorem \ref{thm:inconsistent_lbound}. Unlike in the proof of Theorem \ref{thm:lower_bound_cf}, we do not consider the related mixture formulation. Passing to the mixture formulation is only useful when \(k \leq \frac{n}{2} - C\sqrt{n}\) since the number of outliers behaves like \(k \pm O\left(\sqrt{n}\right)\) in the mixture formulation when \(k \asymp n\). Instead, we work directly with the model (\ref{model:additive}). The lower bound construction is quite similar to that found in \cite{kotekal_sparsity_2023}. However, since the context of \cite{kotekal_sparsity_2023} is quite different, we give a full proof for completeness and for the reader's clarity.   

    \begin{proof}[Proof of Theorem \ref{thm:inconsistent_lbound}]
        Recall we take \(\sigma^2 = 1\) without loss of generality. For ease of notation, set \(\psi^2 = \log\left(1+\frac{n}{(n-2k)^2}\right)\). For further ease, let us assume without loss of generality \(n/2\) is an integer. Define the sets \(E_0 := \left\{1,...,\frac{n}{2}\right\}\) and \(E_1 := \left\{\frac{n}{2}+1,...,n\right\}\). Let \(C > 0\) be a suitably small constant. Let \(\pi_0\) denote the prior in which \(\theta\) is equal to \(C\psi\) and \(\gamma = -2C\psi \mathbf{1}_{S_0}\) where \(S_0 \subset E_0\) is a subset of size \(k\) drawn uniformly at random. Likewise, let \(\pi_1\) denote the prior in which \(\theta = -C\psi\) and \(\gamma = 2C\psi \mathbf{1}_{S_1}\) where \(S_1 \subset E_1\) is a subset of size \(k\) drawn uniformly at random. Observe that if \((\theta_0, \gamma_0) \sim \pi_0\) and \((\theta_1, \gamma_1) \sim \pi_1\), then \(|\theta_0 - \theta_1| = 2C\psi\). With this separation in hand, an invocation of Proposition \ref{prop:fuzzy_hypotheses} yields 
        \begin{equation*}
            \inf_{\hat{\theta}} \sup_{\substack{\theta \in \R, \\ ||\gamma||_0 \leq k}} P_{\theta, \gamma}\left\{ |\theta - \hat{\theta}| > C \psi \right\} \geq \frac{1}{2}(1-\dTV(P_{\pi_0}, P_{\pi_1}))
        \end{equation*}
        where \(P_{\pi_j} = \int P_{\theta, \gamma} \pi_j(d\theta, d\gamma)\) for \(j = 0, 1\) denotes the corresponding mixture measures. It remains to show the total variation is small to obtain the desired result. 

        Consider the random vector \(X \in \R^n\) drawn from the mixture induced from the prior \(X \sim P_{\pi_j}\) for \(j =0, 1\). Let us write \(X_{E_0}, X_{E_1} \in \R^{n/2}\) to denote the random vectors of dimension \(n/2\) obtained by taking the coordinates in \(E_0\) and \(E_1\) respectively. Note \(X_{E_0}\) and \(X_{E_1}\) are independent due to the construction of the priors. Let us write \(P^{I}_{\pi_j}\) and \(P^{II}_{\pi_j}\) to denote the marginal distributions of \(X_{E_0}\) and \(X_{E_1}\) respectively. By the independence, we have 
        \begin{equation}\label{eqn:dtv_triangle_inequality}
            \dTV(P_{\pi_0}, P_{\pi_1}) \leq \dTV(P^{I}_{\pi_0}, P_{\pi_1}^{I}) + \dTV(P^{II}_{\pi_0}, P^{II}_{\pi_1}) = 2\dTV(P^{I}_{\pi_0}, P_{\pi_1}^{I}).
        \end{equation}
        The equality follows from the fact the marginal distributions \(P^{I}_{\pi_0}\) and \(P^{II}_{\pi_1}\), likewise \(P^{II}_{\pi_0}\) and \(P^{I}_{\pi_1}\), are the same up to a deterministic sign flip; hence the two total variation terms are equal in the above display. 

        It remains to bound \(\dTV(P^{I}_{\pi_0}, P_{\pi_1}^{I})\). For ease of notation, set \(Y = X_{E_0}\), \(Q_0 = P^{I}_{\pi_0}\), and \(Q_1 = P^{I}_{\pi_1}\). Observe that 
        \begin{align*}
            Q_0 &= \int N(-C\psi \mathbf{1}_{S_0} + C\psi\mathbf{1}_{S_0^c}, I_{n/2}) \pi_0(dS_0)\\
            Q_1 &= N(-C\psi \mathbf{1}_{\frac{n}{2}}, I_{n/2}). 
        \end{align*}
        By Neyman-Pearson lemma, the quantity \(1 - \dTV(P^{I}_{\pi_0}, P_{\pi_1}^{I})\) is the optimal Type I plus Type II error for the testing problem 
        \begin{align*}
            H_0 &: Y \sim Q_0, \\
            H_1 &: Y \sim Q_1. 
        \end{align*} 
        Equivalently, one can shift the data and consider the problem 
        \begin{align*}
            H_0 &: \tilde{Y} \sim \tilde{Q}_0, \\
            H_1 &: \tilde{Y} \sim \tilde{Q}_1
        \end{align*}
        where \(\tilde{Y} = Y + C\psi \mathbf{1}_{\frac{n}{2}}\) and \(\tilde{Q}_0, \tilde{Q}_1\) are the corresponding distributions of \(\tilde{Y}\) for when \(Y\sim Q_0\) or \(Y\sim Q_1\) respectively. Consider that 
        \begin{align*}
            \tilde{Q}_0 &= \int N(2C\psi\mathbf{1}_{S_0^c}, I_{n/2}) \pi_0(dS_0)\\
            \tilde{Q}_1 &= N(0, I_{n/2}). 
        \end{align*}
        In summary, \(\dTV(P^{I}_{\pi_0}, P_{\pi_1}^{I}) = \dTV(\tilde{Q}_0, \tilde{Q}_1) \leq \frac{1}{2}\sqrt{\chi^2(\tilde{Q}_0\,||\,\tilde{Q}_1)}\). By the Ingster-Suslina method, we have 
        \begin{equation}\label{eqn:ingster_suslina_chisquare}
            \chi^2(\tilde{Q}_0\,||\,\tilde{Q}_1) = E\left( \exp\left( \langle \mu, \mu'\rangle  \right) \right) - 1
        \end{equation}
        where \(\mu = 2C\psi \mathbf{1}_{S_0^c}\) and \(\mu' = 2C\psi \mathbf{1}_{(S_0')^c}\) for \(S_0, S_0'\) independent and uniformly drawn from the collection of size \(k\) subsets of \(E_0\). For ease of notation, denote \(T_0 = S_0^c\) and \(T_0' = (S_0')^c\). Noting that \(|T_0 \cap T_0'|\) is a hypergeometric random variable, Lemma \ref{lemma:hypergeometric} yields  
        \begin{align*}
            E\left( \exp\left( \langle \mu, \mu'\rangle  \right) \right) &= E\left( \exp\left( 4C^2\psi^2 |T_0 \cap T_0'|\right)\right) \\
            &\leq \left(1 - \frac{n/2-k}{n/2} + \frac{n/2-k}{n/2} e^{4C^2\psi^2}\right)^{n/2-k} \\
            &= \left(1 - \frac{n/2-k}{n/2} + \frac{n/2-k}{n/2} \left(1 + \frac{n}{(n-2k)^2}\right)^{4C^2}\right)^{n/2-k} \\
            &\leq \left(1 + \frac{2C^2}{n/2-k}\right)^{n/2-k} \\
            &\leq e^{2C^2} 
        \end{align*}
        where we have used \(2C^2 \leq 1\) (since \(C\) is taken sufficiently small) along with the inequality \((1+x)^\delta \leq 1+\delta x\) for \(0\leq \delta \leq 1\) and \(x \geq 0\). Therefore, from (\ref{eqn:ingster_suslina_chisquare}) we have \(\chi^2(\tilde{Q}_0\,||\,\tilde{Q}_1) \leq e^{2C^2} - 1\), and so we can conclude \(\dTV(P_{\pi_0}^{I}, P_{\pi_1}^{I}) \leq \frac{1}{2}\sqrt{e^{2C^2} - 1}\). From (\ref{eqn:dtv_triangle_inequality}) it thus follows 
        \begin{equation*}
            \inf_{\hat{\theta}} \sup_{\substack{\theta \in \R, \\ ||\gamma||_0 \leq k}} P_{\theta, \gamma}\left\{ |\theta - \hat{\theta}| > C \psi \right\} \geq \frac{1}{2}(1-\sqrt{e^{2C^2} - 1}).
        \end{equation*}
        Taking \(C\) sufficiently small, the right hand side can be bounded below by some positive universal constant \(c\). The proof is complete. 
    \end{proof}

    \section{\texorpdfstring{Adaptation to \(k\)}{Adaptation to k}}
    Recall the setup of Section \ref{section:adapt_k}. The notation of Section \ref{section:adapt_k} is freely used in the following. 

    \begin{proof}[Proof of Theorem \ref{thm:adapt_k}]
        Define the event  
        \begin{equation*}
            \mathcal{G} := \left\{L_\eta^{-1} \leq \frac{\sigma^2}{\tilde{\sigma}^2} \leq L_\eta \right\} \cap \bigcap_{\substack{k \in \mathcal{K}, \\ k \geq k^*}} \left\{ \frac{|\hat{\theta}_k - \theta|}{\sigma} \leq C'_{\delta, \eta} \epsilon(k, n)\right\}
        \end{equation*}
        Consider that 
        \begin{align}
            &\sup_{\substack{\theta \in \R, \\ ||\gamma||_0 \leq k^*, \\ \sigma > 0}} P_{\theta, \gamma, \sigma}\left\{ \frac{|\hat{\theta} - \theta|}{\sigma} > 2C'_{\delta, \eta}L_\eta \epsilon(k^*, n)\right\} \nonumber \\
            &\leq \sup_{\substack{\theta \in \R, \\ ||\gamma||_0 \leq k^*, \\ \sigma > 0}} P_{\theta, \gamma, \sigma}\left(\left\{ \frac{|\hat{\theta} - \theta|}{\sigma} > 2C'_{\delta, \eta}L_\eta \epsilon(k^*, n)\right\} \cap \mathcal{G} \right) + \sup_{\substack{\theta \in \R, \\ ||\gamma||_0 \leq k^*, \\ \sigma > 0}} P_{\theta, \gamma, \sigma}(\mathcal{G}^c). \label{eqn:adapt_master_bound} 
        \end{align}
        Consider that by Corollary \ref{corollary:location_tail} we have \(P_{\theta, \gamma, \sigma}(\mathcal{G}^c) \leq \eta + \delta\). It remains to bound the first term of (\ref{eqn:adapt_master_bound}). Consider that 
        \begin{equation*}
            \mathcal{G} \subset \left\{ L_\eta^{-1} \leq \frac{\sigma^2}{\tilde{\sigma}^2} \leq L_\eta\right\} \cap \bigcap_{\substack{k \in \mathcal{K}, \\ k \geq k^*}} \left\{ |\hat{\theta}_k - \theta| \leq \tilde{\sigma} C'_{\delta, \eta}\sqrt{L_\eta}\epsilon(k, n) \right\}.
        \end{equation*}
        Consequently, on \(\mathcal{G}\) it follows from the definition of \(k'\) in (\ref{eqn:Lepski}) that \(k^* \geq k'\). Importantly, we thus have \(\hat{\theta} \in J_{k^*}\), and so it immediately follows by triangle inequality 
        \begin{equation*}
            \mathcal{G} \subset \left\{ L_\eta^{-1} \leq \frac{\sigma^2}{\tilde{\sigma}^2} \leq L_\eta\right\} \cap \left\{ |\hat{\theta} - \theta| \leq 2\tilde{\sigma} C'_{\delta, \eta}\sqrt{L_\eta}\epsilon(k^*, n)\right\} \subset \left\{|\hat{\theta} - \theta| \leq 2\sigma C'_{\delta, \eta} L_\eta \epsilon(k^*, n)\right\}. 
        \end{equation*}
        Therefore, the first term of (\ref{eqn:adapt_master_bound}) is zero. We have the desired result by taking \(C_{\delta, \eta}\) sufficiently large. 
    \end{proof}
    \begin{corollary}\label{corollary:location_tail}
        Fix \(\delta \in (0, 1)\). There exist constants \(C', \tilde{C}, c > 0\) depending only on \(\delta\) such that the following holds. If \(n\) is sufficiently large depending only on \(\delta\) and \(1 \leq k^* < \frac{n}{2} - \tilde{C}\sqrt{n}\), then
        \begin{equation*}
            \sup_{\substack{\theta \in \R, \\ ||\gamma||_0 \leq k^*, \\ \sigma > 0}} P_{\theta, \gamma, \sigma}\left(\bigcup_{k^* \leq k < \frac{n}{2} - \tilde{C}\sqrt{n}} \left\{ \frac{|\hat{\theta}_k - \theta|}{\sigma} > \frac{C'k}{n\sqrt{\log\left(1 + \frac{k^2(n-2k)^2}{n^3}\right)}}\right\} \right) \leq \delta
        \end{equation*}
        where \(\hat{\theta}_k\) is given by (\ref{estimator:theta_unknown_var}). 
    \end{corollary}
    \begin{proof}
        The proof is largely the same as the proof of Theorem \ref{thm:theta_unknown_var}. Let \(\mathcal{E}\) be the event from the proof of Theorem \ref{thm:theta_unknown_var} with the choice \(L\) depending on \(\delta\) such that for any \(\theta \in \R\), \(||\gamma||_0 \leq k^*,\) and \(\sigma > 0\), we have \(P_{\theta, \gamma, \sigma}(\mathcal{E}^c) \leq \delta/2\). For each \(k\), let \(\mathcal{E}_{\text{var}, k} = \left\{\sigma^2 \in [\sigma^2_{-, k}, \sigma^2_{+, k}]\right\}\) where \(\sigma^2_{-, k}, \sigma^2_{+, k}\) are the quantities associated to \(\hat{\theta}_k\) in (\ref{estimator:theta_unknown_var}). We claim 
        \begin{equation*}
            P_{\theta, \gamma, \sigma} \left( \bigcap_{k^* \leq k < \frac{n}{2}} \mathcal{E}_{\text{var}, k} \right) \geq 1 - \frac{\delta}{2}
        \end{equation*}
        uniformly over \(\theta \in \R\), \(||\gamma||_0 \leq k^*\), and \(\sigma > 0\). As argued in Theorem \ref{thm:var}, we have \(\mathcal{E}_{\text{var}, k} \supset \tilde{\mathcal{E}}\) where \(\tilde{\mathcal{E}}\) is the event \(\mathcal{E} \cap \mathcal{E}'\) in the proof of Theorem \ref{thm:var} (except now at confidence level \(\delta/2\) instead of \(\delta\)). Importantly, \(\tilde{\mathcal{E}}\) does not depend on \(k\) and we have \(P_{\theta, \gamma, \sigma} (\tilde{\mathcal{E}}) \geq 1-\delta/2\) uniformly over \(\theta \in \R\), \(||\gamma||_0 \leq k^*\), and \(\sigma > 0\). Thus, the claim is proved. Then, as argued in Theorem \ref{thm:theta_unknown_var}, we have 
        \begin{equation*}
            \mathcal{E} \cap \bigcap_{k^* \leq k < \frac{n}{2}} \mathcal{E}_{\text{var},k} \subset  \bigcap_{k^* \leq k < \frac{n}{2} - \tilde{C}\sqrt{n}} \left\{ \frac{|\hat{\theta}_k - \theta|}{\sigma} \leq \frac{C'k}{n\sqrt{\log\left(1 + \frac{k^2(n-2k)^2}{n^3}\right)}}\right\}
        \end{equation*}
        since \(\tilde{C}\) is sufficiently large. Thus, we have the desired result since the probability of the event on the left hand side is at least \(1 - \delta\) by union bound. 
    \end{proof}
    
    \noindent The following results pertain to the construction of the variance estimator for the purposes of adaptive null estimation in Section \ref{section:TV}. Recall the setup of Section \ref{section:TV} as we will use that notation freely in what follows. 
    
    \begin{proof}[Proof of Theorem \ref{thm:var_adapt_k}]
        The proof proceeds in the same manner as the proof of Theorem \ref{thm:adapt_k}. As in that argument, define the event 
        \begin{equation*}
            \mathcal{G} := \left\{L_\eta^{-1} \leq \frac{\sigma^2}{\tilde{\sigma}^2} \leq L_\eta\right\} \cap \bigcap_{k^* \leq k < \frac{n}{2}} \left\{ \frac{|\hat{\sigma}^2_k - \sigma^2|}{\sigma^2} \leq \frac{C_{\delta, \eta}' k}{n\log\left(1 + \frac{k}{\sqrt{n}}\right)}\right\}.
        \end{equation*}
        Analogous to (\ref{eqn:adapt_master_bound}), we have
        \begin{align}
            &\sup_{\substack{\theta \in \R, \\ ||\gamma||_0 \leq k^*, \\ \sigma > 0}} P_{\theta, \gamma, \sigma}\left\{ \frac{|\hat{\sigma}^2 - \sigma^2|}{\sigma^2} > \frac{2C'_{\delta, \eta} L_\eta^2 k^*}{n\log\left(1 + \frac{k^*}{\sqrt{n}}\right)}  \right\} \nonumber \\
            &\leq \sup_{\substack{\theta \in \R, \\ ||\gamma||_0 \leq k^*, \\ \sigma > 0}} P_{\theta, \gamma, \sigma}\left(\left\{ \frac{|\hat{\sigma}^2 - \sigma^2|}{\sigma^2} > \frac{2C'_{\delta, \eta} L_\eta^2  k^*}{n\log\left(1 + \frac{k^*}{\sqrt{n}}\right)} \right\} \cap \mathcal{G}\right) + \sup_{\substack{\theta \in \R, \\ ||\gamma||_0 \leq k^*, \\ \sigma > 0}} P_{\theta, \gamma, \sigma}\left(\mathcal{G}^c\right). \label{eqn:var_adapt_master_bound}
        \end{align}
        Consider that by Corollary \ref{corollary:var_tail} we have \(P_{\theta, \gamma, \sigma}(\mathcal{G}^c) \leq \eta + \delta\). To bound the first term in (\ref{eqn:var_adapt_master_bound}), consider that
        \begin{equation*}
            \mathcal{G} \subset \left\{ L_\eta^{-1} \leq \frac{\sigma^2}{\tilde{\sigma}^2} \leq L_\eta\right\} \cap \bigcap_{k^* \leq k < \frac{n}{2}} \left\{|\hat{\sigma}^2_k - \sigma^2| \leq \tilde{\sigma}^2\frac{C'_{\delta, \eta} L_\eta k }{n\log\left(1 + \frac{k}{\sqrt{n}}\right)}\right\}. 
        \end{equation*}
        Consequently, on \(\mathcal{G}\) it follows from the definition of \(k'\) in (\ref{eqn:Lepski_var}) that \(k^* \geq k'\) and so \(\hat{\sigma}^2 \in J_{k^*}\). It immediately follows by triangle inequality that
        \begin{equation*}
            \mathcal{G} \subset \left\{\frac{|\hat{\sigma}^2 - \sigma^2|}{\sigma^2} \leq \frac{2C'_{\delta, \eta} L_\eta^2 k^*}{n\log\left(1 + \frac{k^*}{\sqrt{n}}\right)}\right\}
        \end{equation*}
        and so the first term of (\ref{eqn:var_adapt_master_bound}) is zero. Thus, we have proved the desired result.
    \end{proof}
    
    \begin{corollary}\label{corollary:var_tail}
        Fix \(\delta \in (0, 1)\). There exists \(C > 0\) depending only on \(\delta\) such that the following holds. If \(n\) is sufficiently large depending only on \(\delta\) and \(1 \leq k^* < \frac{n}{2}\), then
        \begin{equation*}
            \sup_{\substack{\theta \in \R, \\ ||\gamma||_0 \leq k^*, \\ \sigma > 0}} P_{\theta, \gamma, \sigma}\left(\bigcup_{k^* \leq k < \frac{n}{2}}\left\{ \frac{|\hat{\sigma}^2_k - \sigma^2|}{\sigma^2} > \frac{C k}{n\log\left(1 + \frac{k}{\sqrt{n}}\right)} \right\}\right) \leq \delta
        \end{equation*}
        where \(\hat{\sigma}^2_k\) is the variance estimator from Theorem \ref{thm:var}.
    \end{corollary}
    \begin{proof}
        The proof is the same as the proof of Theorem \ref{thm:var}, and the argument's structure is the same as that of the proof of Corollary \ref{corollary:location_tail}. In particular, let \(\mathcal{E} \cap \mathcal{E}'\) be the event from the proof of Theorem \ref{thm:var}, having made choices \(L, L'\) depending only on \(\delta\) such that for any \(\theta \in \R\), \(||\gamma||_0 \leq k^*,\) and \(\sigma > 0\), we have \(P_{\theta, \gamma, \sigma}(\mathcal{E}^c \cup \mathcal{E}'^c) \leq \delta\). The arguments of the proof of Theorem \ref{thm:var} show 
        \begin{equation*}
            \mathcal{E} \cap \mathcal{E}' \subset \left\{ \frac{|\hat{\sigma}^2_k - \sigma^2|}{\sigma^2} \leq \frac{C k}{n\log\left(1 + \frac{k}{\sqrt{n}}\right)} \right\}
        \end{equation*}
        for \(k \geq k^*\). Since \(\mathcal{E} \cap \mathcal{E}'\) does not depend on \(k\), it follows 
        \begin{equation*}
            \mathcal{E} \cap \mathcal{E}' \subset \bigcap_{k^* \leq k < \frac{n}{2}} \left\{ \frac{|\hat{\sigma}^2_k - \sigma^2|}{\sigma^2} \leq \frac{C k}{n\log\left(1 + \frac{k}{\sqrt{n}}\right)} \right\},
        \end{equation*}
        and so we have the desired result. 
    \end{proof}

    \section{Auxiliary results}
    
    \begin{proof}[Proof of Proposition \ref{prop:cosines}]
        Let \(p(x) = \lambda e^{-\lambda x}\mathbbm{1}_{\{x > 0\}}\) denote the \(\text{Exponential}(\lambda)\) distribution with \(\lambda = \frac{1}{50\alpha}\). Note this distribution has mean \(\frac{1}{\lambda}\) and variance \(\frac{1}{\lambda^2}\). Let \(W \sim p\) and consider from Lemma \ref{lemma:exponential} 
        \begin{align*}
            E(f(W)) = \frac{1}{k} \sum_{j=1}^{k} \int \cos(\omega \gamma_j) p(\omega) d\omega = \frac{1}{k} \sum_{j=1}^{k} \frac{\lambda^2}{\lambda^2 + \gamma_j^2} \geq 0. 
        \end{align*}
        For ease of notation, denote \(\beta = 100\alpha\). Let \(\bar{p}\) denote the conditional density of \(W\) conditional on \(W \in [\alpha, \beta]\). Note \(\bar{p}\) is supported on \([\alpha, \beta]\). Then 
        \begin{align*}
            \max_{\omega \in [\alpha, \beta]} f(\omega) &\geq \int f(\omega) \bar{p}(\omega) d\omega \\
            &= \frac{E(f(W)\mathbbm{1}_{\{W \in [\alpha, \beta]\}})}{P\{W \in [\alpha, \beta]\}} \\
            &= \frac{E(f(W)) - E(f(W) \mathbbm{1}_{\{W \in [\alpha, \beta]^c\}})}{P\{W \in [\alpha, \beta]\}} \\
            &\geq -\frac{P\{W \in [\alpha, \beta]^c\}}{P\{W \in [\alpha, \beta]\}}.
        \end{align*}
        Consider that 
        \begin{equation*}
            P\left\{ W \in [\alpha, \beta]^c \right\} = 1 - e^{-\lambda \alpha} + e^{-\lambda \beta} = 1 - e^{-\frac{1}{50}} + e^{-2} \leq 0.16
        \end{equation*}
        Therefore, 
        \begin{equation*}
            \max_{\omega \in [\alpha, \beta]} f(\omega) \geq - \frac{P\{W \in [\alpha, \beta]^c\}}{P\{W \in [\alpha, \beta]\}} \geq - \frac{0.16}{1-0.16} \geq -\frac{1}{5}. 
        \end{equation*}
        The proof is complete.
    \end{proof}

    \begin{lemma}[Lemma 1 \cite{laurent_adaptive_2000}]\label{lemma:chisquare_tail}
        For any positive integer \(d\) and \(t > 0\), we have 
        \begin{equation*}
            P\left\{\chi^2_d \geq d + 2\sqrt{dt} + 2t\right\} \leq e^{-t}. 
        \end{equation*}
    \end{lemma}
    
    \begin{lemma}[Lemma 11.1 \cite{verzelen_minimax_2012}]\label{lemma:chisquare_lower_tail}
        For any positive integer \(d\) and \(t > 0\), we have 
        \begin{equation*}
            P\left\{\chi^2_d \leq e^{-1} dt^{2/d}\right\} \leq t. 
        \end{equation*}
    \end{lemma}

    \begin{proposition}[Method of two fuzzy hypotheses \cite{tsybakov_introduction_2009}]\label{prop:fuzzy_hypotheses}
        Suppose \(\mathcal{P}\) is a collection of distributions on a sample space \(\mathcal{X}\) indexed by \(\Theta\) and \((\Upsilon, \rho)\) is a metric space. Let \(\tau : \Theta \to \Upsilon\) be a function, \(\phi : \R_+ \to \R_+\) be a non-decreasing function with \(\phi(0) = 0\). If \(\Theta_0, \Theta_1 \subset \Theta\) and \(\pi_0, \pi_1\) are two priors supported on \(\Theta_0, \Theta_1\) respectively, then 
        \begin{equation*}
            \inf_{\hat{\tau}}\sup_{\theta \in \Theta} P_\theta\left\{ \phi\left(\rho(\hat{\tau}(X), \tau(\theta))\right) > \phi(\delta) \right\} \geq \frac{1}{2}\left(1 - \dTV(P_{\pi_0}, P_{\pi_1})\right)
        \end{equation*}
        where \(2\delta := \inf_{\substack{\theta_0 \in \Theta_0 \\ \theta_1 \in \Theta_1}} \rho(\tau(\theta_0), \tau(\theta_1))\) and \(P_{\pi_j} = \int P_\theta \, \pi_j\) for \(j = 0, 1\). 
    \end{proposition}

    \begin{proposition}[Ingster-Suslina method \cite{ingster_nonparametric_2003}]\label{prop:ingster_suslina}
        Suppose \(\Sigma \in \R^{n\times n}\) is a positive definite matrix and \(\Theta \subset \R^n\) is a parameter space. Let \(P_\theta\) denote the distribution \(N(\theta, \Sigma)\). If \(\pi\) is a probability distribution supported on \(\Theta\), then 
        \begin{equation*}
            \chi^2(P_\pi || P_0) = E\left( \exp\left( \langle \theta, \Sigma^{-1} \tilde{\theta} \rangle \right)\right) - 1
        \end{equation*}
        where \(\theta, \tilde{\theta} \overset{iid}{\sim} \pi\). Here, \(P_\pi = \int_{\theta} P_\theta \pi(d\theta)\) and \(\chi^2(\cdot || \cdot)\) denotes the \(\chi^2\)-divergence. 
    \end{proposition}

    \begin{lemma}[\cite{collier_minimax_2017}]\label{lemma:hypergeometric}
        Suppose \(1 \leq k \leq n\). If \(Y\) is distributed according to the hypergeometric distribution with probability mass function \(P\{Y = \ell\} = \frac{\binom{k}{\ell} \binom{n-k}{k-\ell}}{\binom{n}{k}}\) for \(0 \leq \ell \leq k\), then  
        \begin{equation*}
            E(e^{\lambda Y}) \leq \left(1 - \frac{k}{n} + \frac{k}{n} e^{\lambda}\right)^k
        \end{equation*}
        for \(\lambda > 0\). 
    \end{lemma}

    \begin{lemma}\label{lemma:exponential}
        For \(\lambda > 0\), define the probability density function \(f(x) = \lambda e^{-\lambda x}\mathbbm{1}_{\{x \geq 0\}}\) for the \(\text{Exponential}(\lambda)\) distribution. If \(X \sim \text{Exponential}(\lambda)\), then 
        \begin{equation*}
            E(\cos(tX)) = \frac{\lambda^2}{\lambda^2 + t^2}. 
        \end{equation*}
    \end{lemma}
    \begin{proof}
        Consider that \(E(\cos(tX))\) is the real part of the characteristic function \(E(e^{itX}) = \frac{\lambda}{\lambda - it}\). Calculating directly, 
        \begin{align*}
            \frac{\lambda}{\lambda-it} = \frac{\lambda}{\lambda + \frac{t^2}{\lambda}} + \frac{it\lambda}{\lambda^2+t^2}.
        \end{align*}
        Therefore, 
        \begin{equation*}
            E(\cos(tX)) = \frac{\lambda}{\lambda + \frac{t^2}{\lambda}} = \frac{\lambda^2}{\lambda^2 + t^2}. 
        \end{equation*}
    \end{proof}

    \begin{lemma}\label{lemma:k}
        If \(\tau > 0\) and \(n\) is a nonnegative integer, then 
        \begin{equation*}
            \int_{0}^{\tau} \frac{t^{2n+1}}{\tau^{2n+1}} \sin(tx) \,dt = - \frac{\cos(\tau x)}{x} + R(x)
        \end{equation*}
        with \(|R(x)| \leq \frac{C(1 \vee n)}{\tau x^2}\) where \(C > 0\) is a universal constant. 
    \end{lemma}
    \begin{proof}
        The proof follows by applying integration by parts twice. Suppose \(n \geq 1\). Consider 
        \begin{align*}
            \int_{0}^{\tau} \frac{t^{2n+1}}{\tau^{2n+1}} \sin(tx) \, dt &= -\frac{\cos(\tau x)}{x} + \int_{0}^{\tau} \frac{(2n+1)t^{2n}}{\tau^{2n+1}} \cdot \frac{\cos(t x)}{x} \, dt \\
            &= -\frac{\cos(\tau x)}{x} + \frac{(2n+1)\sin(\tau x)}{\tau x^2} - \int_{0}^{\tau} \frac{2n(2n+1)t^{2n-1}}{\tau^{2n+1}} \cdot \frac{\sin(tx)}{x^2} \, dt.
        \end{align*}
        Clearly \(\left|\frac{(2n+1)\sin(\tau x)}{\tau x^2}\right| \leq \frac{Cn}{\tau x^2}\). Likewise, consider 
        \begin{align*}
            \left|\int_{0}^{\tau} \frac{2n(2n+1)t^{2n-1}}{\tau^{2n+1}} \cdot \frac{\sin(tx)}{x^2} \, dt\right| \leq \frac{2n(2n+1)}{x^2} \int_{0}^{\tau} \frac{t^{2n-1}}{\tau^{2n+1}} \, dt \leq \frac{Cn}{\tau x^2}
        \end{align*}
        and so we have the desired result. If \(n = 0\), the same calculation gives \(\int_{0}^{\tau} \frac{t}{\tau} \sin(tx) \, dt = -\frac{\cos(\tau x)}{x} + \frac{\sin(\tau x)}{\tau x^2}\), which yields the desired result. 
    \end{proof}

    \begin{lemma}\label{lemma:j}
        If \(\tau > 0\), then  
        \begin{equation*}
            \int_{\tau}^{2\tau} \frac{2\tau - t}{\tau} \sin(tx) \, dt = \frac{\cos(\tau x)}{x} + R(x)
        \end{equation*}
        where \(|R(x)| \leq \frac{2}{\tau x^2}\). 
    \end{lemma}
    \begin{proof}
        The proof is a simple application of integration by parts. Consider 
        \begin{align*}
            \int_{\tau}^{2\tau} \frac{2\tau - t}{\tau} \sin(tx) \, dt = \frac{\cos(\tau x)}{x} - \int_{\tau}^{2\tau} \frac{\cos(tx)}{\tau x} \, dt = \frac{\cos(\tau x)}{x} - \left(\frac{\sin(2\tau x) - \sin(\tau x)}{\tau x^2}\right).
        \end{align*}
        Clearly \(\left|\frac{\sin(2\tau x) - \sin(\tau x)}{\tau x^2}\right| \leq \frac{2}{\tau x^2}\) and so we have the desired result.
    \end{proof}

    \begin{theorem}[Cramer's inequality \cite{indritz_inequality_1961}]
        Define the \(n\)th Hermite polynomial \(He_n(x) = (-1)^n e^{\frac{x^2}{2}} \cdot \frac{d^n}{dx^n} e^{-\frac{x^2}{2}}\). Further define the Hermite function 
        \begin{equation*}
            \psi_n(x) = (2^n n! \sqrt{\pi})^{-1/2} e^{-\frac{x^2}{2}}2^{n/2}He_n(\sqrt{2}x).
        \end{equation*}
        Then \(\sup_{x \in \R} |\psi_n(x)| \leq \pi^{-1/4}\). 
    \end{theorem}
    \begin{corollary}\label{corollary:cramer}
        Let \(h(x) = e^{-\frac{x^2}{2}}\). Then \(|h^{(n)}(x)|\leq e^{-\frac{x^2}{4}}\sqrt{n!}\) for all \(x \in \R\). 
    \end{corollary}

    \begin{lemma}\label{lemma:Laguerre}
        We have 
        \begin{equation*}
            \sum_{k=0}^{\infty} \frac{1}{2^k k!} \left(\sum_{\ell=0}^{k} \binom{k}{\ell} \sqrt{\ell!}\right)^2 < \infty. 
        \end{equation*}
        Further, for any fixed positive integer \(s\), we have 
        \begin{equation*}
            \sum_{k=0}^{\infty} \frac{1\vee k^s}{2^k k!} \left(\sum_{\ell=0}^{k} \binom{k}{\ell} \sqrt{\ell!}\right)^2 < \infty.
        \end{equation*}
    \end{lemma}
    \begin{proof}
        Expanding the square gives us 
        \begin{align*}
            \frac{1}{2^k k!}\left(\sum_{\ell=0}^{k} \binom{k}{\ell} \sqrt{\ell!}\right)^2 &= \frac{1}{2^k k!}\sum_{\ell=0}^{k} \binom{k}{\ell}^2 \ell! + \frac{1}{2^k k!} \sum_{0 \leq \ell \neq j \leq k} \binom{k}{\ell} \binom{k}{j} \sqrt{\ell! j!}. 
        \end{align*}
        Examining the first term, consider that 
        \begin{align*}
            \frac{1}{2^k k!} \sum_{\ell=0}^{k} \binom{k}{\ell}^2 \ell! = \frac{1}{2^k} \sum_{\ell=0}^{k} \binom{k}{\ell} \frac{1}{(k-\ell)!} = \frac{1}{2^k} \sum_{\ell=0}^{k} \binom{k}{\ell} \frac{1}{\ell!} = \frac{1}{2^k} L_k(-1) 
        \end{align*}
        where \(L_k\) is the \(k\)th Laguerre polynomial. By the generating function \(\frac{1}{1-t} e^{-\frac{tx}{1-t}} = \sum_{n=0}^{\infty} t^n L_n(x)\), we have     
        \begin{equation*}
            \sum_{k=0}^{\infty} \frac{1}{2^k} \sum_{\ell=0}^{k} \binom{k}{\ell} \frac{1}{\ell!} = \sum_{k=0}^{\infty} \frac{1}{2^k} L_k(-1) = 2e. 
        \end{equation*}
        Thus the first term is handled and it remains to bound the cross term. Consider that for any \(x, y \in \R\), we have \(xy \leq x^2 \vee y^2\). Therefore, 
        \begin{align*}
            \frac{1}{2^k k!} \sum_{0 \leq \ell \neq j \leq k} \binom{k}{\ell} \binom{k}{j} \sqrt{\ell! j!} \leq \frac{k}{2^k k!} \sum_{\ell=0}^{k} \binom{k}{\ell}^2 \ell! = \frac{k}{2^k} L_k(-1) \leq \left(\frac{3}{4}\right)^k L_k(-1) 
        \end{align*}
        where we have used \(k \leq \left(\frac{3}{2}\right)^k\) for all \(k \geq 0\). Employing the generating function as before, we have 
        \begin{equation*}
            \sum_{k=0}^{\infty} \frac{1}{2^k k!} \sum_{0 \leq \ell \neq j \leq k} \binom{k}{\ell} \binom{k}{j} \sqrt{\ell! j!} \leq \sum_{k=0}^{\infty} \left(\frac{3}{4}\right)^k L_k(-1) = 4e^{3}. 
        \end{equation*}
        Thus, we have \(\sum_{k=0}^{\infty} \frac{1}{2^k k!} \left(\sum_{\ell=0}^{k} \binom{k}{\ell} \sqrt{\ell!}\right)^2 \leq 2e + 4e^3 < \infty\) as desired. The second claim follows by a similar argument. 
    \end{proof}

    \begin{lemma}[\cite{tsybakov_introduction_2009}]
        For probability measures \(P_1,...,P_n, Q_1,...,Q_n\), we have 
        \begin{equation*}
            \chi^2\left(\left.\left. \bigotimes_{i=1}^{n} Q_i \,\right|\right|\, \bigotimes_{i=1}^{n} P_i\right) = \left(\prod_{i=1}^{n} (1 + \chi^2(Q_i \,||\, P_i))\right) - 1. 
        \end{equation*}
    \end{lemma}

    \begin{theorem}[Bounded differences - Theorem 6.2 \cite{boucheron_concentration_2013}]\label{thm:bounded_differences}
        Suppose \(f : \mathcal{X}^n \to \R\) satisfies the bounded differences inequality for some nonnegative \(d_1,...,d_n\), that is 
        \begin{equation*}
            \sup_{\substack{x_1,...,x_n \in \mathcal{X}, \\ x_i' \in \mathcal{X}}} |f(x_1,...,x_{i-1},x_i,x_{i+1}...,x_n) - f(x_1,...,x_{i-1},x_{i}',x_{i+1},...,x_n)| \leq d_i
        \end{equation*}
        for all \(1 \leq i \leq n\). Let \(Z = f(X_1,...,X_n)\) where \(X_1,...,X_n\) are independent \(\mathcal{X}\)-valued random variables. If \(u \geq 0\), then 
        \begin{equation*}
            P\left\{|Z - E(Z)| > u\right\} \leq 2 \exp\left(-\frac{2u^2}{\sum_{i=1}^{n} d_i^2}\right).
        \end{equation*}
    \end{theorem}

    \begin{theorem}[Bernstein's inequality - Theorem 2.8.4 \cite{vershynin_high-dimensional_2018}]\label{thm:bernstein_bounded}
        Let \(Y_1,...,Y_n\) be independent mean zero random variables such that \(|Y_i| \leq 1\) for all \(i\). If \(u \geq 0\), then 
        \begin{equation*}
            P\left\{ \left|\sum_{i=1}^{n} Y_i\right| \geq u \right\} \leq 2 \exp\left(-\frac{u^2/2}{\tau^2 + u/3}\right)
        \end{equation*}
        where \(\tau^2 = \sum_{i=1}^{n} E(Y_i^2)\). 
    \end{theorem}

    \section{Location estimation in Huber's contamination model}\label{section:huber}
Recall Huber's contamination model (\ref{model:Huber_deterministic}), which is the same, from the perspective of conclusions about in-probability minimax estimation rates, as the Bayesian version (\ref{model:Huber_mixture}). In \cite{chen_robust_2018}, it has been shown that if \(\varepsilon = \frac{k}{n} < \frac{1}{5}\), then the sample median achieves the optimal rate
\begin{equation*}
    \left|\median\left(X_1,...,X_n\right) - \theta\right|^2 \lesssim \sigma^2\left(\frac{1}{n} + \varepsilon^2\right) \asymp \sigma^2 \left(\frac{1}{n} + \frac{k^2}{n^2}\right). 
\end{equation*}
\noindent In Remark \ref{remark:gaussian_character} of Section \ref{section:main_contribution}, we claimed that for \(\frac{n}{5} \leq k < \frac{n}{2}\),
\begin{equation*}
    \left|\median\left(X_1,...,X_n\right) - \theta\right|^2 \lesssim \sigma^2 \log\left( \frac{1}{1-2\varepsilon}\right) \asymp \sigma^2\log\left(\frac{en}{n-2k}\right). 
\end{equation*}
and that this is the optimal rate. In this section, we give proofs for these two assertions. For convenience, the upper bound will be proved in the context of the frequentist model (\ref{model:Huber_deterministic}) while the lower bound will be proved in the context of the Bayes model (\ref{model:Huber_mixture}) using the modulus of continuity concept of \cite{chen_robust_2018,chen_general_2016}. 

The following proposition states the error achieved by sample median for \(\frac{n}{5} \leq k < \frac{n}{2}\). The proof is exactly the same as the proof of \cite{kotekal_sparsity_2023} in this regime, which we reproduce here with notational changes for the reader's convenience. 

\begin{proposition}
    Consider data \(X_1,...,X_n\) from the model (\ref{model:Huber_deterministic}). Suppose \(\frac{n}{5}\leq k < \frac{n}{2}\). If \(\delta \in (0, 1)\) and \(n\) sufficiently large depending only on \(\delta\), then  
    \begin{equation*}
        \sup_{\substack{\theta \in \R, \\ \eta \in \R^n, \\ \sigma > 0}} P_{\theta, \eta, \sigma}\left\{ \left|\median\left(X_1,...,X_n\right) - \theta\right| > C\sigma \sqrt{\log\left(\frac{en}{n-2k}\right)}\right\} \leq \delta 
    \end{equation*}
    where \(C > 0\) is a universal constant. 
\end{proposition}
\begin{proof}
    Without loss of generality we can assume \(\sigma = 1\) otherwise we can simply work with the normalized data \(\left\{X_j/\sigma\right\}_{j=1}^{n}\) because \(\median(X_1/\sigma,...,X_n/\sigma) = \frac{\median(X_1,...,X_n)}{\sigma}\). Let \(\hat{\theta} := \median(X_1,...,X_n)\) We split into two cases. \newline
    
    \noindent \textbf{Case 1:} Suppose \(n -2k < \sqrt{n}\). Since \(|\mathcal{I}| > \frac{n}{2}\), it follows that 
    \begin{align*}
        P_{\theta, \eta, \gamma} \left\{|\hat{\theta} - \theta| > \sqrt{4\log(en)}\right\} \leq P_{\theta, \eta, \gamma}\left\{ \max_{j \in \mathcal{I}} \left|X_j - \theta\right| > \sqrt{4\log(en)}\right\} \leq 2n\exp\left(-2\log(en)\right) \leq \frac{1}{n} < \delta 
    \end{align*}
    for \(n \geq \frac{1}{\delta}\). Since \(\log\left(\frac{en}{n-2k}\right) \asymp \log(en)\) because \(n-2k \leq \sqrt{n}\), we have the claimed result. \newline
    
    \noindent \textbf{Case 2:} Suppose \(n-2k \geq \sqrt{n}\). Define the interval 
    \begin{equation*}
        E := \left[ \theta - \sqrt{2\log\left(\frac{4|\mathcal{I}|}{\frac{n}{2} - |\mathcal{O}|}\right)}, \theta + \sqrt{2\log\left(\frac{4|\mathcal{I}|}{\frac{n}{2} - |\mathcal{O}|}\right)} \right].
    \end{equation*}
    Consider 
    \begin{align*}
        P_{\theta, \eta, \sigma}\left\{ \hat{\theta} \not \in E \right\} &\leq P_{\theta, \eta, \sigma}\left\{ \sum_{j=1}^{n} \mathbbm{1}_{\{X_j \not \in E\}} > \frac{n}{2} \right\} \\
        &\leq P_{\theta, \eta, \sigma}\left\{ \sum_{j \in \mathcal{I}} \mathbbm{1}_{\{X_j \not \in E\}} > \frac{n}{2} - |\mathcal{O}| \right\} \\
        &\leq P_{\theta, \eta, \sigma}\left\{ \sum_{j \in \mathcal{I}} \mathbbm{1}_{\{X_j \not \in E\}} - p > \frac{n}{2} - |\mathcal{O}| - |\mathcal{I}| p \right\}
    \end{align*}
    where \(p = P\left\{|Z| > \sqrt{2\log\left(\frac{4|\mathcal{I}|}{\frac{n}{2} - |\mathcal{O}|}\right)} \right\}\) with \(Z \sim N(0, 1)\). By Bernstein's inequality (Theorem \ref{thm:bernstein_bounded}), we have for a universal constant \(c > 0\)
    \begin{align*}
        P_{\theta, \eta, \sigma}\left\{ \sum_{j \in \mathcal{I}} \mathbbm{1}_{\{X_j \not \in E\}} - p > \frac{n}{2} - |\mathcal{O}| - |\mathcal{I}| p \right\} &\leq \exp\left(-c\min\left(\frac{\left(\frac{n}{2} - |\mathcal{O}| - |\mathcal{I}| p\right)^2}{|\mathcal{I}| p(1-p)}, \frac{n}{2} - |\mathcal{O}| - |\mathcal{I}| p\right)\right).
    \end{align*}
    Consider that \(p \leq 2\exp\left(-\frac{1}{2} \cdot 2\log\left(\frac{4|\mathcal{I}|}{\frac{n}{2} - |\mathcal{O}|}\right)\right) = \frac{\frac{n}{2} - |\mathcal{O}|}{2|\mathcal{I}|}\). Therefore, 
    \begin{align*}
        P_{\theta, \eta, \sigma}\left\{ \hat{\theta} \not \in E \right\} &\leq \exp\left(-\frac{c}{2} \cdot \left(\frac{n}{2}-|\mathcal{O}|\right)\right) \leq \exp\left(-\frac{c(n-2k)}{4}\right) \leq \exp\left(-\frac{c\sqrt{n}}{4}\right) \leq \delta  
    \end{align*}
    for \(n\) sufficiently large depending on \(\delta\). Consider that \(\log\left(\frac{4|\mathcal{I}|}{\frac{n}{2} - |\mathcal{O}|}\right) \lesssim \log\left(\frac{en}{n-2k}\right)\), and so the proof is complete. 
\end{proof}

\noindent For the discussion of the lower bound, let us consider the Bayesian model (\ref{model:Huber_mixture}). The lower bound is proved through the use of a modulus of continuity \cite{chen_general_2016,chen_robust_2018} which specializes in the context of our problem to 
\begin{equation*}
    \omega(\varepsilon) = \sup\left\{\frac{|\theta_1 - \theta_2|}{\sigma} : \dTV(N(\theta_1, \sigma^2), N(\theta_2, \sigma^2)) \leq \frac{\varepsilon}{1-\varepsilon} \text{ and } \theta_1,\theta_2 \in \R, \sigma > 0\right\}. 
\end{equation*}
\begin{theorem}[Theorem 5.1 \cite{chen_robust_2018}]
    If \(\varepsilon \in [0, 1]\), then 
    \begin{equation*}
        \inf_{\hat{\theta}} \sup_{\substack{\theta \in \R, \\ \sigma > 0, \\ Q}} P_{\varepsilon, \theta, Q, \sigma}\left\{\frac{|\hat{\theta} - \theta|}{\sigma} \geq \frac{1}{\sqrt{n}} \vee \omega(\varepsilon)\right\} \geq c
    \end{equation*}
    for some universal constant \(c > 0\).
\end{theorem}

\noindent A minimax lower bound thus follows by bounding the modulus of continuity from below. 

\begin{proposition}
    If \(\varepsilon \in [0, 1]\), then 
    \begin{equation*}
        \omega(\varepsilon) \geq \sqrt{2 \log\left(\frac{(1-\varepsilon)/2}{1-2\varepsilon}\right)}. 
    \end{equation*}
\end{proposition}
\begin{proof}
    From a well-known relationship between total variation and Kullback-Leibler divergence (e.g. \cite{tsybakov_introduction_2009}), we have that \(\dTV(N(\theta_1, \sigma^2), N(\theta_2, \sigma^2)) \leq 1 - \frac{1}{2} \exp\left(-\KL\left( N(\theta_1, \sigma^2) \,||\, N(\theta_2, \sigma^2) \right)\right) = 1 - \frac{1}{2}\exp\left(-\frac{|\theta_1 - \theta_2|^2}{2\sigma^2}\right)\) where \(\KL(\cdot \,||\, \cdot)\) denotes the Kullback-Leibler divergence. 
    \begin{align*}
        \omega(\varepsilon) &\geq \sup\left\{\frac{|\theta_1 - \theta_2|}{\sigma} : 1 - \frac{1}{2}\exp\left(-\frac{|\theta_1 - \theta_2|^2}{2\sigma^2}\right) \leq \frac{\varepsilon}{1-\varepsilon} \text{ and } \theta_1,\theta_2 \in \R, \sigma > 0\right\} \\
        &= \sup\left\{\frac{|\theta_1 - \theta_2|}{\sigma} : \frac{2(1-2\varepsilon)}{1-\varepsilon} \leq \exp\left(-\frac{|\theta_1 - \theta_2|^2}{2\sigma^2}\right) \text{ and } \theta_1,\theta_2 \in \R, \sigma > 0\right\} \\
        &= \sqrt{2 \log\left(\frac{(1-\varepsilon)/2}{1-2\varepsilon}\right)}.
    \end{align*}
    The proof is complete.
\end{proof}

\end{document}